%% file: lg1_final.tex
\RequirePackage{latexsym, amsmath, amstext, amssymb, amsfonts, amscd, bm, array, amsbsy, mathrsfs, amscd}
\documentclass[11pt, cmp, numbook, final, a4paper]{svjour}
\usepackage[11pt]{extsizes}
\pdfoutput=1
\usepackage{indentfirst}
\usepackage{graphics} 
\usepackage{epsfig}
\usepackage{transparent, colortbl}
\usepackage{booktabs}
\usepackage{multirow}
\usepackage{anysize}
\usepackage{latexsym}
\usepackage{pdfsync}
\usepackage[all]{xy}
\usepackage[usenames,dvipsnames]{xcolor}
\usepackage{enumerate}
\usepackage{float}
\restylefloat{table}
\usepackage{upgreek}
\usepackage{array}
\def\Kh{K\"{a}hler~}
\def\Khr{K\"{a}hlerian~}
\def\rOmega{\mathrm{\Omega}}
\def\rGamma{\mathrm{\Gamma}}
\def\rH{\mathrm{H}}
\def\rR{\mathrm{R}}

\usepackage[colorlinks=true,linkcolor=black]{hyperref}
\colorlet{linkequation}{blue}

\newcommand*{\SavedEqref}{}
\let\SavedEqref\eqref
\renewcommand*{\eqref}[1]{%
  \begingroup
    \hypersetup{
      linkcolor=blue,
      linkbordercolor=blue,
    }%
    \SavedEqref{#1}%
  \endgroup
}

\DeclareSymbolFont{extraup}{U}{zavm}{m}{n}
\DeclareMathSymbol{\varheart}{\mathalpha}{extraup}{86}
\DeclareMathSymbol{\vardiamond}{\mathalpha}{extraup}{87}

\makeatletter
\def\CT@@do@color{%
  \global\let\CT@do@color\relax
        \@tempdima\wd\z@
        \advance\@tempdima\@tempdimb
        \advance\@tempdima\@tempdimc
        \kern-\@tempdimb
\transparent{0.6}%
        \leaders\vrule
                \hskip\@tempdima\@plus  1fill
        \kern-\@tempdimc
        \hskip-\wd\z@ \@plus -1fill }
\makeatother

\makeatletter
\newcommand{\thickhline}{%
    \noalign {\ifnum 0=`}\fi \hrule height 1pt
    \futurelet \reserved@a \@xhline
}
\newcolumntype{"}{@{\hskip\tabcolsep\vrule width 1pt\hskip\tabcolsep}}
\makeatother

\newtheorem{Theorem}{Theorem}[section]
\newtheorem{Lemma}[Theorem]{Lemma}
\newtheorem{Proposition}[Theorem]{Proposition}
\newtheorem{Corollary}[Theorem]{Corollary}
\newtheorem{Conjecture}[Theorem]{Conjecture}
\newtheorem{Definition}[Theorem]{Definition}

\newcommand{\Hom}{{\rm Hom}}

\newcommand{\K}{\Bbb{K}}

\def\ad{\mathrm{ad}}

\newcommand{\im}{\mathrm{im}}

\newcommand{\Coh}{\mathrm{Coh}}

\newcommand{\Jac}{\mathrm{Jac}}

\newcommand{\bp}{\begin{Proposition}}
\newcommand{\ep}{\end{Proposition}}
\newcommand{\bl}{\begin{Lemma}}
\newcommand{\el}{\end{Lemma}}
\newcommand{\bt}{\begin{Theorem}}
\newcommand{\et}{\end{Theorem}}
\newcommand{\bd}{\begin{Definition}}
\newcommand{\ed}{\end{Definition}}
\newcommand{\End}{\mathrm{End}}

\newcommand{\ev}{\mathrm{ev}}
\newcommand{\eqdef}{\stackrel{{\rm def.}}{=}}

\newcommand{\cinf}{{C^\infty(X)}}
\def\rM{\mathrm{M}}

\DeclareFontFamily{U}{rsf}{}
\DeclareFontShape{U}{rsf}{m}{n}{<5> <6> rsfs5 <7> <8> <9> rsfs7 <10-> rsfs10}{}
\DeclareMathAlphabet\Scr{U}{rsf}{m}{n}

\def\cU{\mathcal{U}}

\def\hW{{\widehat W}}

\def\hS{{\hat S}}
\def\hR{{\hat R}}

\def\hD{{\widehat D}}
\def\cE{\check{E}}

\def\Z{\mathbb{Z}}
\def\C{\mathbb{C}}
\def\R{\mathbb{R}}

\def\H{\mathbb{H}}
\def\K{\mathbb{K}}

\def\rk{{\rm rk}}

\def\deg{{\rm deg}}

\def\dd{\mathrm{d}}

\def\supp{\mathrm{supp}}
\def\fd{\mathfrak{d}}

\def\md{\boldsymbol{\fd}}

\def\tW{{\widetilde W}}
\def\tw{{\widetilde w}}

\def\ad{\mathrm{ad}}
\def\Int{\mathrm{Int}}

\def\hOmega{\widehat{\Omega}}

\def\F{\mathrm{F}}
\def\i{\mathbf{i}}
\def\cC{\mathcal{C}}

\def\MF{\mathrm{MF}}
\def\HMF{\mathrm{HMF}}

\def\ext{\mathrm{ext}}
\def\wcR{{\boldsymbol{\mathcal R}}}
\def\ba{{\mathbf{a}}}
\def\fz{\mathfrak{Z}}
\def\Card{\mathrm{Card}}
\def\mpd{\boldsymbol{\pd}}
\def\sm{\mathrm{sm}}

\newcommand{\be}{\begin{equation*}}
\newcommand{\ee}{\end{equation*}}
\newcommand{\ben}{\begin{equation}}
\newcommand{\een}{\end{equation}}
\newcommand{\beqa}{\begin{eqnarray*}}
\newcommand{\eeqa}{\end{eqnarray*}}
\newcommand{\beqan}{\begin{eqnarray}}
\newcommand{\eeqan}{\end{eqnarray}}

\newcommand \pd {{\partial}}
\newcommand \bpd {{\overline{\partial}}}

\newcommand{\nn}{\nonumber}

\newcommand{\Res}{\mathrm{Res}}
\newcommand{\id}{\mathrm{id}}
\newcommand{\Tr}{\mathrm{Tr}}
\newcommand{\tr}{\mathrm{tr}}
\newcommand{\str}{\mathrm{str}}

\def\Hess{\mathrm{Hess}}

\def\cR{{\mathcal R}}

\def\cB{{\mathcal B}}

\def\cK{\mathrm{\cal K}}
\def\cM{\mathrm{\cal M}}
\def\odd{\mathrm{odd}}

\def\O{\mathrm{O}}

\def\cD{\mathcal{D}}

\def\cA{\mathcal{A}}
\def\cE{\mathcal{E}}

\def\cP{\mathcal{P}}

\def\cT{\mathcal{T}}
\def\cF{\mathcal{F}}
\def\cH{\mathcal{H}}

\def\G_2{\mathrm{G_2}}
\def\cO{\mathcal{O}}

\def\cS{\mathcal{S}}

\def\fD{\mathfrak{D}}

\def\Ob{\mathrm{Ob}}
\def\VB{\mathrm{VB}}
\def\Vb{\mathrm{Vb}}

\def\sc{\mathrm{sc}}

\def\HF{\mathrm{HF}}

\def\PV{\mathrm{PV}}
\def\HPV{\mathrm{HPV}}
\def\Cob{\mathrm{Cob}}
\def\Vect{\mathrm{Vect}}
\def\ioda{\boldsymbol{\iota}}
\def\bbpd{\boldsymbol{\bpd}}

\def\0{{\hat{0}}}
\def\1{{\hat{1}}}

\def\cJ{\mathcal{J}}
\def\Jac{\mathrm{Jac}}

\def\O{\mathrm{O}}
\def\DF{\mathrm{DF}}
\def\HDF{\mathrm{HDF}}
\def\Re{\mathrm{Re}}
\def\Sh{\mathrm{Sh}}

\def\J{\mathrm{J}}

\def\mod{\mathrm{mod}}
\def\vect{\mathrm{vect}}

\def\me{\mathbf{e}}
\def\mtr{\boldsymbol{\tr}}

\def\mtr{\boldsymbol{\tr}}
\def\grad{\mathrm{grad}}
\def\Hess{\mathrm{Hess}}

\def\blrcorner{\boldsymbol{\lrcorner}}
\def\Res{\mathrm{Res}}
\def\germ{\mathrm{germ}}

\newcommand{\twopartdef}[4]
{
	\left\{
		\begin{array}{ll}
			#1 & \mbox{if } #2 \\
			#3 & \mbox{if } #4
		\end{array}
	\right.
}

\begin{document}

\title{Differential models for B-type open-closed topological Landau-Ginzburg theories}

\author{Elena Mirela Babalic \and  Dmitry Doryn \and Calin Iuliu Lazaroiu \and Mehdi
  Tavakol}

\institute{Center for Geometry and Physics, Institute for Basic
  Science (IBS), Pohang 37673, Republic of Korea\\
\email{mirela@ibs.re.kr, dmitry@ibs.re.kr, calin@ibs.re.kr, mehdi@ibs.re.kr}}

\date{}

\maketitle

\abstract{We propose a family of differential models for B-type
open-closed topological Landau-Ginzburg theories defined by a pair
$(X,W)$, where $X$ is any non-compact Calabi-Yau manifold and $W$ is
any holomorphic complex-valued function defined on $X$ whose critical
set is compact. The models are constructed at cochain level using
smooth data, including the twisted Dolbeault algebra of
polyvector-valued forms and a twisted Dolbeault category of
holomorphic factorizations of $W$. We give explicit proposals for
cochain level versions of the bulk and boundary traces and for the
bulk-boundary and boundary-bulk maps of the Landau-Ginzburg theory. We
prove that most of the axioms of an open-closed TFT (topological field
theory) are satisfied on cohomology and conjecture that the remaining
two axioms (namely non-degeneracy of bulk and boundary traces and the
topological Cardy constraint) are also satisfied. }

\tableofcontents 

\section{Introduction}
\label{sec:intro}

Classical oriented open-closed topological Landau-Ginzburg theories of
type B are classical field theories parameterized by pairs $(X,W)$ and
defined on compact oriented surfaces with corners, where $X$ is a
non-compact \Khr manifold of dimension $d\eqdef \dim_\C X$ and
$W:X\rightarrow \C$ is a non-constant holomorphic function. The
general construction of such theories was given\footnote{References
\cite{LG1,LG2} considered only the case of oriented surfaces with
boundary. However, the construction of loc. cit. extends easily to
oriented surfaces with corners, leading in the boundary sector to the
differential graded category $\DF_\sm(X,W)$ discussed in Appendix
\ref{app:sc}.} in \cite{LG1,LG2}. As shown in loc. cit., such
classical field theories determine a collection of topological
D-branes, which arise as objects parameterizing the boundary
contribution to the classical action. The Lagrangian construction of
\cite{LG1,LG2} shows that the topological D-branes are described by
{\em special Dolbeault factorizations}, defined as pairs $(S,\fD)$
formed of a smooth complex vector superbundle $S$ on $X$ and a
``special Dolbeault superconnection'' $\fD$ defined on $S$ and which
squares to $W\id_E$. As implicit in \cite{LG2} and explained in
Appendix \ref{app:sc}, such objects naturally form a certain
$\Z_2$-graded differential graded (DG)
category $\DF_\sm(X,W)$, whose morphisms are bundle-valued
differential forms and whose differentials are induced by the
superconnections. The classical action on a compact oriented surface
$\Sigma$ with corners depends on the choice of a special Dolbeault
factorization $(E_i,D_i)$ for each boundary component of $\Sigma$, on
the choice of a morphism of $\DF_\sm(X,W)$ at each corner of
$\Sigma$ as well as on the choice of a K\"{a}hler metric on $X$ and of
an ``admissible'' Hermitian metric $h_i$ on each $E_i$.

It is widely expected that such theories admit a non-anomalous
quantization when $X$ is a Calabi-Yau manifold in the sense that the
canonical line bundle of $X$ is holomorphically trivial. On general
grounds, a physically acceptable quantization procedure for the
Calabi-Yau case must produce, among other structures, a quantum
oriented open-closed topological field theory in the sense axiomatized
in reference \cite{tft} --- a mathematical object which, as shown in
loc. cit., can be described equivalently by a certain algebraic
structure called a {\em TFT datum} of {\em signature} equal to the mod
2 reduction $\mu\in \Z_2$ of the complex dimension of $X$ (see Section
\ref{sec:OCTFT} for the precise definition and terminology). Using a
non-rigorous method of ``partial localization'' of the path integral,
reference \cite{LG2} argued that, when $X$ is Calabi-Yau and the
critical set of $W$ is compact, the TFT datum of quantum B-type
topological Landau-Ginzburg theories can be constructed as the
cohomology of a family of smooth differential models built using a 
dg-algebra $(\PV(X),\updelta_W)$ of polyvector-valued forms defined on
$X$ and the dg-category $\DF_\sm(X,W)$ mentioned above. The ``bulk
differential'' $\updelta_W$ is obtained from the Dolbeault
differential on polyvector-valued forms by adding an operator
proportional to the contraction with $\pd W$.

In this paper, we reconsider the proposal of \cite{LG2} from a
mathematical perspective. Using the Koszul-Malgrange correspondence,
we show that $\DF_\sm(X,W)$ is equivalent with a dg-category
$\DF(X,W)$ of {\em holomorphic factorizations} of $W$. A holomorphic
factorization of $W$ is a pair $(E,D)$, with $E$ a holomorphic vector
superbundle defined on $X$ and $D$ an odd holomorphic section of
$End(E)$ which squares to $W\id_E$. The morphisms of $\DF(X,W)$ are
smooth bundle-valued differential forms while the differentials are
obtained from the Dolbeault operator by adding a correction dependent
on $D$. Using this equivalent description of the dg-category of
topological D-branes, we give a simplified and rigorous construction
of the differential models initially proposed in \cite{LG2} and study
their mathematical properties.

The differential models constructed in this paper provide a cochain
level realization of the constitutive blocks of the TFT datum, from
which the TFT datum can be recovered by passing to cohomology. These
differential models depend explicitly on metric data, namely on the
choice of a K\"{a}hler metric $G$ on $X$ and of an ``admissible''
Hermitian metric $h_a$ on the holomorphic vector superbundle $E$
underlying each holomorphic factorization $a=(E,D)$ of $W$. They also
depend on the choice of a ``localization parameter'' $\lambda\in
\R_{\geq 0}$. More precisely, the bulk and boundary traces
$\Tr$ and $\tr=(\tr_a)_{a\in \Ob\DF(X,W)}$ and the bulk-boundary and
boundary-bulk maps $e=(e_a)_{a\in \Ob \DF(X,W)}$ and $f=(f_a)_{a\in
  \Ob \DF(X,W)}$ of the TFT datum (see Section \ref{sec:OCTFT}) have
cochain level realizations which depend on $\lambda$ and on the
metrics $G$ and $h_a$, such that the maps induced by these on
cohomology recover $\Tr$, $\tr$ and $e$, $f$. We prove that the
cohomological bulk and boundary traces $\Tr$ and $\tr$ constructed in
this manner are independent of the choice of $\lambda$, $G$ and $h_a$
and that the cohomological bulk-boundary maps $e_a$ and cohomological
boundary-bulk maps $f_a$ are independent on the choice of $\lambda$
and $h_a$ and depend only on the K\"{a}hler class of $G$. 
We also show that most of the axioms of a TFT datum are
satisfied on cohomology:

\

\begin{Theorem}
\label{thm}
Suppose that the critical set $Z_W$ of $W$ is compact. Then the
cohomology algebra $\HPV(X,W)$ of $(\PV(X),\updelta_W)$ is
finite-dimensional over $\C$ while the total cohomology category
$\HDF(X,W)$ of $\DF(X,W)$ is Hom-finite over $\C$. Moreover, the
system:
\be
(\HPV(X,W),\HDF(X,W), \Tr, \tr, e)
\ee
obeys all defining properties of a TFT datum of signature equal to the
modulo 2 reduction of $d\eqdef \dim _\C X$, up to non-degeneracy of the
bulk and boundary traces and to the topological Cardy constraint.
\end{Theorem}

\

\noindent In view of the path integral arguments of \cite{LG2} and of
experience with the case when $X$ is Stein and $W$ has finite critical
set (see \cite{nlg2}), we make the following conjecture:

\

\begin{Conjecture}
\label{conj}
Suppose that $Z_W$ is compact. Then $(\HPV(X,W), \HDF(X,W), \Tr, \tr,
e)$ is a TFT datum and hence it defines a quantum open-closed TFT in
the axiomatic sense of \cite{tft}. Moreover, the boundary-bulk maps of
this TFT datum are induced by the cochain level boundary-bulk maps, and the 
same is true for the corresponding bulk-boundary maps. 
\end{Conjecture}

\

\noindent When $X$ is Stein and $W$ has finite-critical set, the
open-closed TFT datum simplifies as explained in \cite{nlg2} and
summarized in Section \ref{sec:Stein}. In that situation, the
cohomological bulk trace reduces to a sum over Grothendieck traces
defined over the analytic Milnor algebras of the holomorphic function germ of
$W$ at its critical points, so its non-degeneracy follows from
non-degeneracy of the later. In this situation, the cohomological
boundary trace of a holomorphic factorization $a$ of $W$ reduces to
the sum of Kapustin-Li traces (see \cite{KL}) over the critical points, so its
non-degeneracy follows from the results of \cite{DM,PV}. As we show in
\cite{tserre}, non-degeneracy of the bulk and boundary traces can be
proved in full generality by adapting the approach originally used by
Serre in \cite{Serre} to prove Serre duality on (not necessarily
compact) complex manifolds. On the other hand, the topological Cardy
constraint can be viewed as a ``deformed'' version of the
Hirzebruch-Riemann-Roch theorem.

\

The paper is organized as follows. Section \ref{sec:OCTFT} recalls the
algebraic description of (non-anomalous and oriented) quantum
open-closed two-dimensional topological field theories (TFTs) which
was derived in \cite{tft} using the definition of the latter as
certain monoidal functors from the category of two-dimensional
oriented cobordisms with corners to the category of finite-dimensional
$\Z_2$-graded vector spaces over $\C$. Section \ref{sec:bulk}
discusses the twisted Dolbeault dg-algebra $(\PV(X),\updelta_W)$ of
polyvector-valued forms defined on $X$, which was proposed in
\cite{LG2} as an ``off-shell'' model for the algebra of bulk
observables of B-type Landau-Ginzburg TFTs. When the critical set of
$W$ is compact, we show that the cohomology $\HPV(X,W)$ of this
algebra is finite-dimensional over $\C$ and that it is isomorphic with
the cohomology $\HPV_c(X,W)$ of the dg-subalgebra $\PV_c(X)$ obtained
by restricting to polyvector-valued forms with compact
support. Section \ref{sec:DF} discusses the dg-category $\DF(X,W)$ of
holomorphic factorizations of $W$, which (as shown in Appendix
\ref{app:sc}) provides an equivalent model for the dg-category of
topological D-branes. When the critical set of $W$ is compact, we show
that the total cohomology category $\HDF(X,W)$ of this dg-category is
Hom-finite over $\C$ and that it is isomorphic with the total
cohomology category $\HDF_c(X,W)$ of a non-full dg-subcategory
$\DF_c(X,W)$ obtained by restricting to morphisms of compact
support. Sections \ref{sec:BulkTrace} and \ref{sec:BoundaryTrace}
discuss cochain level models for the bulk and boundary traces as well
as the maps which they induce on cohomology. The cochain level traces
can always be defined on $\PV_c(X)$ and $\DF_c(X,W)$ and the cochain
level boundary traces satisfy a strict cyclicity condition. When the
critical set of $W$ is compact, the cochain level traces induce
cohomological traces on $\HPV(X,W)$ and $\HDF(X,W)$, which in
particular make the latter into a pre-Calabi-Yau $\Z_2$-graded
category. Section \ref{sec:disk} discusses a dg-algebra which provides
an off-shell model for the algebra of ``disk observables'' of the
B-type Landau-Ginzburg theory. Sections \ref{sec:BoundaryBulk} and
\ref{sec:BulkBoundary} discuss cochain level boundary-bulk and
bulk-boundary maps between $\DF_c(X,W)$ and $\PV_c(X)$, showing that
they satisfy a strict adjointness condition with respect to the
cochain level bulk and boundary traces and that the cohomological maps
which they induce between $\HDF_c(X,W)$ and $\HPV_c(X,W)$ depend on
metric data only through the \Kh class of $G$. When the critical set
of $W$ is compact, these induce cohomological maps between $\HDF(X,W)$
and $\HPV(X,W)$ which satisfy the adjointness condition of a TFT
datum. Section \ref{sec:Datum} gives the proof of Theorem
\ref{thm}. Section \ref{sec:Stein} explains how our proposal for the
TFT datum simplifies when $X$ is a Stein Calabi-Yau manifold and in
particular discusses the example $X=\C^d$.  Finally, Section
\ref{sec:conclusions} presents our conclusions and some directions for
further research. Appendix A gives various expressions in local
coordinates and establishes equivalence of our model for the off-shell
bulk algebra with another description found in the
literature. Appendix B discusses the dg-category $\DF_\sm(X,W)$ of
special Dolbeault factorizations and establishes its equivalence with
the dg-category $\DF(X,W)$ of holomorphic factorizations used in the
present paper.

\subsection{Notations and conventions}
\label{subsec:notation}

\begin{enumerate}[I.]

\item The notation $A\eqdef B$ means that $A$ is defined to equal $B$. 
The notation $A:=B$ means that $A$ is a simplified notation for $B$, 
usually obtained by omitting indices which indicate the dependence of $B$ 
on other quantities, when the latter is clear from the context. 

\item Numbers. The imaginary unit is denoted by $\i$ and the field of
  integers modulo two is denoted by $\Z_2\eqdef \Z/2\Z$. The modulo 2
  reduction of any integer $n\in \Z$ is denoted by $\hat{n}\in \Z_2$.

\item Rings and categories. Let $\Vect_\C$ denote the category of
  vector spaces over $\C$ and $\vect_\C$ denote the full subcategory
  of finite-dimensional $\C$-vector spaces. For any ring $R$ (not
  necessarily commutative) and any element $r\in R$, we denote by
  ${\hat r}:R\rightarrow R$ the operator of left multiplication with
  $r$:
\be
{\hat r}(x)\eqdef rx~~,~~\forall x\in R~~.
\ee

\item Graded categories. Let $\Lambda$ be an Abelian group and $R$ be a
  commutative ring. By a ($\Lambda$-graded) $R$-linear dg-category we
  mean a ($\Lambda$-graded) linear dg-category defined over $R$.
  Given a $\Lambda$-graded $R$-linear category $\cA$, let $\cA^0$
  denote the $R$-linear category obtained from $\cA$ by keeping the
  same objects and keeping only those morphisms of $\cA$ which have
  degree zero. Given a $\Lambda$-graded $R$-linear dg-category 
  $\cC$, let $\rH(\cC)$ denote its {\em total} cohomology
  category, defined as the $\Lambda$-graded $R$-linear category having
  the same objects as $\cC$ and morphism spaces given by the total
  cohomology $R$-modules:
\be
\Hom_{\rH(\cC)}(a,b)\eqdef \rH(\Hom_{\cC}(a,b),d_{a,b})=\bigoplus_{\lambda\in \Lambda}\rH^\lambda(\Hom_{\cC}(a,b),d_{a,b})~~,
\ee
where $d_{a,b}$ is the differential on the $R$-module
$\Hom_{\cC}(a,b)$. In this case, we use the notation
$\rH^0(\cC)=\rH(\cC)^0$.

\item Complex manifolds. All manifolds considered are smooth,
  paracompact, {\em connected} and of non-zero dimension and all
  vector bundles considered are smooth. For a complex
  manifold $X$, we use the following notations:
\begin{enumerate}[1.]
\itemsep 0.0em
\item Given a global section $s$ of any vector bundle $S$ defined on
  $X$, an open subset $U\subset X$ and a section $\sigma$ of $S|_U$,
  the notation $s=_U\sigma$ means $s|_U=\sigma$.
\item Given any holomorphic vector bundle $E$ defined on $X$, let: 
\be
\rOmega(X,E)\eqdef \bigoplus_{i,j=0}^d \rOmega^{i,j}(X,E)
\ee
denote the $\cinf$-module of $E$-valued smooth differential forms
defined on $X$, where $\rOmega^{i,j}(X,E)$ is the $\cinf$-submodule of
smooth $E$-valued differential forms of type $(i,j)$. Let
$\cA^j(X,E)\eqdef \rOmega^{0,j}(X,E)$ and:
\be
\cA(X,E)\eqdef \bigoplus_{j=0}^d\cA^j(X,E)~~.
\ee
Let $\bbpd_E:\rOmega(X,E)\rightarrow \rOmega(X,E)$ denote the
Dolbeault operator of $E$, which satisfies $\bbpd_E^2=0$ as well as
$\bbpd_E(\rOmega^{i,j}(X,E))\subset \rOmega^{i,j+1}(X,E)$. Let $\pd$
and $\bpd$ denote the Dolbeault operators of $X$. We have
$\bpd=\bbpd_{\cO_X}$ and
$\cA(X):=\cA(X,\cO_X)=\oplus_{j=0}^d\cA^j(X)$, where
$\cA^j(X):=\cA^j(X,\cO_X)=\rOmega^{0,j}(X):=\rOmega^{0,j}(X,\cO_X)$.
 
\item Let $TX$ and ${\overline{T}} X$ denote the holomorphic and
  anti-holomorphic tangent bundles of $X$. Let $T^\ast X:=
  (TX)^\vee$, ${\overline{T}}^\ast X:= ({\overline{T}}X)^\vee$
  denote the holomorphic and anti-holomorphic cotangent bundles of
  $X$. Let $\cT_\R X$ denote the tangent bundle of the underlying real
  manifold and $\cT X$ denote the complexification of $\cT_\R X$. We
  have $\cT X=TX \oplus {\overline{T}}X$. The map $\Re:TX\rightarrow
  \cT_\R X$ which takes the real part gives a canonical isomorphism of
  real vector bundles between $TX$ and $\cT_\R X$ (see \cite{Voisin}).
\item Let $\Vb_\sm(X)$ denote the exact category of complex vector
  bundles defined on $X$. Let $\cC^\infty_X$ denote the sheaf of
  smooth complex-valued functions defined on open subsets of $X$
  (notice that this sheaf is not coherent when $X$ has positive
  dimension) and $\Sh_\sm(X)$ denote the Abelian category of
  sheaves of $\cC^\infty_X$-modules. Let $\VB_\sm(X)$ denote the
  {\em full} subcategory of $\Sh_\sm(X)$ whose objects are the
  locally-free sheaves of finite rank. Let
  $\rGamma_\sm:\Vb_\sm(X)\rightarrow \Sh_\sm(X)$ be the
  functor which sends a complex vector bundle $S$ to the corresponding
  locally-free sheaf of $\cC^\infty_X$-modules $\cS:=
  \rGamma_\sm(S)$ consisting of locally-defined smooth sections of
  $S$ and sends a morphism of vector bundles into the corresponding
  morphism of sheaves. Then $\rGamma_\sm$ embeds $\Vb_\sm(X)$ as
  a {\em non-full} subcategory of $\VB_\sm(X)$ having the same
  objects as the latter. This allows us to view the objects of
  $\VB_\sm(X)$ as complex vector bundles defined on $X$. With this
  identification, the morphisms of $\Vb_\sm(X)$ are those morphisms
  of $\VB_\sm(X)$ whose kernel and cokernel are locally-free
  sheaves.  Let $\cinf:=\cC^\infty_X(X)$ denote the commutative
  $\C$-algebra of complex-valued smooth functions defined on $X$ and
  $\rGamma_\sm(X,S):= \rGamma_\sm(S)(X)=\cS(X)$ denote the
  $\cinf$-module of globally-defined smooth sections of a complex
  vector bundle $S$. For any two complex vector bundles
  $S_1$ and $S_2$ defined on $X$, we have:
\be
\Hom_{\VB_\sm(X)}(S_1,S_2)=\rGamma_\sm(X,Hom(S_1,S_2))=Hom(\cS_1,\cS_2)(X)~~,
\ee
  where $Hom(S_1,S_2)\!\eqdef\! S_1^\vee\otimes S_2$ denotes the vector
  bundle of morphisms from $S_1$ to $S_2$ and
  $Hom_{\cC^\infty_X}(\cS_1,\cS_2)\!\eqdef \!\cS_1^\vee\otimes_{\cC^\infty_X} \cS_2$ is the
  inner Hom sheaf of the corresponding locally-free sheaves of
  $\cC^\infty_X$-modules $\cS_i:=\rGamma_\sm(S_i)$.
\item Let $\Vb(X)$ denote the exact category of holomorphic vector
  bundles defined on $X$. Let $\cO_X$ denote the sheaf of holomorphic
  complex-valued functions defined on open subsets of $X$ (which is
  coherent by Oka's coherence theorem) and $\Coh(X)$ denote the
  Abelian category of coherent sheaves of $\cO_X$-modules. Let
  $\VB(X)$ denote the {\em full} subcategory of $\Coh(X)$ whose
  objects are the locally-free sheaves of finite rank. Let
  $\rGamma:\Vb(X)\rightarrow \Coh(X)$ be the functor which sends a
  holomorphic vector bundle $E$ to the corresponding locally-free
  sheaf of $\cO_X$-modules $\cE:=\rGamma(E)$ consisting of locally-defined
  holomorphic sections of $E$ and sends a morphism of holomorphic
  vector bundles to the corresponding morphism of sheaves. Then
  $\rGamma$ identifies $\Vb(X)$ with a {\em non-full} subcategory of
  $\VB(X)$. This allows us to view the objects of $\VB(X)$ as
  holomorphic vector bundles defined on $X$. With this identification,
  the morphisms of $\Vb(X)$ are those morphisms of $\VB(X)$ whose
  kernel and cokernel are locally-free sheaves. Let $\O(X):=
  \cO_X(X)$ denote the commutative $\C$-algebra of complex-valued
  holomorphic functions defined on $X$. For any holomorphic vector
  bundle $E$ defined on $X$, let $\rGamma(X,E):=
  \rGamma(E)(X)=\cE(X)$ denote the $\O(X)$-module of globally-defined
  holomorphic sections of $E$ (which coincides with the zeroth sheaf
  cohomology $\rH^0(\cE)$ of the corresponding locally-free sheaf
  $\cE:=\rGamma(E)$). For any two holomorphic vector bundles $E_1$
  and $E_2$ defined on $X$, we have:
\be
\Hom_{\VB(X)}(E_1,E_2)=\rGamma(X,Hom(E_1,E_2))=Hom(\cE_1,\cE_2)(X)~~,
\ee
  where $Hom(E_1,E_2)=E_1^\vee\otimes E_2$ denotes the holomorphic
  vector bundle of morphisms from $E_1$ to $E_2$ and
  $Hom(\cE_1,\cE_2)=\cE_1^\vee\otimes_{\cO_X} \cE_2$ denotes the sheaf
  inner Hom of the corresponding locally-free sheaves of
  $\cO_X$-modules $\cE_i:=\rGamma(E_i)$.
\item As a general point of terminology inspired by the physics literature, 
``off-shell'' refers to an object defined at cochain level while ``on-shell'' 
refers to an object defined at cohomology level. 
\end{enumerate}

\begin{remark}
Unlike the categories $\Vb_\sm(X)$ and $\Vb(X)$ (which are only
$\C$-linear), the categories $\VB_\sm(X)$ and $\VB(X)$ are
respectively $\cinf$-linear and $\O(X)$-linear. 
\end{remark}
\end{enumerate}

\section{Algebraic description of quantum two-dimensional oriented open-closed TFTs}
\label{sec:OCTFT}

A (non-anomalous) quantum two-dimensional oriented open-closed
topological field theory (TFT) can be defined axiomatically
(see \cite{tft,MS,LP}) as a symmetric monoidal functor from a certain
symmetric monoidal category $\Cob_2^\ext$ of labeled 2-dimensional
oriented cobordisms with corners to the symmetric monoidal category
$\vect^s_\C$ of finite-dimensional super-vector spaces defined over
$\C$. The objects of the category $\Cob_2^\ext$ are finite disjoint
unions of circles and line segments while the morphisms are oriented
cobordisms with corners between such, carrying appropriate labels on
boundary components.  By definition, the {\em closed sector} of such a
theory is obtained by restricting the monoidal functor to the
subcategory of $\Cob_2^\ext$ whose objects are disjoint unions of
circles and whose morphisms are ordinary cobordisms (without
corners). It was shown in \cite{tft} that such a functor can be
described equivalently by an algebraic structure which we shall call a
{\em TFT datum}. We start by describing a simpler algebraic structure,
which forms part of any such datum.

\subsection{Pre-TFT data}

\

\

\begin{Definition}
\label{def:preTFT}
A {\em pre-TFT datum} is an ordered triple $(\cH,\cT,e)$ consisting of:
\begin{enumerate}[A.]
\itemsep 0.0em
\item A finite-dimensional unital and supercommutative superalgebra
  $\cH$ defined over $\C$ (called the {\em bulk algebra}), whose unit
  we denote by $1_\cH$
\item A Hom-finite $\Z_2$-graded $\C$-linear category $\cT$ (called
  the {\em category of topological D-branes}), whose composition of
  morphisms we denote by $\circ$ and whose units we denote by $1_a\in
  \Hom_\cT(a,a)$ for all $a\in\Ob\cT$
\item A family $e=(e_a)_{a\in \Ob\cT}$ consisting of even
  $\C$-linear maps $e_a:\cH\rightarrow \Hom_\cT(a,a)$ defined for each
  object $a$ of $\cT$ (the map $e_a$ is called the {\em bulk-boundary
    map} of $a$)
\end{enumerate}
such that the following conditions are satisfied:
\begin{enumerate}[1.]
\itemsep 0.0em
\item For any object $a\in \Ob\cT$, the map $e_a$ is a unital morphism
  of $\C$-superalgebras from $\cH$ to the endomorphism algebra
  $(\End_\cT(a),\circ)$, where $\End_\cT(a)\eqdef \Hom_\cT(a,a)$.
\item For any two objects $a,b\in \Ob\cT$ and for any $\Z_2$-homogeneous
  elements $h\in\cH$ and $t\in \Hom_\cT(a,b)$, we have:
\be
e_b(h)\circ t=(-1)^{\deg h\,\deg t} t \circ e_a(h)~~.
\ee
\end{enumerate}
\end{Definition}

\noindent Given a pre-TFT datum $(\cH,\cT,e)$, the objects of the
category $\cT$ are called {\em topological D-branes}.  Elements $h\in
\cH$ are called {\em on-shell bulk states}. For any topological
D-branes $a, b\in \Ob \cT$, elements $t\in \Hom_{\cT}(a,b)$ are called
{\em on-shell boundary states} from $a$ to $b$.

\

\begin{Definition}
Let $\cT$ be a $\Z_2$-graded $\C$-linear category.  The {\em total
  endomorphism algebra} of $\cT$ is the associative $\C$-superalgebra
whose underlying $\Z_2$-graded vector space is defined through:
\be
\End(\cT)\eqdef \bigoplus_{a,b\in \Ob\cT}\Hom_\cT({a,b})
\ee
and whose multiplication is given by the obvious $\C$-bilinear
extension of the composition of morphisms of $\cT$. 
\end{Definition}

\

\noindent The total endomorphism algebra $\End(\cT)$ is
unital with unit $1_\cT\eqdef \oplus_{a\in \Ob\cT}{1_a}$. Notice that 
$\End(\cT)$ is infinite-dimensional if $\cT$ has an infinity of objects, 
even when $\cT$ is Hom-finite. 

\

\begin{Definition}
Let $(\cH,\cT,e)$ be a triplet satisfying conditions A., B. and C. of
Definition \ref{def:preTFT}. The {\em total bulk-boundary map} of
$(\cH,\cT,e)$ is the even $\C$-linear map
$\me:\cH\rightarrow \End(\cT)$ defined through:
\be
\me(h)\eqdef \bigoplus_{a\in \Ob\cT}e_a(h)~~,~~\forall h\in \cH~~.
\ee
\end{Definition} 

\noindent The following statement is obvious: 

\

\begin{Proposition}
Let $(\cH,\cT,e)$ be a triplet satisfying conditions A., B. and C.
above. Then $(\cH,\cT,e)$ is a pre-TFT datum iff the total
bulk-boundary map $\me$ of $(\cH,\cT, e)$ is a unital morphism of
$\Z_2$-graded algebras whose image lies in the supercenter of the
total endomorphism algebra $(\End(\cT),\circ)$.
\end{Proposition}

\subsection{Calabi-Yau and pre-Calabi-Yau supercategories}

\

\

\begin{Definition}
A {\em Calabi-Yau supercategory of parity $\mu\in \Z_2$} is a pair
$(\cT,\tr)$, where:
\begin{enumerate}[A.]
\itemsep 0.0em
\item $\cT$ is a $\Z_2$-graded and $\C$-linear Hom-finite category
\item $\tr=(\tr_a)_{a\in \Ob \cT}$ is a family of $\C$-linear maps
  $\tr_a:\Hom_\cT(a,a)\rightarrow \C$ of $\Z_2$-degree $\mu$
\end{enumerate} 
such that the following conditions are satisfied:
\begin{enumerate}[1.]
\itemsep 0.0em
\item For any two objects $a,b\in \Ob\cT$, the $\C$-bilinear pairing
  $\langle \cdot , \cdot \rangle_{a,b}:\Hom_\cT(a,b)\times
  \Hom_\cT(b,a)\rightarrow \C$ defined through:
\be
\langle t_1,t_2\rangle_{a,b}\eqdef \tr_b (t_1\circ t_2)~,~~\forall t_1\in \Hom_\cT(a,b)~,~\forall t_2\in \Hom_\cT(b,a)
\ee
is non-degenerate.
\item For any two objects $a,b\in \Ob\cT$ and any $\Z_2$-homogeneous
  elements $t_1\in \Hom_\cT(a,b)$ and $t_2\in \Hom_\cT(b,a)$, we
  have:
\ben
\label{tcyc}
\langle t_1,t_2\rangle_{a,b}=(-1)^{\deg t_1\,\deg t_2}\langle t_2,t_1\rangle_{b,a}~~.
\een
\end{enumerate}
If only condition $2.$ above is satisfied, we say that $(\cT,\tr)$ is
a {\em pre-Calabi-Yau supercategory} of parity $\mu$.
\end{Definition}

\

\noindent The condition that $\tr_a$ has degree $\mu$ amounts to the
requirement:
\be
\tr_a(h)=0~~\mathrm{unless}~~h\in \Hom_\cT^\mu(a,a)~~,
\ee
where $\Hom_\cT^\mu(a,a)$ denotes the homogeneous component of
$\Hom_\cT(a,a)$ lying in $\Z_2$-degree $\mu$. Considering the total
endomorphism algebra $(\End(\cT),\circ)$ as above, $\tr$ induces a
$\C$-linear map of degree $\mu$:
\be
\mtr:\End(\cT)\rightarrow \C~~,
\ee
defined as follows for $t\in \Hom_{\cT}(a,b)$:
\be
\mtr(t)\eqdef\twopartdef{\tr(t)~,}{a=b}{0~,}{a\neq b}~~.
\ee
This map satisfies: 
\be
\mtr(t_1\circ t_2)=(-1)^{\deg t_1\,\deg t_2}\mtr(t_2\circ t_1)~~,
\ee
for any $\Z_2$-homogeneous elements $t_1,t_2\in \End(\cT)$. 

\subsection{TFT data}

As shown in \cite{tft}, quantum two-dimensional topological field
theories of the type considered above can be identified with the
following algebraic structure:

\

\begin{Definition}
A {\em TFT datum} of parity $\mu\in \Z_2$ is a system
$(\cH,\cT,e,\Tr,\tr)$, where:
\begin{enumerate}[A.]
\itemsep 0.0em
\item $(\cH,\cT,e)$ is a pre-TFT datum
\item $\Tr:\cH\rightarrow \C$ is an even $\C$-linear map (called
  the {\em bulk trace})
\item $\tr=(\tr_a)_{a\in \Ob\cT}$ is a family of $\C$-linear maps
  $\tr_a:\Hom_\cT(a,a)\rightarrow \C$ of $\Z_2$-degree $\mu$ (called the {\em
  boundary traces})
\end{enumerate}
such that the following conditions are satisfied:
\begin{enumerate}[1.]
\itemsep 0.0em
\item $(\cH,\Tr)$ is a supercommutative Frobenius superalgebra. This
  means that the pairing induced by $\Tr$ on $\cH$ is
  non-degenerate, in the sense that the condition $\Tr(hh')=0$ for all
  $h'\in \cH$ implies $h=0$.
\item $(\cT,\tr)$ is a Calabi-Yau supercategory of parity $\mu$.
\item The following relation, called the {\em topological Cardy
  constraint}, is satisfied for all $a, b\in \Ob\cT$:
\ben
\label{Cardy}
\Tr\big(f_a(t_1)f_b(t_2)\big)=\str (\Phi_{ab}\big(t_1,t_2)\big)~~,~~\forall t_1\in \Hom_{\cT}(a,a)~,~\forall t_2\in \Hom_{\cT}(b,b)~~,
\een
where $\str$ is the supertrace on the finite-dimensional $\Z_2$-graded 
vector space $\End_\C(\Hom_\cT(a,b))$ and:
\begin{enumerate}[(a)]
\item The $\C$-linear map $f_a:\Hom_\cT(a,a)\rightarrow \cH$ of $\Z_2$-degree
  $\mu$ (which is called the {\em boundary-bulk map of $a$}) is
  defined as the adjoint of the bulk-boundary map $e_a:\cH\rightarrow
  \Hom_{\cT}(a,a)$ with respect to the non-degenerate traces $\Tr$
  and $\tr_a$, being determined by the relation:
\be
\Tr\big(h f_a(t)\big)=\tr_a\big(e_a(h)\circ t\big)~~,~~\forall h\in \cH~,~\forall t \in \Hom_\cT(a,a)~~.
\ee 
\item For any $a,b\in \Ob\cT$ and any $t_1\in \Hom_\cT(a,a)$ and
  $t_2\in \Hom_\cT(b,b)$, the $\C$-linear map
  $\Phi_{ab}(t_1,t_2):\Hom_\cT(a,b)\rightarrow \Hom_\cT(a,b)$ is
  defined through:
\be
\Phi_{ab}(t_1,t_2)(t)\eqdef t_2\circ t\circ t_1~~,~~\forall t\in \Hom_\cT(a,b)~~.
\ee
\end{enumerate}
\end{enumerate}
\end{Definition}

\begin{remark}
The topological Cardy constraint in the form given above was first
derived in \cite[Subsection 4.7]{tft}. 
\end{remark}

\subsection{B-type topological Landau-Ginzburg theories}

A quantum open-closed B-type topological Landau-Ginzburg theory is a particular
quantum two-dimensional field theory defined on compact oriented
surfaces with corners, which is conjecturally associated to pairs
$(X,W)$, where $X$ is a non-compact Calabi-Yau manifold and
$W:X\rightarrow \C$ is a non-constant complex-valued holomorphic
function defined on $X$. At the classical (i.e. non-quantum) level,
such theories admit a Lagrangian description whose general
construction for compact oriented surfaces with boundary was given in
\cite{LG1,LG2}\footnote{The case of oriented surfaces {\em without}
  boundary was studied much earlier in \cite{Vafa,Labastida}.}. That
construction extends immediately to compact oriented surfaces with
corners, as outlined in Appendix \ref{app:sc}. As explained in
\cite{LG2}, a (non-rigorous) procedure of ``partial path integral
localization'' allows one to extract an explicit family of smooth
differential models for the off-shell state spaces of this theory,
whose cohomology conjecturally recovers the TFT datum of the
open-closed TFT defined by the on-shell un-integrated correlators. In
the next sections, we give an equivalent mathematical description of
that model, referring the reader to Appendix \ref{app:sc} for the
precise relation to the superconnection language used in
\cite{LG2}. Let us first define the data parameterizing such
theories. By a K\"{a}hlerian manifold we mean a complex manifold which
admits at least one K\"{a}hler metric.

\

\begin{Definition}
A {\em Landau-Ginzburg (LG) pair} of dimension $d$ is a pair $(X,W)$, where:
\begin{enumerate}
\itemsep 0.0em
\item $X$ is a non-compact K\"{a}hlerian manifold of complex dimension $d$
  which is {\em Calabi-Yau} in the sense that the canonical line
  bundle $K_X\eqdef \wedge^d T^\ast X$ is holomorphically trivial.
\item $W:X\rightarrow \C$ is a {\em non-constant} complex-valued holomorphic
  function defined on $X$.
\end{enumerate}
The {\em signature} $\mu(X,W)$ of a Landau-Ginzburg pair $(X,W)$ is
defined as the modulo 2 reduction of the complex dimension of $X$:
\be
\mu(X,W)\eqdef {\hat d}\in \Z_2~~.
\ee
\end{Definition}

\

\begin{Definition} 
The {\em critical set of $W$} is the set: 
\be
Z_W\eqdef \big\{p\in X \,\big|\,(\pd W)(p)=0\big\}
\ee 
of critical points of $W$.  
\end{Definition}

\section{The off-shell bulk algebra}
\label{sec:bulk}

Let $(X,W)$ be a Landau-Ginzburg pair with $\dim_\C X=d$. In this
subsection, we discuss a dg-algebra $(\PV(X,W),\updelta_W)$ of
polyvector-valued forms. When the critical set of $W$ is compact, the
path integral argument of \cite{LG2} shows that the total cohomology
algebra of this dg-algebra can be interpreted as the bulk
algebra of the B-type open-closed topological Landau-Ginzburg theory
defined by $(X,W)$.

\subsection{The bicomplex of smooth differential forms}

Let $\cT X=T X \oplus {\overline{T}} X$ denote the complexified tangent
bundle of the real manifold underlying $X$ and $\cT^\ast X\eqdef (\cT
X)^\vee$ denote the complexified cotangent bundle. Let $\wedge
\cT^\ast X\eqdef \oplus_{k=0}^{2d} \wedge^k \cT^\ast X$ denote the
complexified exterior bundle of $X$. Let $\wedge T^\ast X\eqdef
\oplus_{k=0}^d \wedge^k T^\ast X$ and $\wedge \overline{T}^\ast X\eqdef
\oplus_{k=0}^d \wedge^k {\overline{T}}^\ast X$ denote the holomorphic and
antiholomorphic exterior bundles. Let:
\be
\rOmega(X)\eqdef \rGamma_\sm(X,\wedge \cT^\ast X)
\ee
denote the Grassmann algebra of $X$, i.e. the $\cinf$-algebra of smooth
inhomogeneous differential forms defined on $X$, with multiplication
given by the wedge product. Let: 
\be
\rOmega^{i,j}(X)\eqdef \rGamma_\sm(X,\wedge^i T^\ast X \otimes \wedge ^j {\overline{T}}^\ast X)
\ee
denote the $\cinf$-submodule of $\rOmega(X)$ consisting of smooth
differential forms of type $(i,j)$. Set $\rOmega^{i,j}(X)=0$ for
$i\not\in \{0,\ldots, d\}$ or $j\not \in \{0,\ldots, d\}$.  We view
$\rOmega(X)$ as a unital bigraded $\cinf$-algebra using the
decomposition:
\be
\rOmega(X)=\bigoplus_{i,j=0}^d \rOmega^{i,j}(X)~~.
\ee
The {\em rank grading} of $\rOmega(X)$ is the total $\Z$-grading
induced by this bigrading, which corresponds to the decomposition: 
\be
\rOmega(X)=\bigoplus_{k=0}^{2d} \rOmega^k(X)~~\mathrm{where}~~\rOmega^k(X)\eqdef \bigoplus_{i+j=k}\rOmega^{i,j}(X)~~.
\ee
Let $\dd$ be the de Rham differential and $\partial$ and ${\bar
  \partial}$ be the Dolbeault differentials of $X$. We have:
\be
\dd=\pd+\bpd~~
\ee
and: 
\be
\dd^2=\pd^2=\bpd^2=\pd\bpd+\bpd\pd=0~~.
\ee
In particular, $(\rOmega(X),\pd,\bpd)$ is a bicomplex of $\C$-vector
spaces whose total differential equals $\dd$. When $\rOmega(X)$ is
endowed with the rank grading, the triplet $(\rOmega(X),\wedge,\dd)$ is
a unital $\Z$-graded supercommutative dg-algebra over $\C$. Notice
that the differential $\bpd$ is $\O(X)$-linear.

\subsection{The dg-algebra of $(0,\ast)$-forms}
Let: 
\be
\cA^j(X)\eqdef \rOmega^{0,j}(X)=\rGamma_\sm(X,\wedge^j \overline{T}^\ast X)~~
\ee
be the $\cinf$-module of smooth differential forms of type $(0,j)$
defined on $X$. Consider the $\Z$-graded $\cinf$-module:
\be
\cA(X)\eqdef \rGamma_\sm(X,\wedge \overline{T}^\ast X)=\bigoplus_{j=0}^d\cA^j(X)~~
\ee
consisting of all differential forms which have degree zero with
respect to the first grading of $\rOmega(X)$. Then $(\cA(X),\wedge,
\bpd)$ is a $\Z$-graded $\O(X)$-linear dg-algebra.

\subsection{The algebra of polyvector fields}

Let $\wedge TX\eqdef \oplus_{i=0}^d \wedge^i TX$ be the bundle of
holomorphic polyvectors. Consider the $\cinf$-module: 
\be
\cB^i(X)\eqdef \rGamma_\sm(X,\wedge^i TX)~~
\ee
and the $\Z$-graded $\cinf$-module: 
\be
\cB(X)\eqdef \rGamma_\sm(X,\wedge TX)=\bigoplus_{i=0}^d\cB^i(X)~~.
\ee
Then $\cB(X)$ becomes a unital associative and supercommutative
$\Z$-graded $\cinf$-algebra when endowed with the multiplication
induced by the wedge product of polyvector fields.

\subsection{The twisted Dolbeault algebra of polyvector-valued forms}
\label{bulk}

Consider the $\cinf$-module\,\footnote{In some references, elements of
  $\PV(X)$ are called ``polyvector fields'', though such language does
  not match the traditional terminology used in differential
  geometry. Strictly speaking, a smooth polyvector field of type
  $(\ast,0)$ is a smooth section of the bundle $\wedge TX$ and not a
  differential form valued in that bundle.}:
\be
\PV(X)\eqdef \cA(X,\wedge TX) \simeq \cA(X)\otimes_\cinf\cB(X)~~,
\ee
endowed with the unital and associative $\cinf$-linear multiplication
determined uniquely by the condition:
\be
(\rho_1\otimes P_1)(\rho_2\otimes P_2)=(-1)^{ij}(\rho_1\wedge \rho_2)\otimes (P_1\wedge P_2)~~
\ee
for all $\rho_1\in \cA(X), \rho_2\in \cA^j(X), P_1\in \cB^i(X),P_2\in \cB(X)$.  
For any $i=-d,\ldots, 0$ and $j=0,\ldots, d$, let: 
\be
\PV^{i,j}(X)\eqdef \cA^j(X,\wedge^{|i|}TX)\simeq \cA^j(X)\otimes_\cinf \cB^{|i|}(X)
\ee
denote the space of smooth $(0,j)$-forms valued in the holomorphic
vector bundle $\wedge^{|i|} TX$. Set $\PV^{i,j}(X)=0$ for $i\not\in
\{-d,\ldots, 0\}$ or $j\not\in \{0,\ldots, d\}$. With these
conventions, the decomposition:
\be
\PV(X)= \bigoplus_{i=-d}^0 \bigoplus_{j=0}^d \PV^{i,j}(X)~~
\ee
makes $\PV(X)$ into a unital associative $\Z\times \Z$-graded
$\cinf$-algebra with multiplication given by the wedge product, 
whose grading is concentrated in bidegrees $(i,j)$
satisfying $i\in \{-d,\ldots, 0\}$ and $j\in \{0,\ldots, d\}$. Notice
that $\PV(X)$ is a unital supercommutative $\cinf$-algebra when
endowed with the total grading, which we denote by $\deg$.

\

\begin{Definition}
The {\em canonical grading of $\PV(X)$} is the total $\Z$-grading of
the bigrading introduced above, i.e. the $\Z$-grading whose
homogeneous components are defined through:
\be 
\PV^k(X)\eqdef \bigoplus_{i+j=k} \PV^{i,j}(X)~~(k\in \Z)~~.
\ee 
The {\em reduced grading of $\PV(X)$} is the $\Z_2$-grading
defined as the mod 2 reduction of the canonical grading:
\beqa
& & \PV^{\hat 0}(X)\eqdef \bigoplus_{k=\ev} \PV^k(X)\nn\\
& & \PV^{\hat 1}(X)\eqdef \bigoplus_{k=\odd} \PV^k(X)\nn~~.
\eeqa
\end{Definition}

\

\noindent Notice that $\PV^k(X)=0$ unless $k\in \{-d,\ldots, d\}$. We
have:
\be
\PV^{-d}(X)=\PV^{-d,0}(X)=\cB^d(X)~~,~~\PV^d(X)=\PV^{0,d}(X)=\cA^d(X)~~. 
\ee
Let $\bbpd:=\bbpd_{\wedge TX}:\PV(X)\rightarrow \PV(X)$ denote the
Dolbeault operator of the holomorphic vector bundle $\wedge
TX$. This is an $\O(X)$-linear odd derivation of $\PV(X)$ which
preserves the first $\Z$-grading and has degree $+1$ with respect to
the second $\Z$-grading:
\be
\bbpd\big(\PV^{i,j}(X)\big)\subset \PV^{i,j+1}(X)~~.
\ee 
Let $\ioda_W:\PV(X)\rightarrow \PV(X)$ be the unique $\cinf$-linear
map which satisfies the following conditions:
\begin{itemize}
\itemsep 0.0em
\item $\ioda_W$ has bidegree $(1,0)$: 
\ben
\ioda_W\big(\PV^{i,j}(X)\big)\subset \PV^{i+1,j}(X)~~.
\een
\item $\ioda_W$ is an odd derivation of $\PV(X)$ with respect to the canonical grading:
\ben
\label{frml2.01}
\ioda_W(\omega\eta)=(\ioda_W\omega)\eta+(-1)^k\omega (\ioda_W\eta)~~,~~\forall \omega\in \PV^k(X)~,~\forall \eta\in \PV(X)~~.
\een
\item $\ioda_W$ coincides with the contraction with the holomorphic 1-form
$-\i\partial W\in \rGamma(X,T^\ast X)\subset \rOmega^{1,0}(X)$ when
restricted to the submodule
$\PV^{-1,0}(X)=\cB^1(X)=\rGamma_\sm(X,TX)$. 
\end{itemize} 
Both $\bbpd$ and $\ioda_W$ have degree $+1$ with respect to the
canonical grading and they are odd $\O(X)$-linear derivations of the
algebra $\PV(X)$, when the latter is endowed with the reduced
grading. The derivation $\ioda_W$ is also $\cinf$-linear.

\

\begin{Definition}
Let $(X,W)$ be a Landau-Ginzburg pair. The {\em twisted differential}
determined by $W$ on $\PV(X)$ is the operator
$\updelta_W:\PV(X)\rightarrow \PV(X)$ defined through:
\be
\updelta_W\eqdef \bbpd+\ioda_W~~.  
\ee
\end{Definition}

\noindent The twisted differential is $\O(X)$-linear. It has degree
$+1$ with respect to the canonical grading and it is odd with respect
to the reduced grading. We have:
\be 
\updelta_W^2=\bbpd^2=(\ioda_W)^2=\ioda_W\circ
\bbpd+\bbpd\circ \ioda_{W}=0~~,
\ee 
which shows that $(\PV(X),\bbpd,\ioda_W)$ is a bicomplex of $\O(X)$-modules.  

\

\begin{Definition}
The {\em twisted Dolbeault algebra of polyvector-valued forms} of the
Landau-Ginzburg pair $(X,W)$ is the supercommutative $\Z$-graded
$\O(X)$-linear dg-algebra $(\PV(X),\updelta_W)$, where $\PV(X)$ is
endowed with the canonical $\Z$-grading. The {\em cohomological
  twisted Dolbeault algebra} of $(X,W)$ is the supercommutative
$\Z$-graded $\O(X)$-linear algebra defined through:
\be
\HPV(X,W)\eqdef \rH(\PV(X),\updelta_W)~~.
\ee
\end{Definition}

\begin{Proposition}
\label{prop:bulkfd}
Assume that the critical set $Z_W$ is compact. Then $\HPV(X,W)$ is
finite-dimensional over $\C$.
\end{Proposition}

\begin{proof}
Recall (see \cite[Sec 8.1]{Voisin}) that the hypercohomology of a
complex of analytic sheaves which is bounded from below can be
computed using the double complex of sheaves obtained by taking
acyclic resolutions of each sheaf of the original complex\footnote{The
total complex of such a double complex of acyclic sheaves is a complex
of acyclic sheaves which is quasi-isomorphic with the original
complex, so one can apply \cite[Proposition 8.12]{Voisin}.}. Since
Dolbeault resolutions of holomorphic vector bundles are acyclic
(because any sheaf of $\cC^\infty_X$-modules is a fine sheaf), we can
use such resolutions to present the $\Z$-graded $\O(X)$-module
$\HPV(X,W)$ as the total hypercohomology
$\H(\cK_W)=\rH(\rR\rGamma(X,\cK_W))$ of the following Koszul complex 
of vector bundles:
\be
(\cK_W):~~0 \rightarrow \wedge^d TX \stackrel{\ioda_W}{\rightarrow} 
\wedge^{d-1} TX \stackrel{\ioda_W}{\rightarrow} \ldots \stackrel{\ioda_W}{\rightarrow} 
\cO_X \rightarrow 0~~.
\ee
Here $\rR\rGamma(X,\cK_W)$ is viewed as the quasi-isomorphism class of
a complex of $\O(X)$-modules (the image of the complex $\cK_W$ through
the right derived functor of the global sections functor
$\rGamma(X,-)$) and $\rH(\rR\rGamma(X,\cK_W))$ denotes its total
cohomology.  We have $\H(\cK_W)=\cH(X)$, where $\cH$ is the
$\Z$-graded hypercohomology presheaf of $\cK_W$, which associates to
each open subset $U\subset X$ the $\Z$-graded $\O(X)$-module
$\cH(U)=\rH(\rR\rGamma(U,\cK_W))$. The sheafification of $\cH$ is
coherent by Grauert's direct image theorem (see \cite[Chapter 10]{GRs}), since
it coincides with the derived direct image $\rR\id_\ast(\cK_W)$ of the
finite complex of coherent analytic sheaves $\cK_W$ through the proper
holomorphic map $\id_X:X\rightarrow X$. For any open subset $U$ of the
complement $X\setminus Z_W$, the restricted sheaf Koszul complex
$\cK_W|_U$ is exact, hence $\cH(U)$ vanishes for any such $U$. Thus
$\supp\cH$ is contained in the critical set $Z_W$. Since $Z_W$ is
compact and the support is closed by definition, it follows that the
coherent sheaf $\cH$ has compact support, hence its space of sections
$\H(\cK_W)=\cH(X)$ is finite-dimensional.
~\qed
\end{proof}

\

\noindent When $Z_W$ is compact, the path integral argument of
\cite{LG2} shows that the bulk algebra $\cH$ of the open-closed topological
field theory defined by the un-integrated correlators of on-shell
observables of the B-type Landau-Ginzburg theory of the pair $(X,W)$
can be identified with the $\Z_2$-graded $\C$-superalgebra obtained
from $\HPV(X,W)$ by restricting scalars from $\O(X)$ to $\C$ and
reducing the canonical grading modulo 2.

\subsection{The reduced contraction determined by a holomorphic volume form}

Recall that the canonical line bundle $K_X\eqdef \wedge^d T^\ast X$ is
holomorphically trivial, so it admits nowhere-vanishing holomorphic
sections (called holomorphic volume forms). Let $\Omega\in
\rGamma(X,K_X)\subset \rOmega^{d,0}(X)$ be a holomorphic volume form.

\

\begin{Definition}
\label{def:RedCont}
The {\em reduced contraction with $\Omega$} is the $\cinf$-linear map
$\Omega\lrcorner_0:\PV(X)\rightarrow \cA(X)$ defined through:
\be
\Omega\lrcorner_0 \omega=\twopartdef{0~,~}{\omega\not \in \PV^{-d,\ast}(X)}{\Omega\lrcorner \omega~,~}{\omega\in \PV^{-d,\ast}(X)}~~.
\ee
\end{Definition}
We have: 
\be
\Omega\lrcorner_0(\PV^{i,j}(X))=\twopartdef{0~,~}{i\neq -d}{\cA^j(X)~,~}{i=-d}~~.
\ee
Since $\Omega$ is nowhere-vanishing, the restricted map
$\Omega\lrcorner_0:\PV^{-d,\ast}(X)=\cA^\ast(X,\wedge^d TX)
\rightarrow \cA^\ast(X)$ is an isomorphism of $\cinf$-modules. Since
$\Omega$ is holomorphic while $\ioda_W$ decreases polyvector rank, the
reduced contraction satisfies:
\ben
\Omega\lrcorner_0 (\updelta_W\omega)=\Omega\lrcorner_0 (\bbpd\omega)=(-1)^d \bpd(\Omega\lrcorner_0\omega)~~,~~\forall \omega\in \PV(X)~~.
\een
Thus $\Omega\lrcorner_0$ is a map of complexes of $\O(X)$-modules from
$(\PV(X),\updelta_W)$ to $(\cA(X),\bpd)$ which reduces polyvector rank
by $d$ without changing the form rank. This map has $\Z$-degree $+d$
when $\PV(X)$ is endowed with the canonical $\Z$-grading and
$\Z_2$-degree $\mu={\hat d}$ when $\PV(X)$ is endowed with the reduced
grading.

\begin{remark}
Picking a holomorphic
volume form $\Omega$, the operator $\iota_\Omega=\Omega\lrcorner $ of contraction with
$\Omega$ gives an isomorphism of $\cinf$-modules:
\be
\iota_\Omega:\PV^{i,j}(X) \stackrel{\sim}{\rightarrow} \rOmega^{d+i,j}(X)~~.
\ee
This can be used to transport the holomorphic Dolbeault differential
$\pd$ of $\rOmega(X)$ to a differential $\pd_\Omega:\PV(X)\rightarrow
\PV(X)$ which satisfies:
\be
\pd_\Omega(\PV^{i,j}(X))\subset \PV^{i+1,j}(X)~~.
\ee
Consider the bracket $\{\cdot,\cdot\}_{_\Omega}:\PV(X)\times \PV(X)
\rightarrow \PV(X)$ defined through:
\ben
\label{OmegaBracket}
\{\omega,\eta\}_{_\Omega}\eqdef \pd_{\Omega}(\omega\wedge\eta)-(\pd_{\Omega}\omega)\wedge\eta-(-1)^{\deg \omega}\omega\wedge\pd_{\Omega}\eta~.
\een
Then $(\PV(X), \{\cdot,\cdot\}_{_\Omega},\wedge,\pd_\Omega)$ is a Batalin-Vilkovisky
algebra. One can check (see Appendix \ref{app:coord}) that we have:
\ben
\label{ioda_bracket}
\ioda_W=-\i\{W,\cdot\}_{_\Omega}~~,
\een
so the operator $\delta_W$ coincides with the operator denoted by
$\bpd_f$ in \cite{LLS} provided that one takes $f=-\i W$.  Notice that
$\ioda_W$ is independent of $\Omega$, so the formulation in terms of
$\ioda_W$ (which was already used in \cite{LG2}) is more natural.  In
particular, the dg-algebra $(\PV(X),\updelta_W)$ and its cohomology
$\HPV(X,W)$ are independent of $\Omega$. 
\end{remark}

\subsection{Compact support version}

Let: 
\be
\PV_c(X)\subset \PV(X)
\ee
denote the $\cinf$-submodule of $\PV(X)$ consisting of
compactly-supported polyvector-valued forms. Then $\PV_c(X)$ is a subalgebra
of the bidifferential algebra $(\PV(X),\bbpd,\ioda_W)$. In particular
$\PV_c(X)$ (when endowed with the canonical $\Z$-grading) is a 
dg-subalgebra of the $\O(X)$-linear dg-algebra $(\PV(X),\updelta_W)$.

\

\begin{Definition}
The {\em compactly supported twisted Dolbeault algebra of
  polyvector-valued forms} is the $\Z$-graded $\O(X)$-linear dg-algebra
$(\PV_c(X),\updelta_W)$, where $\PV_c(X)$ is endowed with the
canonical $\Z$-grading.
\end{Definition}

\

\noindent Let $\HPV_c(\!X,\!W\!)\!\eqdef\! \rH(\PV_c(\!X\!),\!\updelta_W\!)$ 
be the total
cohomology algebra of the dg-algebra $(\PV_c(\!X\!),\!\updelta_W\!)$. Then
$\HPV_c(X,W)$ is a $\Z$-graded $\O(X)$-linear {\em non-unital}
associative and supercommutative algebra when endowed with the
canonical $\Z$-grading. Let $i_\ast:\HPV_c(X,W)\rightarrow \HPV(X,W)$
denote the morphism of $\Z$-graded $\O(X)$-algebras induced on
cohomology by the inclusion $i:\PV_c(X)\rightarrow \PV(X)$~.

\

\begin{Proposition}
\label{prop:i}
Assume that the critical set $Z_W$ is compact. Then the inclusion
$i:\PV_c(X)\hookrightarrow \PV(X)$ is a homotopy equivalence (and
hence a quasi-isomorphism) of differential $\Z$-graded $\O(X)$-modules
from $(\PV_c(X),\updelta_W)$ to $(\PV(X),\updelta_W)$.
\end{Proposition}

\begin{proof}
This is an extension to the global case of \cite[Lemma 2.3]{LLS} (see
also \cite[Lemma 3.1]{ChiangLi}); we give the proof in detail for
clarity and completeness. Let $\updelta:=\updelta_W=\bbpd+\ioda_W$ 
and $\updelta_c:=\updelta|_{\PV_c(X)}$. Let $I:= \id_{\PV(X)}$ and 
$I_c:=\id_{\PV_c(X)}$. We will construct maps:
\be
\pi:\PV(X) \rightarrow \PV_c(X)~~\mathrm{and}~~\wcR:\PV(X) \rightarrow \PV(X) 
\ee
such that $\pi$ is $\O(X)$-linear (in fact, $\cC^\infty(X)$-linear) and such that:
\ben
\label{htrels}
I-i\circ \pi~=[\updelta,~\wcR]~~\mathrm{and}~~I_c -\pi \circ i=[\updelta_c,\wcR_c]~~,
\een
where $\wcR$ preserves the subspace $\PV_c(X)$ and $\wcR_c:=
\wcR|_{\PV_c(X)}$. Let $X_0\eqdef X\setminus Z_W$ and $W_0:=
W|_{X_0}$. To construct $\wcR$, pick a Hermitian metric $G$ on $X$ and
consider the following smooth section of the holomorphic tangent
bundle of $X_0$:
\be
s\eqdef \frac{\i}{||\pd W||_G^2} \grad_G \overline{W}\in \rGamma_\sm(X_0,TX_0)=\PV^{-1,0}(X_0)~~,
\ee
where $||\cdot||_G$ denotes the norm induced by $G$ on $\wedge T^\ast
X$, $\overline{W}:X\rightarrow \C$ denotes the complex conjugate of
$W$ and $\grad_G \overline{W}\in \rGamma_\sm(X,TX)$ denotes the
gradient of $\overline{W}$ taken with respect to $G$ (see Subsection
\ref{subsec:bulkflow} for details). The section $s$ satisfies (cf. equation 
\eqref{normpdW}):
\ben
\label{iodas}
\ioda_W(s)=1_{X_0}~~,
\een
where $1_{X_0}\in \cC^\infty(X_0)$ is the unit function of the open
subset $X_0\subset X$. For any element $\eta\in \PV(X_0)$, let
$\hat{\eta}:\PV(X_0)\rightarrow \PV(X_0)$ denote the operator of
multiplication from the left with $\eta$ in the algebra $\PV(X_0)$:
\be
{\hat \eta}(\omega)\eqdef \eta\omega~~,~~\forall \omega\in \PV(X_0)~~.
\ee
Let $I_0\eqdef \id_{\PV(X_0)}$. Relation \eqref{iodas} implies:
\be
[\updelta, {\hat s}]=I_0+\widehat{\bbpd s}~~,
\ee
where $[\cdot,\cdot]$ is the graded commutator of operators acting in
$\PV(X_0)$:
\be
[A,B]=A\circ B-(-1)^{\deg A\,\deg B}B\circ A~~,~~\forall A,B:\PV(X_0)\to \PV(X_0)~~.
\ee
Since $\bbpd s\in \PV^{-1,1}(X_0)$, the operator $\widehat{\bbpd s}$
is nilpotent. Hence the operator $[\updelta, {\hat s}]$ is
invertible on $\PV(X_0)$ and we have:
\be
[\updelta, {\hat s}]^{-1}=\big(I_0+\widehat{\bbpd s}\big)^{-1}=\sum_{k=0}^d(-1)^k(\widehat{\bbpd s})^k=\hS~~,
\ee
where: 
\be
S\eqdef \big(1+\bbpd s\big)^{-1}=\sum_{k=0}^d(-1)^k (\bbpd s)^k\in \PV^0(X_0)~~. 
\ee
This allows us to define the following operator on the space
$\PV(X_0)$:
\be
\cR\eqdef {\hat s}\circ [\updelta, {\hat s}]^{-1}={\hat s}\circ {\hS}=\hR~~,
\ee
where: 
\be
R\eqdef s S=s\big(1+\bbpd s\big)^{-1}=s\sum_{k=0}^d (-1)^k (\bbpd s)^k\in \PV^{-1}(X_0)~~.
\ee
Since $\cR$ is the operator of left multiplication with $R$, it
preserves the subspace $\PV_c(X_0)$. Applying the operator
$[\updelta,\cdot]$ to the relation: 
\be
[\updelta, {\hat s}]\circ {\hS}=I_0
\ee
gives
\be
[\updelta,{\hS}]=0~~,
\ee
since $\big[\updelta, [\updelta, {\hat s}]\big]=0$ and $[\updelta,{\hat s}]$
is invertible on $\PV(X_0)$.  Thus: 
\ben
\label{ucR}
[\updelta, \cR]=[\updelta, {\hat s}\circ {\hS}]=[\updelta, {\hat s}]\circ {\hS}=I_0~~.
\een
Since the critical set $Z_W$ is compact and of codimension at least
one (recall that $W$ is not constant), there exists a relatively
compact open neighborhood $U$ of $Z_W$. Let $U_1$ be an open subset of
$X$ such that $Z_W\subset U_1\subset \overline{U_1}\subset U$ and
$\rho\in \cC^\infty_c(X)$ be a smooth compactly-supported function
which equals $1$ on $U_1$ and vanishes outside $U$. The
desired operators $\wcR$ and $\pi$ are defined as:
\beqan
&& \wcR\eqdef (1_X-\rho) \cR=(1_X-\rho){\hR}~~\nn\\
&& \pi \eqdef \rho I+(\bpd \rho)\cR= \rho I +(\bpd \rho) {\hR}~~,
\eeqan
where $1_X\in \cinf$ is the unit function of $X$. We have: 
\be
\wcR_c:= \wcR|_{\PV_c(X)}=(1_X-\rho) \cR|_{\PV_c(X_0)}~~.
\ee
Noticing that $\updelta\rho=\bpd\rho$, we compute: 
\ben
\label{udR}
[\updelta, \wcR]=-(\bpd \rho)\cR+(1_X-\rho)[\updelta, \cR]=-(\bpd \rho)\cR+(1_X-\rho)I_0=-(\bpd \rho)\cR+(1_X-\rho)I=I-i\circ \pi~~,
\een
where we used \eqref{ucR} and noticed that
$(1_X-\rho)I_0=(1_X-\rho)I$, since $1_X-\rho$ vanishes on $Z_W$. This
shows that the first relation in \eqref{htrels} is satisfied. On the
other hand, we have:
\be
\pi\circ i=\pi|_{\PV_c(X)}=\rho I_c +(\bpd \rho)\wcR_c~~,
\ee
so restricting \eqref{udR} to $\PV_c(X)$ gives the second relation in
\eqref{htrels}. Since $\pi$ is $\cC^\infty(X)$-linear, relations
\eqref{htrels} imply that $\pi$ is an $\O(X)$-linear map of complexes
and that $i$ is a homotopy equivalence having $\pi$ as an inverse up
to the homotopies provided by $\wcR$ and $\wcR_c$. \qed
\end{proof}

\section{The dg-category of topological D-branes}
\label{sec:DF}

Let $(X,W)$ be a Landau-Ginzburg pair with $\dim_\C X=d$. In this
subsection, we describe a $\Z_2$-graded dg-category $\DF(X,W)$ whose
objects are so-called {\em holomorphic factorizations} of $W$ (which
can be viewed as a certain global version of matrix factorizations) but
whose morphisms are bundle-valued differential forms; the differentials 
on the Hom spaces are certain deformations of the Dolbeault differential. 
The path integral arguments of \cite{LG2} imply that the total
cohomology category of $\DF(X,W)$ can be interpreted as the category
of topological D-branes of the B-type Landau-Ginzburg model defined by
$(X,W)$. The precise relation with the superconnection formalism used
in loc. cit. is explained in Appendix \ref{app:sc}.

\subsection{A category of holomorphic vector superbundles} 

We start by describing a category of holomorphic vector superbundles
which is equivalent with the $\Z_2$-graded $\O(X)$-linear category of
finite rank sheaves of locally-free $\cO_X$-supermodules defined on
$X$. Recall the categories $\VB(X)$ and $\VB_\sm(X)$ defined in
Subsection \ref{subsec:notation}.

\

\begin{Definition}
A {\em holomorphic vector superbundle} on $X$ is a $\Z_2$-graded
holomorphic vector bundle defined on $X$, i.e. a complex holomorphic
vector bundle $E$ endowed with a direct sum decomposition
$E=E^\0\oplus E^\1$, where $E^\0$ and $E^\1$ are holomorphic
sub-bundles of $E$.
\end{Definition}

\

\begin{Definition} 
The $\Z_2$-graded $\O(X)$-linear category $\VB^s(X)$ is defined as
follows.
\begin{itemize}
\itemsep 0.0em
\item The objects are the holomorphic vector superbundles on $X$.
\item Given two holomorphic vector superbundles $E$ and $F$ on $X$, let: 
\be
\Hom_{\VB^s(X)}(E,F)\eqdef \rGamma(X,Hom(E,F))
\ee
be the $\O(X)$-module of holomorphic sections of the bundle
$Hom(E,F)\eqdef E^\vee\otimes F$, endowed with the $\Z_2$-grading with
homogeneous components:
\beqa
&& \Hom^\0_{\VB^s(X)}(E,F)\eqdef \rGamma(X,Hom(E^\0,F^\0))\oplus \rGamma(X, Hom(E^\1,F^\1))~~\nn\\
&& \Hom_{\VB^s(X)}^\1(E,F)\eqdef \rGamma(X,Hom(E^\0,F^\1))\oplus \rGamma(X, Hom(E^\1,F^\0))~~~~.
\eeqa
\item The composition of morphisms is induced from that of $\VB(X)$. 
\end{itemize}
\end{Definition}

\begin{remark}
Given a holomorphic vector superbundle $E=E^\0\oplus E^\1$, the
locally-free sheaf of $\cO_X$-modules $\cE\eqdef \rGamma(E)$ defined by
$E$ is $\Z_2$-graded by the decomposition $\cE=\cE^\0\oplus \cE^\1$,
where $\cE^\0:= \rGamma(E^\0)$ and $\cE^\1:= \rGamma(E^\1)$ are
the locally-free sheaves corresponding to the holomorphic vector
bundles $E^\0$ and $E^\1$. We have:
\beqa
&& \Hom^\0_{\VB^s(X)}(E,F)\eqdef \Hom(\cE^\0,\cF^\0)\oplus \Hom(\cE^\1,\cF^\1)~~\nn\\
&& \Hom_{\VB^s(X)}^\1(E,F)\eqdef \Hom(\cE^\0,\cF^\1)\oplus \Hom(\cE^\1,\cF^\0)~~,
\eeqa
where $\Hom$ denotes the outer Hom of sheaves of $\cO_X$-modules. Thus
$\VB^s(X)$ can be identified with the $\Z_2$-graded $\O(X)$-linear
category of $\Z_2$-graded locally-free sheaves of $\cO_X$-modules.
\end{remark}

\

\noindent Let $E$ and $F$ be two holomorphic vector superbundles on
$X$. The {\em graded direct sum} of $E$ with $F$ is the direct sum
$E\oplus F$ of the underlying holomorphic vector bundles, endowed with
the $\Z_2$-grading given by:
\be
(E\oplus F)^\kappa=E^\kappa\oplus F^\kappa ~~,~~\forall \kappa\in \Z_2~~. 
\ee
The {\em graded tensor product} of $E$ with $F$ is the ordinary tensor
product $E\otimes F$ of the underlying holomorphic vector bundles, endowed
with the $\Z_2$-grading given by:
\beqa
& & (E\otimes F)^\0\eqdef (E^\0\otimes F^\0)\oplus (E^\1\otimes F^\1)\\
& & (E\otimes F)^\1\eqdef (E^\0\otimes F^\1)\oplus (E^\1\otimes F^\0)~~.
\eeqa
The {\em graded dual} $E$ is the holomorphic vector superbundle whose
underlying holomorphic vector bundle is the ordinary dual
$E^\vee$ of $E$, endowed with the $\Z_2$-grading
given by:
\be
(E^\vee)^\kappa\eqdef (E^\kappa)^\vee ~~,~~\forall \kappa\in \Z_2~~. 
\ee
The {\em holomorphic vector superbundle of morphisms} from $E$ to $F$
is the holomorphic vector superbundle $Hom(E,F)\eqdef E^\vee\otimes
F$, where $E^\vee$ is the graded dual and $\otimes$ is the graded
tensor product. Thus $Hom(E,F)$ is the usual bundle of morphisms
between the underlying holomorphic vector bundles, endowed with the
$\Z_2$-grading given by:
\beqa
&& Hom^\0(E,F)\eqdef Hom(E^\0,F^\0)\oplus Hom(E^\1,F^\1)~~\nn\\
&& Hom^\1(E,F)\eqdef Hom(E^\0,F^\1)\oplus Hom(E^\1,F^\0)~~~~.
\eeqa 
When $F=E$, we set $End(E)\eqdef Hom(E,E)$, 
$End^\kappa(E)\eqdef Hom^\kappa(E,E)$ etc.

\subsection{Holomorphic factorizations of $W$}

We next introduce holomorphic factorizations of $W$ and two natural 
dg-categories constructed with such objects, which we denote by
$\F_\sm(X,W)$ and $\F(X,W)$. The total cohomology category
$\HF(X,W)$ of $\F(X,W)$ can be viewed as a {\em naive} candidate for the
category of topological D-branes of the Landau-Ginzburg theory defined
by the pair $(X,W)$. In general, however, the category $\HF(X,W)$ {\em
  differs} from the correct category $\HDF(X,W)$ which results from
the path integral arguments of \cite{LG2} and is described in the next
subsection.

\

\begin{Definition}
A {\em holomorphic factorization} of $W$ is a pair $a=(E,D)$, where
$E=E^\0\oplus E^\1$ is a holomorphic vector superbundle on $X$ and $D\in
\rGamma(X,End^\1(E))$ is a holomorphic section of the bundle
$End^\1(E)=Hom(E^\0,E^\1)\oplus Hom(E^\1,E^\0)\subset End(E)$ which
satisfies the condition $D^2=W\id_E$.
\end{Definition}

\

\begin{remark}
Let $a=(E,D)$ be a holomorphic factorization of
$W$. Decomposing $E=E^\0\oplus E^\1$, the condition that $D$ is odd
implies that $D$ has the form:
\ben
\label{Dmatrix}
D=\left[\begin{array}{cc} 0 & v\\ u & 0\end{array}\right]~~,
\een
where $u\in \rGamma(X, Hom(E^\0,E^\1))$ and $v\in
\rGamma(X,Hom(E^\1,E^\0))$. The condition $D^2=W\id_E$ amounts to the
relations:
\be
v\circ u=W \id_{E^\0}~~,~~u\circ v=W \id_{E^\1}~~.
\ee
\end{remark}

\

\begin{Definition}
The {\em smooth dg-category of holomorphic factorizations} of $W$ is
the $\Z_2$-graded $\cinf$-linear dg-category $\F_\sm(X,W)$ defined
as follows:
\begin{itemize}
\itemsep 0.0em
\item The objects are the holomorphic factorizations of $W$.
\item Given two holomorphic factorizations
  $a_1=(E_1,D_1)$ and $a_2=(E_2,D_2)$ of $W$, we set: 
\be
\Hom_{\F_\sm(X,W)}(a_1,a_2)=\rGamma_\sm(X,Hom(E_1,E_2))~~,
\ee
endowed with the $\Z_2$-grading with homogeneous components: 
\be
\Hom^\kappa_{\F_\sm(X,W)}(a_1,a_2)\eqdef \rGamma_\sm(X,Hom^\kappa(E_1,E_2))~~,~~\forall \kappa\in \Z_2
\ee
and with the differentials (called {\em defect differentials}) $\fd_{a_1,a_2}$ determined
uniquely by the condition:
\be
\fd_{a_1,a_2}(f)\eqdef D_2\circ f-(-1)^\kappa f\circ D_1~~,~~\forall f\in \rGamma_\sm(X,Hom^\kappa(E_1,E_2))~~,~~\forall \kappa\in \Z_2~~.
\ee
\item The composition of morphisms is induced by that of
  $\VB_\sm(X)$.
\end{itemize}
The {\em smooth cohomological category of holomorphic factorizations}
is the $\Z_2$-graded $\cinf$-linear category $\HF_\sm(X,W)\eqdef
\rH(\F_\sm(X,W))$.
\end{Definition}

\

\begin{Definition}
The {\em holomorphic dg-category of holomorphic factorizations} of $W$
is the $\Z_2$-graded $\O(X)$-linear dg-category $\F(X,W)$ defined as
follows:
\begin{itemize}
\itemsep 0.0em
\item The objects are the holomorphic factorizations of $W$.
\item Given two holomorphic factorizations
  $a_1=(E_1,D_1)$ and $a_2=(E_2,D_2)$ of $W$, we set: 
\be
\Hom_{\F(X,W)}(a_1,a_2)=\rGamma(X,Hom(E_1,E_2))~~,
\ee
endowed with the $\Z_2$-grading with homogeneous components: 
\be
\Hom_{\F(X,W)}^\kappa(a_1,a_2)\eqdef \rGamma(X, Hom^\kappa(E_1,E_2))~~,~~\forall \kappa\in \Z_2
\ee
and with the differentials $\fd_{a_1,a_2}$ determined
uniquely by the condition:
\be
\fd_{a_1,a_2}(f)\eqdef D_2\circ f-(-1)^\kappa f\circ D_1~~,~~\forall f\in \rGamma(X, Hom^\kappa(E_1,E_2))~~,~~\forall \kappa\in \Z_2~~.
\ee
\item The composition of morphisms is induced by that of $\VB(X)$.
\end{itemize}
The {\em holomorphic cohomological category of holomorphic
  factorizations} is the $\Z_2$-graded $\O(X)$-linear category
$\HF(X,W)\eqdef \rH(\F(X,W))$.
\end{Definition}

\

\noindent Notice that $\F(X,W)$ is a non-full dg-subcategory of the
$\O(X)$-linear dg-category obtained from $\F_\sm(X,W)$ by
restriction of scalars.

\

\begin{remark} 
Given two holomorphic factorizations $a_1=(E_1,D_1)$ and
$a_2=(E_2,D_2)$ of $W$, write $D_i=\left[\begin{array}{cc} 0 &
    v_i\\ u_i & 0\end{array}\right]$ as in \eqref{Dmatrix}, where
$u_i\in \rGamma(X,Hom(E_i^\0,E_i^\1))$ and $v_i\in
\rGamma(X,Hom(E_i^\1,E_i^\0))$. Then an even morphism $f\in
\Hom_{\F(X,W)}^\0(a_1,a_2)$ has the form $f=\left[\begin{array}{cc}
    f_{\0\0} & 0\\ 0 & f_{\1\1}\end{array}\right]$ with $f_{\0\0}\in
\rGamma(X, Hom(E^\0_1,E^\0_2))$, $f_{\1\1}\in
\rGamma(X,Hom(E^\1_1,E^\1_2))$ and we have:
\ben
\label{even_mor_need}
\fd_{a_1,a_2}(f)=D_2\circ f-f\circ D_1=\left[\begin{array}{cc} 0 & v_2\circ f_{\1\1}-f_{\0\0}\circ v_1\\ u_2\circ f_{\0\0}-f_{\1\1}\circ u_1 &
    0\end{array}\right]~~.
\een
On the other hand, an odd morphism $g\in
\Hom^\1_{\F(X,W)}(a_1,a_2)$ has the form
$g=\left[\begin{array}{cc} 0 & g_{\1\0}\\ g_{\0\1} & 0\end{array}\right]$
with $g_{\1\0}\in \rGamma(X, Hom(E^\1_1,E^\0_2))$, $g_{\0\1}\in
\rGamma(X, Hom(E^\0_1,E^\1_2))$ and we have:
\ben
\label{odd_mor_need}
\fd_{a_1,a_2}(g)=D_2\circ g+g\circ D_1=\left[\begin{array}{cc} v_2\circ g_{\0\1}+g_{\1\0}\circ u_1 & 0 \\  0 &
  u_2\circ g_{\1\0}+g_{\0\1}\circ v_1  \end{array}\right]~~.
\een
\end{remark}

\subsection{The twisted Dolbeault category of holomorphic factorizations}

We are now ready to describe the $\Z_2$-graded dg-category $\DF(X,W)$
which results from the path integral arguments of \cite{LG2} and whose 
total cohomology category $\HDF(X,W)$ is the topological D-brane category 
of the B-type Landau-Ginzburg theory defined by $(X,W)$. We refer
the reader to Appendix \ref{app:sc} for the precise relation with the
superconnection formalism used in loc. cit. The category $\DF(X,W)$ is
closely related to the Dolbeault category of holomorphic vector
superbundles, so we start by introducing the latter.

\paragraph{The Dolbeault category of holomorphic vector superbundles}

For any holomorphic vector superbundle $E=E^\0\oplus E^\1$ on $X$, the
$\cinf$-module of smooth sections $\rGamma_\sm(X,E)$ is
$\Z_2$-graded with homogeneous components
$\rGamma^\kappa_\sm(X,E)\eqdef
\rGamma_\sm(X,E^\kappa)$. Accordingly, the $\cinf$-module
$\cA(X,E)\eqdef \cA(X)\otimes_{\cinf} \rGamma_\sm(X,E)$ has an
induced $\Z\times \Z_2$-grading. This bigrading corresponds to the
decomposition:
\be
\cA(X,E)=\bigoplus_{k=0}^d\bigoplus_{\kappa\in \Z_2}\cA^k(X,E^\kappa)~~,
\ee
where: 
\be
\cA^k(X,E^\kappa)=\cA^k(X)\otimes_\cinf \rGamma_\sm(X,E^\kappa)~~.
\ee
The $\Z$-grading is called the {\em rank grading} of $\cA(X,E)$ and
corresponds to the decomposition:
\be
\cA(X,E)=\bigoplus_{i=0}^d \cA^i(X,E)~~,
\ee
where $\cA^i(X,E)\eqdef \cA^i(X)\otimes_{\cinf} \rGamma_\sm (X,E)$.
The $\Z_2$-grading is called the {\em bundle grading} of $\cA(X,E)$
and corresponds to the decomposition:
\be
\cA(X,E)=\cA(X,E^\0)\oplus \cA(X,E^\1)~~,
\ee
where $\cA(X,E^\kappa)\eqdef \cA(X)\otimes_\cinf
\rGamma_\sm(X,E^\kappa)$. For any $\alpha\in \cA^k(X,E)$, let
$\rk\alpha\eqdef k\in \Z$. For any $\alpha\in \cA(X,E^\kappa)$, let
$\sigma(\alpha)\eqdef \kappa\in \Z_2$. The {\em total grading} of
$\cA(X,E)$ is the $\Z_2$-grading given by the decomposition:
\be
\cA(X,E)=\cA(X,E)^\0\oplus \cA(X,E)^\1~~, 
\ee
where: 
\beqa
& &\cA(X,E)^\0\eqdef \bigoplus_{i=\ev} \cA^i(X,E^\0)\oplus \bigoplus_{i=\odd} \cA^i(X,E^\1)\\
& &\cA(X,E)^\1\eqdef \bigoplus_{i=\ev} \cA^i(X,E^\1)\oplus \bigoplus_{i=\odd} \cA^i(X,E^\0)~~.
\eeqa
For any $\alpha\in \cA(X,E)^\kappa$, we set $\deg\alpha=\kappa\in
\Z_2$. For all $\rho\in \cA^k(X)$ and all $s \in
\rGamma_\sm(M,E^\kappa)$, we have:
\begin{equation}
\deg(\rho\otimes s) = {\hat k}+ \kappa \in \Z_2~~.
\end{equation}

\

\begin{Definition}
The {\em Dolbeault category $\fD^s(\!X\!)$ of holomorphic vector
  superbundles} is the $\cinf$-linear $\Z\times \Z_2$-graded category
defined as follows:
\begin{itemize}
\itemsep 0.0em
\item The objects are the holomorphic vector superbundles defined on
  $X$.
\item The $\cinf$-modules of morphisms are given by:
\be
\Hom_{\fD^s(X)}(E,F)=\cA(X,Hom(E,F))=\cA(X)\otimes_\cinf\rGamma_\sm(X,Hom(E,F))~~,
\ee
endowed with the obvious $\Z\times \Z_2$-grading.
\item The $\cinf$-bilinear composition of morphisms
$\circ:\cA(X,Hom(F,G))\times \cA(X,Hom(E,F))\rightarrow
\cA(X,Hom(E,G))$ is determined uniquely through the condition:
\ben
\label{Dscomp}
(\rho\otimes f)\circ (\eta\otimes g)=(-1)^{\sigma(f) \, \rk \eta}(\rho\wedge \eta) \otimes (f\circ g)~~
\een
for all pure rank forms $\rho,\eta\in \cA(X)$ and all pure
$\Z_2$-degree elements $f\in \rGamma_\sm(X,Hom(F,G))$ and $g\in
\rGamma_\sm(X, Hom(E,F))$.
\end{itemize}
\end{Definition}

\paragraph{The twisted Dolbeault category of holomorphic factorizations}

Consider two holomorphic factorizations $a_1=(E_1,D_1)$ and
$a_2=(E_2,D_2)$ of $W$. Then the space $\cA(X,Hom(E_1,E_2))$ 
carries two natural differentials: 
\begin{itemize}
\itemsep 0.0em
\item The {\em Dolbeault differential}
  $\bbpd:=\bbpd_{a_1,a_2}=\bbpd_{Hom(E_1,E_2)}$ determined on
  $\cA(X,Hom(E_1,E_2))$ by the holomorphic structure of the vector bundle
  $Hom(E_1,E_2)$. This differential is $\O(X)$-linear, preserves the
  bundle grading and has degree $+1$ with respect to the rank grading.
\item The {\em defect differential}, i.e. the $\cinf$-linear
  endomorphism $\md\!:=\!\md_{a_1,a_2}$ of $\cA(X,Hom(E_1,E_2))$
  determined uniquely by the condition:
\be
\md_{a_1,a_2}(\rho\otimes f)=(-1)^{\rk \rho} \rho \otimes (D_2\circ f)-(-1)^{\rk \rho +\sigma(f)} \rho\otimes (f\circ D_1)
\ee
for all pure rank forms $\rho \in \cA(X)$ and all pure $\Z_2$-degree
elements $f\in \rGamma_\sm(X,Hom(E_1,E_2))$. Notice that $\md$
squares to zero because $D_i$ square to $W\id_{E_i}$. Moreover, this
differential preserves the rank grading and is odd with respect to the
bundle grading.
\end{itemize}

\noindent Since $D_i$ are holomorphic, the Dolbeault and defect
differentials anticommute. Thus:
\be
\bbpd^2=\md^2=\bbpd\circ\md+\md\circ \bbpd=0~~,
\ee
and hence $(\cA(X,Hom(E_1,E_2)),\bbpd_{a_1,a_2},\md_{a_1,a_2})$ is a
$\Z\times\Z_2$-graded bicomplex.

\

\begin{Definition}
The {\em twisted differential $\updelta:=\updelta_{a_1,a_2}$} of the
ordered pair $(a_1,a_2)$ is the total differential of the bicomplex
$\big(\cA(X,Hom(E_1,E_2)),\bbpd_{a_1,a_2},\md_{a_1,a_2}\big)$:
\be
\updelta_{a_1,a_2}\eqdef \bbpd_{a_1,a_2}+\md_{a_1,a_2}~~.
\ee
\end{Definition}

\noindent Notice that the twisted differential is odd with respect to the total
$\Z_2$-grading.

\

\begin{Definition}
The {\em twisted Dolbeault category of holomorphic factorizations} of
$(X,W)$ is the $\Z_2$-graded $\O(X)$-linear dg-category $\DF(X,W)$
defined as follows:
\begin{itemize}
\itemsep 0.0em
\item The objects of $\DF(X,W)$ are the holomorphic factorizations of
  $W$.
\item Given two holomorphic factorizations $a_1=(E_1,D_1)$ and
  $a_2=(E_2,D_2)$, define:
\be
\Hom_{\DF(X,W)}(a_1,a_2)\eqdef \cA(X,Hom(E_1,E_2))~~,
\ee
endowed with the total $\Z_2$-grading and with the twisted
differentials $\updelta_{a_1,a_2}$.
\item The composition of morphisms coincides with that of $\fD^s(X)$
  (see \eqref{Dscomp}).
\end{itemize}
\end{Definition}

\

\noindent It is easy to check that $\DF(X,W)$ is indeed a dg-category. 
As a $\Z_2$-graded category, $\DF(X,W)$ coincides with the
$\O(X)$-linear category obtained from $\fD^s(X)$ by restriction of
scalars. Let:
\be
\HDF(X,W)\eqdef \rH(\DF(X,W))
\ee
denote its total cohomology category, which is $\Z_2$-graded and
$\O(X)$-linear. This can also be viewed as a $\Z_2$-graded
$\C$-linear category by restriction of scalars.

\

\begin{Proposition}
\label{prop:boundaryfd}
Assume that the critical set $Z_W$ is compact. Then $\HDF(X,W)$ is Hom-finite 
as a $\C$-linear category. 
\end{Proposition}

\begin{proof}
Let $a_1=(E_1,D_1)$ and $a_2=(E_2,D_2)$ be two holomorphic
factorizations of $W$. Set: 
\be
\updelta_1:=\updelta_{a_1, a_1}~~,~~\updelta_2:=\updelta_{a_2 ,a_2}~~,~~\updelta:=\updelta_{a_1,a_2}~~. 
\ee
The space $\Hom_{\DF(X,W)}(a_1,a_2)$ coincides with the space of
global sections of the coherent sheaf of $\cO_X$-modules
$\cH_{a_1,a_2}$ defined by $\cH_{a_1,a_2}(U)\eqdef
\rH(\cA(U,Hom(E_1,E_2)),\updelta)$ for any open subset $U\subset
X$. Since $W$ is non-constant, the critical set $Z_W$ is a closed
proper subset of $X$. Choose a point $x\in X\setminus Z_W$ and let
$U\subset X\setminus Z_W$ be an open subset containing $x$ and
supporting local complex coordinates $z^1,\ldots, z^d$. We can choose
$U$ small enough such that $E_i|_U$ are holomorphically
trivializable. Picking holomorphic trivializations of $E_i|_U$ allows
us to identify $D_i|_U$ with matrices whose entries are complex-valued
holomorphic functions defined on $U$ and we tacitly do so below; this
allows us to define the holomorphic partial derivatives $\partial_k D_i=\frac{\partial D_i}{\partial z^k}$.
Since $x$ does not belong to $Z_W$, at least one of the partial
derivatives $\pd_i W:=\frac{\partial W}{\partial z^i}$ does not vanish
at $x$ and, without loss of generality, we can assume that $\pd_1
W(x)\neq 0$. By shrinking $U$ if necessary, we can further assume that
$\pd_1 W$ doesn't vanish at any point of $U$. Then the relation
$D_2^2=W\id_{E_2}$ implies:
\be
[D_2,\pd_1 D_2]=D_2(\pd_1 D_2)+(\pd_1 D_2) D_2=(\pd_1 W)\id_{E_2}~~,
\ee
which gives: 
\be
\id_{E_2}=[D_2,f]=\updelta_2(f)~~,
\ee
where $f\eqdef \frac{1}{\pd_1 W} \pd_1 D_2\in \rGamma(U,
End^1(E_2))\subset \cA(U,End^1(E_2))$ and we used the fact that $f$ is
holomorphic. For any $\alpha\in \cA(U,Hom(E_1,E_2))$ such that
$\updelta \alpha=0$, we have:
\be
\alpha=\id_{E_2}\alpha=\updelta_2(f)\alpha=\updelta(f\alpha)~~.
\ee
Thus $\cH_{a_1,a_2}(U)=\rH(\cA(U,Hom(E_1,E_2)),\updelta)=0$. This
shows that $\cH_{a_1,a_2}$ is supported on the critical set
$Z_W$. Thus $\Hom_{\DF(X,W)}(a_1,a_2)=\cH_{a_1,a_2}(X)$ is
finite-dimensional since $Z_W$ is compact.  ~\qed
\end{proof}

\subsection{The extended supertrace}
\label{subsec:strext}

Let $a=(E,D)$ be a holomorphic factorization of $W$ and set
$\updelta_a:=\updelta_{a,a}$. Let $\str:\rGamma_\sm(X,End(E))\rightarrow \cinf$ 
denote the fiberwise supertrace, which is an even map when $\C$ is 
viewed as a $\Z_2$-graded $\C$-algebra concentrated in degree 
zero\footnote{This means that $\str(s)=0$ for $s\in
  \rGamma_\sm(X,End^\1(E))$.}. Noticing that $\cA(X,End(E))\simeq
\cA(X)\otimes_\cinf\rGamma_\sm(X,End(E))$ is a left module over
$\cA(X)$, we extend the fiberwise supertrace to a left $\cA(X)$-linear
map denoted by the same symbol:
\be
\str\eqdef \id_{\cA(X)} \otimes_\cinf \str:\cA(X,End(E))\rightarrow \cA(X)~~.
\ee
For any $\rho\in \cA(X)$ and any $s\in \rGamma_\sm(X,End(E))$, we have:
\be
\str(\rho\otimes s)=\rho \,\str(s)~~
\ee
and the extended supertrace satisfies: 
\ben
\label{strextcyc}
\str(\alpha\beta)=(-1)^{\deg \alpha\, \deg \beta}\str(\beta\alpha)~~,~~\forall \alpha,\beta\in \End_{\DF(X,W)}(a)=\cA(X,End(E))~~.
\een
The extended supertrace is a morphism of complexes of
$\O(X)$-modules from $(\End_{\DF(X,W)}(a),\updelta_a)$ to
$(\cA(X),\bpd)$, when the latter is endowed with the mod 2 reduction
of its $\Z$-grading:
\ben
\label{strbpd}
\str(\updelta_a\alpha)=\bpd \,\str(\alpha)~~,~~\forall \alpha\in \End_{\DF(X,W)}(a)~~
\een
and satisfies the relation: 
\ben
\label{strfd}
\str(\fd_a\alpha)=0~~,~~\forall \alpha\in \End_{\DF(X,W)}(a)~~.
\een
Notice that $\str(\omega)$ vanishes unless $\omega\in \cA(X,End(E))$ is even
with respect to the bundle grading. 

\subsection{Compact support version}

\

\

\begin{Definition}
The {\em compactly-supported twisted Dolbeault category of holomorphic
  factorizations of $W$} is the $\O(X)$-linear $\Z_2$-graded dg-category 
  $\DF_c(X,W)$ defined as follows:
\begin{itemize}
\itemsep 0.0em
\item The objects of $\DF_c(X,W)$ are the holomorphic factorizations
  of $W$.
\item Given two objects $a_1=(E_1,D_1)$ and $a_2=(E_2,D_2)$, set:
\be
\Hom_{\DF_c(X,W)}(a_1,a_2)\!\eqdef \!\!\big\{\alpha\in \Hom_{\DF(X,W)}(a_1,\!a_2)\!=\!\cA(X,Hom(E_1,\! E_2))\big|\, \supp(\alpha)\mathrm{~is ~compact}\big\}
\ee
\item The differentials and composition of morphisms are induced from
  $\DF(X,W)$.
\end{itemize}
\end{Definition}

\noindent Notice that $\DF_c(X,W)$ is a non-full subcategory of
$\DF(X,W)$. Let:
\be
\HDF_c(X,W)\eqdef \rH(\DF_c(X,W))
\ee
denote the total cohomology category of $\DF_c(X,W)$; this is a
$\Z_2$-graded $\O(X)$-linear category.  Let
$j:\DF_c(X,W)\rightarrow \DF(X,W)$ be the inclusion functor and
$j_\ast:\HDF_c(X,W)\rightarrow \HDF(X,W)$ denote the functor induced
by $j$ between the total cohomology categories.

\

\begin{Proposition}
\label{prop:j}
Assume that the critical set $Z_W$ is compact. Then the inclusion
functor $j:\DF_c(X,W)\rightarrow \DF(X,W)$ is a quasi-equivalence of
$\Z_2$-graded $\O(X)$-linear dg-categories.
\end{Proposition}

\begin{proof}
This is an extension to the global case of \cite[Proposition
  3.3]{DS3}; the proof is similar to that of Proposition \ref{prop:i}.
Let $a_1=(E_1,D_1)$ and $a_2=(E_2,D_2)$ be two holomorphic
factorizations of $W$. Let:
\beqan
&& \updelta:=\updelta_{a_1,a_2}~~,~~\updelta_c:=\updelta|_{\Hom_{\DF_c(X,W)}(a_1,a_2)}~~\nn\\
&& I:= \id_{\Hom_{\DF(X,W)}(a_1,a_2)}~~,~~I_c:= \id_{\Hom_{\DF_c(X,W)}(a_1,a_2)}~~.
\eeqan
Let $\updelta_i:=\updelta_{a_i}$, $\md_i:=\md_{a_i}=[D_i,\cdot]$,
$\bbpd_i:=\bbpd_{a_i}=\bbpd_{End(E_i)}$ and $\updelta:=\updelta_{a_1,a_2}$,
$\md:=\md_{a_1,a_2}$, $\bbpd:=\bbpd_{a_1,a_2}=\bbpd_{Hom(E_1,E_2)}$. Let
$j:=j_{a_1,a_2}:\Hom_{\DF_c(X,W)}(a_1,a_2)\rightarrow \Hom_{\DF(X,W)}(a_1,a_2)$
denote the inclusion morphism. We will construct maps:
\beqan
&& \pi:=\pi_{a_1,a_2}: \Hom_{\DF(X,W)}(a_1,a_2) \rightarrow \Hom_{\DF_c(X,W)}(a_1,a_2)\nn\\
&& \wcR:=\wcR_{a_1,a_2}:\Hom_{\DF(X,W)}(a_1,a_2)\rightarrow \Hom_{\DF(X,W)}(a_1,a_2)
\eeqan
such that $\pi$ is a map of complexes and such that: 
\beqan
\label{hrels}
I~-j\circ \pi &=& [\delta, ~\wcR]\nn\\
I_c-\pi\circ j &=& [\delta_c, \wcR_c]~~,
\eeqan
where $\wcR$ preserves the subspace $\Hom_{\DF_c(X,W)}(a_1,a_2)$ and
$\wcR_c\eqdef \wcR|_{\Hom_{\DF_c(X,W)}(a_1,a_2)}$. Let $X_0\eqdef
X\setminus Z_W$ and $W_0\eqdef W|_{X_0}$. To construct $\wcR$, pick a
Hermitian metric $G$ on $X$ and a Hermitian metric $h$ on $E$ such
that $h_{E^\0\times E^\1}=0$\footnote{This is a so-called admissible
  metric on $E$, see Subsection \ref{subsec:hhf}.} and let $\langle
\cdot , \cdot \rangle$ denote the Hermitian metric induced by $G$ and
$h$ on $\wedge T^\ast X$ and $||\cdot||$ denote the corresponding
norm. We trivially extend the former to a pairing:
\be
\langle \cdot , \cdot \rangle:(\wedge T^\ast X)\times (\wedge T^\ast X \otimes End(E))\rightarrow End(E)~~.
\ee
Let $\mpd^h_{E_2}$ be the operator defined in \eqref{mpd_E^h} and
$\mpd$ be the induced differential of the graded associative algebra
$\rOmega(X,End(E_2))$ (see Subsection \ref{subsec:hhf}). Differentiating
the relation $D_2^2=W\id_{E_2}$ gives:
\be
D_2\circ (\mpd D_2)+(\mpd D_2) \circ D_2=(\pd W)\id_{E_2}~~.
\ee
This implies that the following relation holds on $X_0$:
\ben
\label{mdrel}
\md_2(s)=[D_2,s]=D_2|_{X_0}\circ s+s\circ D_2|_{X_0}=\id_{E_2}|_{X_0}~~,
\een
where: 
\be
s\eqdef \frac{\langle \pd W, \partial D_2 \rangle}{||\pd W||^2}\in \rGamma_{\infty}(X_0,End^\1(E_2))\subset \Hom_{\DF(X_0,W_0)}(a_2,a_2)^\1~~.
\ee
For any element $\beta\in \Hom_{\DF(X_0,W_0)}(a_2,a_2)$, let ${\hat
  \beta}: \Hom_{\DF(X_0,W_0)}(a_1,a_2)\rightarrow
\Hom_{\DF(X_0,W_0)}(a_1,a_2)$ denote the operator of left composition
with $\beta$ in the dg-category $\DF(X_0,W_0)$ (see the commutative diagram
\eqref{diagram:hatbeta}):
\be
{\hat \beta}(\alpha)\eqdef \beta\alpha\in \Hom_{\DF(X_0,W_0)}(a_1,a_2)~~,~~\forall \alpha\in \Hom_{\DF(X_0,W_0)}(a_1,a_2)~~.
\ee
\ben
\scalebox{1.0}{
\xymatrix{
a_1 \ar[dr]_{{\hat \beta}(\alpha)} \ar[r]^\alpha & a_2\ar[d]^{\beta}\\
& a_2}}
\label{diagram:hatbeta}
\een
Relation \eqref{mdrel} implies:
\be
[\md,{\hat s}]=\widehat{\md_2(s)}=\widehat{\id_{E_2}}=I_0~~,
\ee
where $I_0\eqdef\id_{\Hom_{\DF(X_0,W_0)}(a_1,a_2)}$. Since $\bbpd_2
s\in \cA^1(X,End^\1(E_2))$ has rank one, the operator $[\bbpd,
  {\hat s}]=\widehat{\bbpd_2 s}$ is nilpotent on the space
$\Hom_{\DF(X_0,W_0)}(a_1,a_2)$. This implies that the operator:
\be
[\updelta,{\hat s}]=I_0+[\bbpd,{\hat s}]=I_0+\widehat{\bbpd_2 s}
\ee
is invertible on the same space. We have: 
\be
[\updelta, {\hat s}]^{-1}=\big(I_0+\widehat{\bbpd_2 s}\big)^{-1}=\hS~~,
\ee
where: 
\be
S\eqdef (\id_{E_2}^{(0)}+\bbpd_2 s)^{-1}=\sum_{k=0}^d (-1)^k (\bbpd_2 s)^k\in \cA(X,End(E_2))^\0=\Hom_{\DF(X_0,W_0)}(a_2)^\0~~
\ee
and $\id_{E_2}^{(0)}\in \rGamma_\sm(X_0,End(E_2))$ is the identity endomorphism of $E_2|_{X_0}$. 
Consider the following operator defined on $\Hom_{\DF(X_0,W_0)}(a_1,a_2)$:
\be
\cR\eqdef {\hat s} \circ {\hS}=\hR~~,
\ee
where: 
\be
R\eqdef s S~~.
\ee
Applying the operator $[\updelta,\cdot]$ to the relation: 
\be
[\updelta, {\hat s}]\circ {\hS}=I_0
\ee
gives:
\be
[\updelta, {\hS}]=0~~,
\ee
since $[\updelta, [\updelta, {\hat s}]]=0$ and $[\updelta,{\hat s}]$
is invertible.  Thus:
\ben
\label{ucRb}
[\updelta, \cR]=[\updelta, {\hat s}]\circ {\hS}=I_0~~.
\een
Since the critical set $Z_W$ is compact, we can find a relatively
compact open neighborhood $U$ of $Z_W$. Let $U_1$ be an open subset of
$X$ such that $Z_W\subset U_1 \subset \overline{U}_1 \subset U$ and
let $\rho\in \cC^\infty_c(X)$ be a smooth function which is identically $1$ on
$U_1$ and vanishes outside $U$. The operator $\wcR$ is now defined as:
\be
\wcR\eqdef (1_X-\rho)\cR~~,
\ee
where $1_X\in \cinf$ is the unit function of $X$. Notice that
$\wcR$ is odd with respect to the total $\Z_2$-grading of
$\Hom_{\DF(X,W)}(a_1,a_2)$ and that $\wcR$ preserves the subspace
$\Hom_{\DF_c(X,W)}(a_1,a_2)$. We now define $\pi\eqdef\rho I +(\bpd
\rho) \mathcal{R}$. Relation \eqref{ucRb} implies:
\be
[\updelta,\wcR]=I-j\circ \pi~~,
\ee
showing that the first of equations \eqref{hrels} is
satisfied. Restricting this to $\Hom_{\DF_c(X,W)}(a_1,a_2)$ shows that
the second equation in \eqref{hrels} also holds, which implies that
$j_{a_1,a_2}$ is a quasi-isomorphism. Since $a_1$ and $a_2$ are
arbitrary, this gives the conclusion.  ~\qed
\end{proof}

\section{Off-shell bulk traces and the bulk flow}
\label{sec:BulkTrace}

In this section, we describe the off-shell models for the bulk trace
which result from the path integral analysis of \cite{LG2}.  Let
$(X,W)$ be a Landau-Ginzburg pair with $\dim_\C X=d$ and $\Omega$ be a
holomorphic volume form on $X$.

\subsection{The Serre trace induced by $\Omega$ on $\cA_c(X)$}

\

\

\begin{Definition}
The {\em Serre trace} induced by $\Omega$ on $\cA_c(X)$ is the
$\C$-linear map $\int_\Omega:\cA_c(X)\rightarrow \C$ defined through:
\ben
\label{intOmega}
\int_\Omega\rho\eqdef \int_X \Omega\wedge \rho~~,~~\forall \rho\in \cA_c(X)~~.
\een
\end{Definition}

\

\noindent Let $\C[0]$ denote the complex of vector spaces over $\C$ 
whose only non-trivial term equals $\C$ and sits in degree zero.

\

\begin{Proposition}
The Serre trace $\int_\Omega$ is a map of complexes of degree $-d$
from $(\cA_c(X),\bpd)$ to $\C[0]$.
\end{Proposition}

\begin{proof} 
Since $\Omega$ has type $(d,0)$, we have $\int_\Omega\rho=0$ unless
$\rho\in \cA^d_c(X)$, so $\int_\Omega$ has degree $-d$. Since $\Omega$
is holomorphic, we have:
\be
\int_\Omega(\bpd\rho)=\int_X \Omega\wedge \bpd\rho=(-1)^d \int_X \bpd(\Omega\wedge \rho)=(-1)^d\int_X \dd (\Omega\wedge \rho)=0~~,~~\forall \rho\in \cA_c(X)~~,
\ee
where we noticed that $\pd(\Omega\wedge \rho)=0$ for degree
reasons. This shows that $\int_\Omega$ is a map of complexes. \qed
\end{proof}

\subsection{The canonical off-shell bulk trace induced by $\Omega$ on $\PV_c(X)$}

\

\

\begin{Definition}
The {\em canonical off-shell bulk trace induced by $\Omega$ on $\PV_c(X)$}
is the $\C$-linear map $\Tr_B:=\Tr_B^\Omega:\PV_c(X)\rightarrow \C$
defined through:
\ben
\label{Tr0}
\Tr_B^\Omega(\omega)=\int_X \Omega\wedge (\Omega \lrcorner \omega)~~,~~\forall \omega\in \PV_c(X)~~.
\een
\end{Definition}

\noindent Since $\Omega$ has bidegree $(d,0)$, we have:
\ben
\label{TrRedCont}
\Tr_B^\Omega(\omega)=\int_X\Omega\wedge (\Omega\lrcorner_0\omega)=\int_\Omega (\Omega\lrcorner_0 \omega)~~,~~\forall \omega\in \PV_c(X)~~,
\een
where $\Omega\lrcorner_0$ is the reduced contraction introduced in Definition
\ref{def:RedCont}. Thus $\Tr_B^\Omega=\int_\Omega\circ (\Omega\lrcorner_0)$~. 
We also have:
\be
\Tr_B^\Omega (\omega)=0~~\mathrm{unless}~~\omega\in \PV_c^{-d,d}(X)~~.
\ee

\subsection{Cohomological bulk traces}

In the following, we simplify notation by omitting to indicate dependence of the traces on $\Omega$. 

\

\begin{Proposition}
\label{prop:TrBclosed}
For any $\eta\in \PV_c(X)$, we have:
\be
\Tr_B(\updelta_W\eta)=\Tr_B(\bbpd \eta)=\Tr_B(\ioda_W\eta)=0~~.
\ee
In particular, $\Tr_B$ descends to $\HPV_c(X,W)$. 
\end{Proposition}

\begin{proof}
For any $\eta\in \PV_c(X)$, we have:
\be
\Tr_B(\bbpd \eta)=\int_X \Omega\wedge [\Omega \lrcorner (\bbpd \eta)]=\int_X \bpd \left[\Omega\wedge (\Omega \lrcorner \eta)\right]=
\int_X \dd \left[\Omega\wedge (\Omega \lrcorner \eta)\right]=0~~,
\ee
where in the third equality the Dolbeault differential $\bpd$ can be
replaced with the de Rham differential $\dd$ because $\Omega\wedge
(\Omega \lrcorner \eta)$ either vanishes or has type $(d,\ast)$. The
last equality follows from the Stokes theorem since $\eta$ (and hence
also $\Omega\wedge (\Omega \lrcorner \eta)$) has compact support. To
prove the relation $\Tr_B(\ioda_W\eta)=0$, notice that $\ioda_W \eta$
has polyvector rank at most $d-1$, which implies $\Omega\lrcorner_0
(\ioda_W\eta)=0$ and hence $\Tr_B(\ioda_W\eta)=0$ by
\eqref{TrRedCont}. Finally, the relation $\Tr_B(\updelta_W\eta)=0$
follows by adding the two relations established above. ~\qed
\end{proof}

\

\begin{Definition}
The {\em cohomological bulk trace induced by $\Omega$ on $\HPV_c(X,W)$} is
the $\C$-linear map $\Tr_c:=\Tr_c^\Omega:\HPV_c(X,W)\rightarrow \C$
induced by $\Tr_B^\Omega$. 
\end{Definition}

\

\noindent When the critical set of $W$ is compact, Proposition
\ref{prop:i} implies that the inclusion map $i:\PV_c(X)\rightarrow
\PV(X)$ induces an isomorphism of $\Z$-graded $\O(X)$-modules
$i_\ast:\HPV_c(X,W)\stackrel{\sim}{\rightarrow} \HPV(X,W)$ (moreover,
$\HPV(X,W)$ is finite-dimensional over $\C$ by Proposition
\ref{prop:bulkfd}). This allows us to transfer $\Tr_c$ to a trace on
$\HPV(X,W)$:

\

\begin{Definition}
Assume that the critical set $Z_W$ is compact. In this case, the {\em
  cohomological bulk trace induced by $\Omega$ on $\HPV(X,W)$} is the
$\C$-linear map $\Tr:=\Tr^\Omega\eqdef \Tr_c^\Omega\circ
i_\ast^{-1}:\HPV(X,W)\rightarrow \C$ obtained by composing $\Tr_c$
with the inverse of the linear isomorphism
$i_\ast:\HPV_c(X,W)\stackrel{\sim}{\rightarrow} \HPV(X,W)$ induced on
cohomology by the inclusion map.
\end{Definition}

\subsection{The bulk flow determined by a K\"{a}hler metric on $\PV(X)$}
\label{subsec:bulkflow} 

Let $G$ be a K\"{a}hler metric on $X$ and $\nabla:\rGamma_\sm(X,\cT_\R
X)\rightarrow \rOmega^1(X,\cT_\R X)$ be its Levi-Civita connection. Let
$G^c$ denote the complexification of $G$ (which is a complex-valued
and fiberwise non-degenerate pairing on $\cT X$) and let $h_G$ denote
the Hermitian metric induced by $G$ on $\cT X$. We have
$G^c|_{TX\otimes TX}=G^c|_{{\overline{T}X}\otimes {\overline{T}}X}=0$ while $h_G$
restricts to Hermitian metrics on $TX$ and ${\overline{T}}X$.  We
also have $h_G(s_1,s_2)=G^c(s_1,\overline{s_2})$ for all $s_1,s_2\in
\rGamma_\sm(X,\cT X)$.  Identifying $\cT_\R X$ with $TX$ through the
isomorphism of real vector bundles $\Re:T X\rightarrow \cT_\R X$, the
Levi-Civita connection $\nabla$ can be viewed as a connection on
$TX$. Since $G$ is K\"{a}hler, $\nabla$ is $\C$-linear as a connection
on $TX$ and coincides (see \cite[Theorem 3.13]{Voisin}) with the Chern
connection of the Hermitian vector bundle $(TX,h_G)$. Hence $\nabla$
can be viewed as the unique Hermitian connection on $(TX,h_G)$ which
satisfies:
\ben
\label{nablaChern}
\nabla^{0,1}=\bbpd_{TX}|_{\rGamma_\sm(X,TX)}~~.
\een
Let $\overline{W}:X\rightarrow \C$ denote the complex conjugate of the
function $W$. Let $\grad_G \overline{W}\in \rGamma_\sm(X,T
X)=\PV^{-1,0}(X)$ denote the gradient of $\overline{W}$ with respect to
$G^c$, which is defined through the relation:
\ben
\label{grad}
(\bpd \,\overline{W}) (\bar{s})=G^c(\grad_G \overline{W},\bar{s})=h_G(\grad_G \overline{W}, s)~~,~~\forall s\in \rGamma_\sm(X,T X)~~.
\een
Let:
\be
\Hess_G(\overline{W})\eqdef \nabla (\grad_G \overline{W})\in \rOmega^1(X,T X)
\ee
denote the Hessian operator of $\overline{W}$. Let: 
\ben
\label{Hdef}
H_G\eqdef \Hess^{0,1}_G(\overline{W})=\nabla^{0,1}(\grad_G \overline{W})=\bbpd_{TX} (\grad_G \overline{W})\in \cA^1(X,TX)=\PV^{-1,1}(X)
\een
denote the $(0,1)$-part of the Hessian operator. Let ${\hat h}_G$ denote
the Hermitian metric induced by $h_G$ on the bundle $T^\ast X$ and:
\ben
\label{normpdW}
\!\!\!\!\! ||\pd W||_G^2\!\eqdef \!{\hat h}_G(\pd W,\pd W)\!=\!h_G(\grad_G \overline{W}\!,\grad_G \overline{W})\!=\!(\pd W)(\grad_G \overline{W})\in \cA^0(X)\!=\!\PV^{0,0}(X)
\een
denote the squared norm of the holomorphic 1-form $\pd W\in
\rGamma(X,T^\ast X)\subset \rOmega^{1,0}(X)$. Notice that $H_G$ is a 
nilpotent element of the algebra $\PV(X)$, which allows us to define its 
exponential. For any $\lambda\in [0,+\infty)$, we have:
\be
e^{-\i \lambda H_G}=\sum_{p=0}^{d}\frac{1}{p!} (-\i\lambda)^p (H_G)^p\in \PV^0(X)~~,
\ee
where the expansion of the exponential reduces to the first $d+1$
terms because $H_G$ belongs to $\PV^{-1,1}(X)$ and hence $(H_G)^p$
belongs to $\PV^{-p,p}(X)$, while $\PV^{-p,p}(X)=0$ for $p>d$.

\

\begin{Definition}
The {\em bulk flow generator} determined by the K\"{a}hler
metric $G$ is the element:
\be
L_G\eqdef ||\pd W||_G^2+\i H_G\in \PV^{0,0}(X)\oplus \PV^{-1,1}(X)\subset \PV^0(X)~~.
\ee
\end{Definition}

\noindent Notice that $L_G$ has degree zero with respect to
the canonical $\Z$-grading of $\PV(X)$.

\

\begin{Proposition}
We have: 
\be
L_G=\updelta_W v_G~~,
\ee
where:
\be
v_{G}\eqdef \i~\grad_G \overline{W} \in \rGamma_\sm(X,TX)=\PV^{-1,0}(X)~~.
\ee
\end{Proposition}

\begin{proof} 
Follows by direct computation using $\updelta_W=\bbpd+\ioda_W$, the
definition \eqref{Hdef} and the relation:
\ben
\label{iodagrad}
\ioda_W(\grad_G \overline{W})=-\i~||\pd W||_G^2~~,
\een
which in turn follows from \eqref{normpdW}.~\qed
\end{proof}

\

\noindent Consider the polyvector-valued form:
\ben
\label{expLG}
e^{-\lambda L_G}\eqdef e^{-\lambda ||\pd W||_G^2}e^{-\i \lambda H_G}\in \PV^0(X)~~.
\een

\begin{Proposition}
\label{prop:expL}
For any $\lambda\in \R_{\geq 0}$, we have: 
\ben
\label{expL}
e^{-\lambda L_G}=1-\updelta_W S_G(\lambda)~~,
\een
where: 
\be
\label{SGdef}
S_G(\lambda)=v_G\int_0^\lambda \dd t e^{-t L_G}\in \PV^{-1}(X)~~.
\ee
In particular, we have: 
\ben
\label{expLclosed}
\updelta_W(e^{-\lambda L_G})=0~~.
\een
\end{Proposition}

\begin{proof}
For any $x\in X$, the map $\R_{\geq 0}\ni \lambda\rightarrow
e^{-\lambda L_G(x)}\in (\wedge {\overline{T}}_x^\ast
M\otimes \wedge T_xM)^0$ is smooth and gives the unique solution of
the linear equation:
\ben
\label{ODE}
\frac{\dd}{\dd\lambda} f(\lambda)=-L_G(x)f(\lambda)~~\mathrm{with}~~f:\R_{\geq 0}\rightarrow (\wedge {\overline{T}}_x^\ast
M\otimes \wedge T_xM)^0
\een
which satisfies the initial condition $f(0)=1$. Since $\updelta_W
L_G=0$, the function $\lambda\rightarrow [\updelta_W (e^{-\lambda
    L_G})](x)$ is the unique solution of \eqref{ODE} which vanishes at
$\lambda=0$; thus $\updelta_W (e^{-\lambda L_G})=0$. Since
$L_G=\updelta_W v_G$, this gives:
\be
\frac{\dd}{\dd\lambda} e^{-\lambda L_G}=-\updelta_W(v_G e^{-\lambda L_G})
\ee
which implies the conclusion. \qed 
\end{proof}

\

\noindent Proposition \ref{prop:expL} implies that the $\updelta_W$-cohomology
class of $e^{-\lambda L_G}$ coincides with the unit of the algebra
$\HPV(X,W)$. Let $\widehat{L_G}$ denote the operator of left
multiplication with the element $L_G$ in the algebra $\PV(X)$.

\

\begin{Definition}
The {\em bulk flow} determined by the K\"{a}hler metric $G$ is the
semigroup $(U_G(\lambda))_{\lambda \geq 0}$ generated by
$\widehat{L_G}$. Thus $U_G(\lambda)$ is the even $\cinf$-linear
endomorphism of \,$\PV(X)$ defined through:
\be
U_G(\lambda)(\omega)\eqdef e^{-\lambda L_G}\omega~~,~~\forall \omega\in \PV(X)~~.
\ee
\end{Definition}

\noindent Notice that $U_G(\lambda)$ has degree zero with respect to
the canonical $\Z$-grading of $\PV(X)$ and that it preserves the subspace
$\PV_c(X)$. We have $U_G(0)=\id_{\PV(X)}$ and
$U_G(\lambda_1)U_G(\lambda_2)=U_G(\lambda_1+\lambda_2)$ for all
$\lambda_1,\lambda_2\geq 0$.

\

\begin{Proposition}
For any $\lambda\in [0,+\infty)$, the endomorphism $U_G(\lambda)$ is 
homotopy equivalent with $\id_{\PV(X)}$. In particular, we have:
\ben
\label{UGclosed}
\updelta_W\circ U_G(\lambda)=U_G(\lambda)\circ \updelta_W~~.
\een
Thus $U_G(\lambda)$ preserves the subspaces $\ker(\updelta_W)$ and
$\im(\updelta_W)$ and it induces the identity endomorphism of
$\HPV(X,W)$ on the cohomology of $\updelta_W$.
\end{Proposition}

\begin{proof}
Relation \eqref{expL} implies:
\be
U_G(\lambda)=\id_{\PV(X)}-\big[\updelta_W,\widehat{S_G(\lambda)}\big]~~,
\ee
where $\widehat{S_G(\lambda)}:\PV(X)\rightarrow \PV(X)$ is the
$\cinf$-linear operator of left multiplication in the algebra $\PV(X)$
with the element $S_G(\lambda)\in \PV^{-1}(X)$ defined in
\eqref{SGdef}:
\be
\widehat{S_G(\lambda)}(\omega)\eqdef S_G(\lambda)\omega~~,~~\forall \omega\in \PV(X)~~.
\ee
Notice that $\widehat{S_G(\lambda)}$ is an operator of degree $-1$
with respect to the canonical $\Z$-grading of $\PV(X)$. The remaining
statements are now obvious.  ~\qed
\end{proof}

\subsection{Tempered traces on $\PV_c(X)$}

\

\

\begin{Definition}
For any $\lambda\geq 0$, the {\em $\lambda$-tempered trace induced by
  $G$ and $\Omega$ on $\PV_c(X)$} is the $\C$-linear map
$\Tr^{(\lambda)}:=\Tr^{(\lambda),\Omega, G}:\PV_c(X)\rightarrow \cinf$
defined through:
\be
\Tr^{(\lambda),\Omega,G}\eqdef \Tr_B^\Omega\circ U_G(\lambda)~~.
\ee
This map has degree zero with respect to
the canonical $\Z$-grading of $\PV_c(X)$.
\end{Definition}

\

\noindent We have $\Tr^{(0),\Omega,G}=\Tr_B^\Omega$ (which is independent of $G$) and:
\be
\label{Tr_Tr0}
\Tr^{(\lambda),\Omega,G}(\omega)=\Tr_B^\Omega (e^{-\lambda L_G} \omega)~~
\ee
for all $\omega\in \PV_c(X)$. In the following, we simplify notation by omitting
to indicate the dependence of the traces on $\Omega$ and $G$.

\

\begin{Proposition}
\label{prop:canbulk}
For any $\omega\in \PV_c^{i,j}(X)$, we have: 
\be
\Tr^{(\lambda)}(\omega)=0~~\mathrm{unless}~~i+j=0
\ee
and: 
\ben
\label{TrLambdaExp}
\Tr^{(\lambda)}(\omega)=\frac{(-\i \lambda)^{d-j}}{(d-j)!} \int_X \rOmega\wedge \left(\rOmega\lrcorner [(H_G)^{d-j}
  \omega]\right)e^{-\lambda
  ||\pd W||_G^2}~~,~~\mathrm{when}~~ \omega\in \PV_c^{-j,j}(X)~~.
\een
\end{Proposition}

\begin{proof}
For any $\omega\in \PV_c^{i,j}(X)$, we have: 
\be
\Omega\lrcorner (e^{-\lambda L_G} \omega)=e^{-\lambda ||\pd W||_G^2} \sum_{p=0}^{d}\frac{(-\i\lambda)^p}{p!} \Omega \lrcorner [(H_G)^p \omega]~~,
\ee
with $(H_G)^p \omega\in \PV^{i-p,j+p}(X)$ and $\Omega\lrcorner
[(H_G)^p \omega]\in \PV^{d+i-p,j+p}(X)$. Thus $\Omega\wedge
\left(\Omega\lrcorner [(H_G)^p \omega]\right)\in
\rOmega^{2d+i-p,j+p}(X)$ and $\int_X \Omega\wedge \left(\Omega\lrcorner
      [(H_G)^p \omega]\right)$ vanishes unless $2d+i-p=j+p=d$, which
      requires $p=d-j$ and $i+j=0$. Thus $\Tr^{(\lambda)}(\omega)$
      vanishes unless $i+j=0$, in which case
      \eqref{TrLambdaExp} holds.  ~\qed
\end{proof}

\

\begin{Proposition}
\label{prop:OnShellTr}
Let $\omega\in \PV_c(X)$. Then the following statements hold for any
$\lambda\geq 0$:
\begin{enumerate}[1.]
\itemsep 0.0em
\item If $\omega=\updelta_W\eta$ for some $\eta\in \PV_c(X)$, then
  $\Tr^{(\lambda)}(\omega)=0$.
\item If $\updelta_W\omega=0$, then $\Tr^{(\lambda)}(\omega)$ does not
  depend on $\lambda$ or $G$ and coincides with $\Tr_B(\omega)$:
\ben
\label{TrTr0}
\Tr^{(\lambda)}(\omega)=\Tr^{(0)}(\omega)=\Tr_B(\omega)~~.
\een
In particular, the map induced by $\Tr^{(\lambda)}(\omega)$ on
$\HPV_c(X,W)$ coincides with $\Tr_c$.
\end{enumerate}
\end{Proposition}

\begin{proof}
\

\begin{enumerate}
\itemsep 0.0em
\item Since  $\updelta_W(e^{-\lambda L_G})=0$ (see \eqref{expLclosed}), we have:
\ben
\label{Trpartial}
\Tr^{(\lambda)}(\updelta_W \eta)=\Tr_B\left[e^{-\lambda L_G} \updelta_W\eta\right]=
\Tr_B\left[\updelta_W(e^{-\lambda L_G} \eta)\right]=0~~,
\een
where in the last equality we used Proposition \ref{prop:TrBclosed}.
\item Notice that $\Tr^{(\lambda)}(\omega)$ is differentiable with
  respect to $\lambda$.  Since $L_G=\updelta_W v_G$ and $\updelta_W
  \omega=0$, we have:
\be
\frac{\dd}{\dd \lambda}\Tr^{(\lambda)}(\omega)=\frac{\dd}{\dd
  \lambda}\Tr_B(e^{-\lambda L_G}\omega)= -\Tr_B\left[e^{-\lambda L_G}
  \updelta_W (v_G)\omega\right]= -\Tr^{(\lambda)}[\updelta_W
  (v_G\omega)]~~,
\ee
which vanishes by point 1. This shows that $\Tr^{(\lambda)}(\omega)$
does not depend on $\lambda$ or $G$ and hence \eqref{TrTr0} holds.
\end{enumerate}    ~\qed
\end{proof}

\section{Off-shell boundary traces and boundary flows}
\label{sec:BoundaryTrace} 

\subsection{The canonical off-shell boundary traces induced by $\Omega$ on $\DF_c(X,W)$}

Let $a=(E,D)$ be a holomorphic factorization of $W$. Let
$\updelta_a:=\updelta_{a,a}$ and $\md_a:=\md_{a,a}$ denote the twisted
and defect differentials on $\End_{\DF(X,W)}(a)$. Let
$\bbpd_a:=\bbpd_{a,a}=\bbpd_{End(E)}$ denote the Dolbeault operator of
$End(E)$. We have:
\be
\updelta_a=\bbpd_a+\md_a~~,~~\md_a=[D,\cdot]~~,
\ee
where $[\cdot,\cdot]$ denotes the graded commutator. 

\

\begin{Definition}
The {\em canonical off-shell boundary trace} induced by $\Omega$ on
$\End_{\DF_c(X,W)}(a)$ is the $\C$-linear map
$\tr_a^B:=\tr_a^{B,\Omega}:\End_{\DF_c(X,W)}(a)\rightarrow \C$ defined
through:
\ben
\label{trdef}
\tr_a^{B,\Omega}(\alpha)=\int_X\Omega\wedge \str(\alpha)=\int_\Omega \str(\alpha)~~,
\een 
for all $\alpha\in \End_{\DF_c(X,W)}(a)=\cA_c(X,End(E))$, where $\str$
denotes the extended supertrace of Subsection \ref{subsec:strext}.
\end{Definition}

\

\noindent We have: 
\be
\tr_a^{B,\Omega}(\alpha)=0~~\mathrm{unless}~~\alpha\in \cA_c^d(X,End^\0(E))\subset \cA_c(X,End(E))^\mu~~,
\ee
where $\mu={\hat d}$ is the signature of $(X,W)$. In the following, we
simplify notation by omitting to indicate dependence of the traces on
$\Omega$.

\

\begin{Proposition}
\label{prop:cancyclic}
For any holomorphic factorizations $a_1$ and $a_2$ of $W$, we have: 
\ben
\label{trBcyc}
\tr^B_{a_2}(\alpha\beta)=(-1)^{\deg \alpha \, \deg \beta}\tr_{a_1}^B(\beta\alpha)~~,
\een
when $\alpha\in \Hom_{\DF_c(X,W)}(a_1,a_2)$ and $\beta\in
\Hom_{\DF_c(X,W)}(a_2,a_1)$ have pure total $\Z_2$-degree. 
\end{Proposition}

\begin{proof}
Follows immediately from \eqref{trdef} upon using relation
\eqref{strextcyc}.  ~\qed
\end{proof}

\

\begin{remark}
Relation \eqref{trBcyc} is an off-shell (cochain level) version
of the cyclicity condition \eqref{tcyc}.
\end{remark}

\subsection{Cohomological boundary traces}

\

\

\begin{Proposition}
\label{prop:bdrytrace}
For any $\alpha\in \End_{\DF_c(X,W)}(a)$, we have:
\be
\tr_a^B(\updelta_a\alpha)=\tr_a^B(\bbpd_a \alpha)=\tr_a^B(\md_a\alpha)=0~~.
\ee
In particular, $\tr_a^B$ descends to
$\End_{\HDF_c(X,W)}(a)=\rH(\cA_c(X,End(E)),\updelta_a)$.
\end{Proposition}

\begin{proof}
We have:
\beqa
\tr_a^B(\bbpd_a \alpha)&=&\int_X \Omega\wedge \str(\bbpd_a \alpha)=\int_X \Omega\wedge [\bpd\,\str(\alpha)]=\nn\\
&=& (-1)^d \int_X \bpd \left[\Omega\wedge \str(\alpha)\right]=
(-1)^d\int_X \dd \left[\Omega\wedge \str(\alpha)\right]=0~~,
\eeqa
where we used \eqref{strbpd} and \eqref{strfd} and the fact that
$\Omega\wedge \str(\alpha)$ has type $(d,\ast)$. The relation
$\tr_a^B(\md_a(\alpha))=0$ follows from \eqref{strfd}, while the
relation $\tr_a^B(\updelta_a \alpha)=0$ follows by combining these two
properties.
\end{proof}

\

\begin{Definition}
The {\em cohomological boundary trace induced by $\Omega$ on
  $\End_{\HDF_c(X,W)}(a)$} is the $\C$-linear map
$\tr_a^c:=\tr_a^{c,\Omega}:\End_{\HDF_c(X,W)}(a) \rightarrow \C$
induced by $\tr_a^{B,\Omega}$ on $\End_{\HDF_c(X,W)}(a)$.
\end{Definition}

\

\noindent When the critical set $Z_W$ is compact, Proposition \ref{prop:j}
implies that the inclusion induces a strictly surjective and fully-faithful
functor $j_\ast:\HDF_c(X,W)\stackrel{\sim}{\rightarrow} \HDF(X,W)$
(moreover, $\HDF(X,W)$ is Hom-finite over $\C$ by Proposition \ref{prop:boundaryfd}). 
This allows us to transport $\tr_a^c$ to a trace on $\End_{\HDF(X,W)}(a)$:

\

\begin{Definition}
Assume that the critical set $Z_W$ is compact. Then the {\em
  cohomological boundary trace induced by $\Omega$ on
  $\End_{\HDF(X,W)}(a)$} is the $\C$-linear map $\tr_a\eqdef
\tr_a^c\circ j_{\ast,a}^{-1}: \End_{\HDF(X,W)}(a) \rightarrow \C$,
where
$j_{\ast,a}:\End_{\HDF_c(X,W)}(a)\stackrel{\sim}{\rightarrow} \End_{\HDF(X,W)}(a)$
is the linear isomorphism induced by the inclusion functor.
\end{Definition}

\

\noindent Let $\tr^c\eqdef (\tr^c_a)_{a\in \Ob\HDF_c(X,W)}$ and
$\tr\eqdef (\tr_a)_{a\in \Ob\HDF(X,W)}$. When the critical set $Z_W$
is compact, Proposition \ref{prop:cancyclic} and the remarks above
imply that $(\HDF_c(X,W),\tr^c)$ and $(\HDF(X,W),\tr)$ are equivalent
pre-Calabi-Yau supercategories.

\subsection{Hermitian holomorphic factorizations}
\label{subsec:hhf} 

Let $E=E^\0\oplus E^\1$ be a holomorphic vector superbundle on $X$. 

\

\begin{Definition}
A Hermitian metric $h$ on $E$ is called {\em admissible} if the
sub-bundles $E^\0$ and $E^\1$ of $E$ are $h$-orthogonal:
\be
h|_{E^\0\times E^\1}=h|_{E^\1\times E^\0}=0~~.
\ee
\end{Definition}

\

\begin{Definition}
A {\em Hermitian holomorphic factorization} of $W$ is a triplet
$\ba=(E,h,D)$, where $a=(E,D)$ is a holomorphic factorization of $W$
and $h$ is an admissible Hermitian metric on $E$.
\end{Definition}

\subsection{The boundary flow of a Hermitian holomorphic factorization induced by a K\"{a}hler metric}
\label{subsec:boundary flow}

\

\

\noindent Let us fix a Hermitian holomorphic factorization $\ba=(E,h,
D)$ of $W$ with underlying holomorphic factorization $a=(E,D)$. Let
$\nabla:=\nabla_\ba$ denote the Chern connection of the Hermitian
holomorphic vector bundle $(E,h)$. This is the unique $h$-compatible
$\C$-linear connection on $E$ which satisfies:
\be
\nabla_\ba^{0,1}=\bbpd_E|_{\rGamma_\sm(X,E)}~~,
\ee
where $\bbpd_E$ is the Dolbeault operator of $E$. Let
$\rOmega(X,E)\eqdef \rOmega(X)\otimes_\cinf \rGamma_\sm(X,E)$ and
$\mpd_E^h:\rOmega(X,E)\rightarrow \rOmega(X,E)$ be the unique
$\C$-linear operator which satisfies the Leibniz rule:
\ben
\label{mpd_E^h}
\mpd_E^h(\rho\otimes s)=(\pd \rho)\otimes s+(-1)^k\rho\wedge \nabla_\ba^{1,0}(s)~~
\een
for all $\rho\in \rOmega^k(X)$ and all $s\in \rGamma_\sm(X,E)$. We have 
$\mpd_E^h(\rOmega^{i,j}(X,E))\subset \rOmega^{i+1,j}(X,E)$ for all $i,j$
and:
\be
\mpd_E^h|_{\rGamma_\sm(X,E)}=\nabla^{1,0}_\ba~~.
\ee
Let $F_\ba\in \rOmega^{1,1}(X,End^\0(E))=\cA^1(X,T^\ast X\otimes
End^\0(E))$ denote the curvature of $\nabla_\ba$. Since $F_\ba$ has
type $(1,1)$, we have $(\nabla_\ba^{1,0})^2=(\nabla_\ba^{0,1})^2=0$
and
$\nabla^{1,0}_\ba\nabla^{0,1}_\ba+\nabla^{0,1}_\ba\nabla^{1,0}_\ba=F_\ba$. These
relations imply:
\be
(\mpd_E^h)^2=\bbpd_E^2=0~~,~~\mpd_E^h\bbpd_E+\bbpd_E\mpd_E^h=\widehat{F_\ba}~~,
\ee
where $\widehat{F_\ba}$ denotes the operator of left multiplication in
the associative algebra $\rOmega(X, End(E))$. Let
$\mpd_\ba:\rOmega(X,End(E))\rightarrow \rOmega(X,End(E))$ denote the
differential induced by $\mpd_E^h$ on the graded associative algebra 
$\rOmega(X,End(E))$. We have:
\ben
\label{mpdrels}
\mpd_\ba^2=0~~,~~\mpd_\ba\bbpd_a+\bbpd_a\mpd_\ba=[F,\cdot]~~,
\een
where $\bbpd_a=\bbpd_{End(E)}$. 

\

\begin{Definition}
The {\em boundary flow generator of $\ba=(E,h, D)$ determined by the
K\"{a}hler metric $G$} is defined through:
\be
L_\ba^{G}\eqdef ||\pd W||_G^2 \id_E + H_G \lrcorner (\mpd_\ba D+F)\in \cA^0(X,End^\0(E))\oplus \cA^1(X,End^\1(E))\oplus \cA^2(X,End^\0(E))~~.
\ee
\end{Definition}

\

\begin{Proposition}
We have: 
\be
L_\ba^G=\updelta_a v_\ba^G~~,
\ee
where: 
\be
v_\ba^{G}\eqdef \grad_G \overline{W}\lrcorner (\mpd_\ba D+F)\in \cA^0(X,End^\1(E))\oplus \cA^1(X,End^\0(E))~~.
\ee
\end{Proposition}

\begin{proof}
The decomposition $\updelta_a=\bbpd_a+\md_a$ gives: 
\be
\updelta_a(v_\ba^G)=\bbpd_a(v_\ba^G)+\fd_a(v_\ba^G)~~.
\ee
Since $H_G=\bbpd_{T X}(\grad_G\overline{W})$, we have: 
\ben
\label{e1}
\bbpd_a(v_\ba^G)=H_G\lrcorner (\mpd_\ba D+F)- \grad_G \overline{W}\lrcorner [F,D]~~,
\een
where we used the second relation in \eqref{mpdrels} and the fact that
$\bbpd_a D=\bbpd_a F=0$.  On the other hand, we have:
\ben
\label{e2}
\fd_a(v_\ba^G)=-\grad_G \overline{W}\lrcorner \big([D,\mpd_\ba D]+[D,F]\big)= ||\pd W||_G^2\id_E +\grad_G \overline{W}\lrcorner [F,D]~~,
\een
where we used the relation $[D,\mpd_\ba D]=D\mpd_\ba D-(\mpd_\ba
D)D=-(\pd W)\id_E$ (which follows by applying $\mpd_a$ to the identity
$D^2=W\id_E$) and the relation $(\grad_G \overline{W}) \lrcorner \pd W=(\pd
W)(\grad_G \overline{W})=||\pd W||_G^2$ (see \eqref{normpdW}). Adding
\eqref{e1} and \eqref{e2} gives the conclusion.~\qed
\end{proof}

\

\noindent Notice that $L_\ba^G\in \End^\0_{\DF(X,W)}(a) $ and $v_\ba^G\in
\End^\1_{\DF(X,W)}(a)$. Also notice that $H_G \lrcorner (\mpd_\ba D+F)$ is
nilpotent, which allows us to define its exponential. For any
$\lambda\geq 0$, we have:
\be
e^{-\lambda H_G \lrcorner (\mpd_\ba D+F)}=\sum_{k=0}^d \frac{(-\lambda)^k}{k!} \big[H_G \lrcorner (\mpd_\ba D+F)\big]^k\in \End^\0_{\DF(X,W)}(a)~~,
\ee
where the series reduces to the first $d+1$ terms because $H_G
\lrcorner (\mpd_\ba D+F)$ belongs to $\cA^1(X,End^\1(E))\oplus
\cA^2(X,End^\0(E))$ and hence $\left[H_G \lrcorner (\mpd_\ba D+F)\right]^k$
vanishes for $k>d$. Define: 
\be
e^{-\lambda L_\ba^G}\eqdef e^{-\lambda ||\pd W||_G^2}e^{-\lambda H_G \lrcorner (\mpd_\ba D+F)} \in \End^\0_{\DF(X,W)}(a)~~.
\ee

\begin{Proposition}
\label{prop:expLa}
For any $\lambda \geq 0$, we have:
\ben
\label{expLa}
e^{-\lambda L_\ba^G}=1-\updelta_a S_\ba^G(\lambda)~~, 
\een
where:
\ben
\label{Sadef}
S_\ba^G(\lambda)\eqdef v_\ba^G \int_0^\lambda \dd t e^{-t L_\ba^G}\in \End^\1_{\DF(X,W)}(a)~~.
\een
In particular, we have: 
\ben
\label{Laclosed}
\updelta_a(e^{-\lambda L_\ba^G})=0~~.
\een
\end{Proposition}

\begin{proof}
The proof is almost identical to that of Proposition \ref{prop:expL}.\qed
\end{proof}

\

\noindent
Let $\widehat{L_\ba^G}$ be the operator of left multiplication
with the element $L_\ba^G$ in the algebra
$\End_{\DF(X,W)}(a)\!=\!\cA(X,End(E))$.

\

\begin{Definition}
The {\em boundary flow of $\ba=(E,h,D)$} determined by the K\"{a}hler metric
  $G$ is the semigroup $(U_\ba^G(\lambda))_{\lambda \geq 0}$
generated by $\widehat{L_\ba^G}$. Thus $U_\ba^G(\lambda)$ is the even
$\cinf$-linear endomorphism of $\End_{\DF(X,W)}(a)$ defined through:
\be
U_\ba^G(\lambda)(\alpha)\eqdef e^{-\lambda L_\ba^G}\alpha~~,~~\forall \alpha\in \End_{\DF(X,W)}(a)~~.
\ee
\end{Definition}

\noindent Notice that $U_\ba^G(\lambda)$ is even with respect to the
canonical $\Z_2$-grading and that it preserves the subspace
$\End_{\DF_c(X,W)}(a)\!=\!\cA_c(X,End(E))$. We have
$U_\ba^G(0)\!=\!\id_{\End_{\DF(X,W)}(a)}$ and
$U_\ba^G(\lambda_1)U_\ba^G(\lambda_2)\!=\!U_\ba^G(\lambda_1+\lambda_2)$
for all $\lambda_1,\lambda_2\geq 0$~.

\

\begin{Proposition}
For any $\lambda\geq 0$, the endomorphism $U_\ba^G(\lambda)$ is
homotopy equivalent with $\id_{\End_{\DF(X,W)}(a)}$. In particular, we
have:
\ben
\label{Uaclosed}
\updelta_a\circ U_\ba^G(\lambda)=U_\ba^G(\lambda)\circ \updelta_a~~.
\een
Hence $U_\ba^G(\lambda)$ preserves the subspaces
  $\ker(\updelta_a)$ and $\im(\updelta_a)$ and it induces the identity
  endomorphism of $\End_{\DF(X,W)}(a)$ on the cohomology of
$\updelta_a$.
\end{Proposition}

\begin{proof}
Relation \eqref{expLa} implies:
\be
U_\ba^G(\lambda)=\id_{\End_{\DF(X,W)}(a)} - \big[\updelta_a,\widehat{S_\ba^G(\lambda)}\big]~~,
\ee
where $\widehat{S_\ba^G(\lambda)}$ is the $\cinf$-linear operator of
left multiplication with the element $S_\ba^G(\lambda)$ defined in
\eqref{Sadef} in the algebra $\End_{\DF(X,W)}(a)$, which is odd with
respect to the canonical $\Z_2$-grading. This implies all statements
in the proposition.  ~\qed
\end{proof}

\subsection{Tempered traces on $\DF_c(X,W)$}

\

\

\begin{Definition}
Let $\lambda\in \R_{\geq 0}$. The {\em $\lambda$-tempered trace} of
  $\ba=(E,h, D)$ induced by $\Omega$ and $G$ is the $\C$-linear map
$\tr_\ba^{(\lambda)}:=\tr_\ba^{(\lambda),\Omega, G}:\End_{\HDF_c(X,W)}(a)\rightarrow
\C$ defined through:
\be
\tr_\ba^{(\lambda),\Omega, G}\eqdef \tr_a^{B,\Omega}\circ U_\ba^{G}(\lambda)~~.
\ee 
\end{Definition}

\

\noindent We have $\tr_\ba^{(0), \Omega, G}=\tr_a^{B,\Omega}$ (which
is independent of $G$ and $h$) and:
\be
\tr_\ba^{(\lambda), \Omega,G}(\alpha)=\tr_a^{B,\Omega}(e^{-\lambda L_\ba^G} \alpha)~~,~~\forall \alpha\in \End_{\DF_c(X,W)}(a)~~.
\ee
In the following, we simplify notation by omitting to indicate the dependence of the traces on $\Omega$ and $G$. 

\

\begin{Proposition}
\label{prop:onshellbdrytr}
Let $\alpha\in \End_{\DF_c(X,W)}(a)$. Then the following
statements hold for any $\lambda\geq 0$:
\begin{enumerate}[1.]
\itemsep 0.0em
\item If $\alpha=\updelta_a \beta$ for some $\beta\in \End_{\DF_c(X,W)}(a)$,
  then $\tr_\ba^{(\lambda)}(\alpha)=0$.
\item If $\updelta_a \alpha=0$, then $\tr_\ba^{(\lambda)}(\alpha)$ does
  not depend on $\lambda$ or on the metrics $G$ and $h$:
\ben
\label{tr_tr0}
\tr_\ba^{(\lambda)}(\alpha)=\tr_\ba^{(0)}(\alpha)=\tr^B_a(\alpha)~~.
\een
\end{enumerate}
In particular, the map induced by $\tr_\ba^{(\lambda)}$ on
$\End_{\HDF_c(X,W)}(a)$ coincides with $\tr_a^c$.
\end{Proposition}

\begin{proof}

\

\begin{enumerate}
\itemsep 0.0em
\item Since $e^{-\lambda L_\ba^G}$ is even and $\updelta_a$-closed, we have: 
\ben
\label{trpartial}
\tr_\ba^{(\lambda)}(\updelta_a \beta)=\tr_a^B\left[e^{-\lambda L_\ba^G} 
\updelta_a \beta \right]=\tr_a^B\left[\updelta_a (e^{-\lambda L_\ba^G} \beta) \right]=0~~,
\een
where in the last equality we used Proposition \ref{prop:bdrytrace}.
\item Notice that $\tr_a^{(\lambda)}$ is differentiable with respect
  to $\lambda$.  Since $L_\ba^G=\updelta_a v_\ba^G$ while $\alpha$ is
  $\delta_a$-closed, we have:
\be
\frac{\dd}{\dd\lambda} \tr_\ba^{(\lambda)}(\alpha)=-\tr_\ba^{(\lambda)}\big[\updelta_a(v_\ba^G)\alpha\big]=-\tr_\ba^{(\lambda)} \big[\updelta_a(v_\ba^G\alpha)\big]~~,
\ee
which vanishes by point 1. This implies \eqref{tr_tr0}. 
\end{enumerate}  
~\qed
\end{proof}

\

\begin{Proposition}
\label{prop:tr_cyc}
Let $\ba_1=(E_1,h_1,D_1)$ and $\ba_2=(E_2,h_2, D_2)$ be two Hermitian
holomorphic factorizations of $W$ with underlying holomorphic
factorizations $a_1=(E_1,D_1)$ and $a_2=(E_2,D_2)$.  Let $\alpha\in
\Hom_{\DF_c(X,W)}(a_1,a_2)$ and $\beta\in \Hom_{\DF_c(X,W)}(a_2,a_1)$
have pure total $\Z_2$-degree and satisfy
$\updelta_{a_1,a_2}\alpha=\updelta_{a_2,a_1}\beta=0$. Then:
\be
\tr^{(\lambda),G}_{\ba_2}(\alpha\beta)=(-1)^{\deg \alpha \, \deg \beta}\tr_{\ba_1}^{(\lambda), G}(\beta\alpha)~~.
\ee
\end{Proposition}

\begin{proof}
Since $\updelta_{a_1,a_2}\alpha=\updelta_{a_2,a_1}\beta=0$, we have
$\updelta_{a_2}(\alpha\beta)=0$ and $\updelta_{a_1}(\beta\alpha)=0$.
Thus $\tr^{(\lambda),G}_{\ba_2}(\alpha\beta)=\tr^B_{a_2}(\alpha\beta)$
and $\tr^{(\lambda),G}_{\ba_1}(\beta\alpha)=\tr^B_{a_1}(\beta\alpha)$
by Proposition \ref{prop:onshellbdrytr}. This implies the conclusion
upon using Proposition \ref{prop:cancyclic}.  ~\qed
\end{proof}

\section{The disk algebra}
\label{sec:disk}

\subsection{The dg-algebra $\PV(X,End(E))$}

\noindent Let $a=(E,D)$ be a holomorphic factorization of $W$ and set
$\bbpd_a:=\bbpd_{a,a}=\bbpd_{End(E)}$, $\fd_a:=\fd_{a,a}=[D,\cdot]$ and
$\updelta_a:=\updelta_{a,a}=\bbpd_a+\fd_a$. Consider the
$\Z_2$-graded unital associative $\cinf$-algebra:
\be
\PV(X,End(E))\eqdef \cA(X,\wedge TX \otimes End(E))\simeq \PV(X){\hat \otimes}_\cinf\rGamma_\sm(X,End(E))~~,
\ee
where ${\hat \otimes}_\cinf$ denotes the graded tensor product and
$\PV(X)$ in the right hand side is endowed with the canonical
$\Z_2$-grading. The isomorphism of $\Z_2$-graded $\cinf$-algebras
$\cA(X,End(E))\simeq \cA(X){\hat \otimes}_\cinf
\rGamma_\sm(X,End(E))$ implies:
\be
\PV(X,End(E))\simeq \PV(X){\hat \otimes}_{\cA(X)} \cA(X,End(E))=\PV(X){\hat \otimes}_{\cA(X)} \End_{\DF(X,W)}(a)~~.
\ee
The unit of $\PV(X,End(E))$ is the identity endomorphism $\id_E$ of
$E$.

\

\begin{Definition}
The {\em total twisted differential} $\Delta_a$ on $\PV(X,End(E))$ is the 
odd $\O(X)$-linear differential:
\be
\Delta_a\eqdef \updelta_W{\hat \otimes}_{\cA(X)} \id_{\cA(X,End(E))}+\id_{\PV(X)}{\hat \otimes}_{\cA(X)}\updelta_a
\ee
induced on $\PV(X,End(E))$ by the differentials $\updelta_W$ of
$\PV(X)$ and $\updelta_a$ of $\End_{\DF(X,W)}(a)$.
\end{Definition}

\

\noindent It is easy to see that $\Delta_a$ is an odd
derivation of $\PV(X,End(E))$ which squares to zero. To simplify
notation, we write $\updelta_W$ instead of $\updelta_W{\hat
  \otimes}_{\cA(X)} \id_{\cA(X,End(E))}$ and $\updelta_a$ instead of
$\id_{\PV(X)}{\hat \otimes}_{\cA(X)}\updelta_a$. Then
$(\PV(X,End(E)),\updelta_W,\updelta_a)$ is a $\Z\times \Z_2$-graded
bicomplex when endowed with the $\Z$-grading induced by the canonical
$\Z$-grading of $\PV(X)$ and with the $\Z_2$-grading induced by the
bundle grading of $E$; the $\Z_2$-grading of $\PV(X,End(E))$ is the
total $\Z_2$-grading induced by this bigrading. Moreover, $\Delta_a$ 
is the total differential of this bicomplex. 

\

\begin{Definition}
The {\em off-shell disk algebra of the holomorphic factorization
  $a=(E,D)$} is the $\O(X)$-linear $\Z_2$-graded unital dg-algebra
$(\PV(X,End(E)),\Delta_a)$. The {\em cohomological disk algebra of
  $a$} is the $\Z_2$-graded $\O(X)$-linear algebra $\HPV(X,a)$ defined
as the total cohomology algebra of the off-shell disk algebra:
\be
\HPV(X,a)\eqdef \rH(\PV(X,End(E)), \Delta_a)~~.
\ee
\end{Definition}

\begin{remark}
Let $\bbpd_{\wedge TX\otimes End(E)}:\PV(X,End(E))\rightarrow
\PV(X,End(E))$ be the Dolbeault operator of the holomorphic vector
bundle $\wedge TX\otimes End(E)$. Then:
\ben
\label{bbpd}
\bbpd_{\wedge TX\otimes End(E)}=\bbpd_{\wedge TX}{\hat \otimes}_{\cA(X)} \id_{\cA(X,End(E))}+\id_{\PV(X)} {\hat \otimes}_{\cA(X)} \bbpd_{End(E)}
\een
and hence:
\ben
\label{Delta}
\Delta_a=\bbpd_{\wedge TX\otimes End(E)}+\ioda_W{\hat \otimes}_{\cA(X)}\id_{\cA(X,End(E))}+\id_{\PV(X)}{\hat \otimes}_{\cA(X)}\fd_a~~.
\een
For simplicity, we denote $\bbpd_{\wedge TX}{\hat \otimes}_{\cA(X)}
\id_{\cA(X,End(E))}$ by $\bbpd$ and $\id_{\PV(X)} {\hat
  \otimes}_{\cA(X)} \bbpd_{End(E)}$ by $\bbpd_a$. Similarly, we denote $\ioda_W{\hat
  \otimes}_{\cA(X)}\id_{\cA(X,End(E))}$ and $\id_{\PV(X)}{\hat
  \otimes}_{\cA(X)}\fd_a$ by $\ioda_W$ and $\fd_a$. With these
notations, combining \eqref{bbpd} and \eqref{Delta} gives a
decomposition of $\Delta_a$ as a sum of four odd differentials which
mutually anticommute:
\be
\Delta_a=\bbpd+\bbpd_a+\ioda_W+\fd_a~~.
\ee
\end{remark}

\subsection{The disk extended supertrace}

Notice that $\PV(X,End(E))$ is a left module over the superalgebra
$\PV(X)$ (when the latter is endowed with the canonical
$\Z_2$-grading) as well as a right module over the superalgebra
$\rGamma_\sm(X,End(E))$.  We extend the fiberwise supertrace
$\str:\rGamma_\sm(X,End(E))\rightarrow \cinf$ to a left
$\PV(X)$-linear map denoted by the same symbol and called the 
{\em disk extended supertrace}:
\be
\str\eqdef \id_{\PV(X)} \otimes_\cinf \str:\PV(X,End(E))\rightarrow \PV(X)~~.
\ee
For any $\omega\in \PV(X)$ and any $s\in \rGamma_\sm(X,End(E))$, we have:
\be
\str(\omega\otimes_\cinf s)=\omega \, \str(s)~~
\ee
and the disk extended supertrace satisfies: 
\be
\str(\fz_1\fz_2)=(-1)^{\deg \fz_1\, \deg \fz_2}\str(\fz_2\fz_1)~~,~~\forall \fz_1,\fz_2\in \PV(X,End(E))~~.
\ee
The disk extended supertrace is an even map of $\Z_2$-graded complexes from
$(\PV(X,End(E)),\Delta_a)$ to $(\PV(X),\updelta_W)$, when the latter is 
endowed with the canonical $\Z_2$-grading:
\ben
\label{strDelta}
\str(\Delta_a\fz)=\updelta_W \str(\fz)~~,~~\forall \fz\in \PV(X,End(E))~~
\een
and, with the simplified notations introduced above, it satisfies: 
\ben
\label{diskstrprops}
\str(\ioda_W \fz)=\ioda_W\str(\fz)~~\mathrm{and}~~\str(\fd_a\fz)=0~~,~~\forall \fz\in \PV(X,End(E))~~.
\een

\subsection{The extended reduced contraction induced by a holomorphic volume form}

Given a holomorphic volume form $\Omega$ on $X$, we extend the reduced
contraction of Definition \ref{def:RedCont} to a right
$\rGamma_\sm(X,End(E))$-linear map of $\Z_2$-degree $\mu$:
\be
\Omega \lrcorner_0\eqdef (\Omega \lrcorner_0) \otimes_\cinf \id_{\rGamma_\sm(X,End(E))}:\PV(X,End(E))\rightarrow \End_{\DF(X,W)}(a)=\cA(X,End(E))~~.
\ee
For any $\omega\in \PV(X)$ and $s\in \rGamma_\sm(X,End(E))$, we have:
\be
\Omega\lrcorner_0(\omega\otimes_\cinf s)=(\Omega\lrcorner_0\omega)\otimes_\cinf s~~.\nn\\
\ee
With these definitions, the following relation is satisfied (see the
commutative diagram \eqref{diagram:strOmega}): 
\ben
\label{strlrcorner}
\str (\Omega\blrcorner_0\fz)=\Omega\lrcorner_0\str(\fz)~~,~~\forall \fz\in \PV(X,End(E))~~.
\een

\begin{Lemma}
\label{lemma:DeltaContr}
For any $\fz\in \PV(X,End(E))$, we have: 
\be
\updelta_a(\Omega\lrcorner_0 \fz)=(-1)^d \Omega\lrcorner_0 (\Delta_a\fz)~~.
\ee
Thus $\Omega\lrcorner_0$ is a map of complexes of $\Z_2$-degree
$\mu\!=\!{\hat d}$ from $(\PV(\!X\!,\!End(\!E)\!),\!\Delta_a\!)$ to
$(\End_{\DF(\!X,W\!)}\!(a\!),\!\updelta_a\!)$.
\end{Lemma}

\begin{proof}
It is enough to prove the restriction of this relation to an open
subset $U\subset X$ supporting complex coordinates $z^1,\ldots, z^d$.
Write $\Omega=\varphi(z)\dd z^1\wedge\ldots \wedge \dd z^d$, where
$\varphi\in \cO_X(U)$. Expand $\fz=_U\sum_{k=0}^d \sum_{1\leq
  i_1<i_2<\ldots <i_k\leq d}\theta^{i_k}\ldots
\theta^{i_1}\fz_{i_1\ldots i_k}$, with coefficients $\fz_{i_1\ldots
  i_k}\in \cA(U,End(E))$. Then $\Omega\lrcorner_0
\fz=\varphi(z)\fz_{1\ldots d}$. Thus:
\ben
\label{rel1}
\updelta_a(\Omega\lrcorner_0 \fz)=_U \varphi(z)\updelta_a \fz_{1\ldots d}~~,
\een
where we used the relations $\dd z^i\lrcorner \theta_j=\dd
z_i\lrcorner \pd_j=\delta_{ij}$. On the other hand, we have:
\be
\!\!\!\!\Delta_a\fz\!=_U\!\!\sum_{k=0}^d\,\sum_{1\!\leq i_1\!<\ldots <\!i_k\leq\!  d}
\!\left[\!(-1)^k\theta^{i_k}\!\!\!\ldots \theta^{i_1} \!(\updelta_a \fz_{i_1\ldots i_k})
-\!\i\! \sum_{s=1}^k\! (-1)^{k-s} (\pd_{i_s} W) \theta_{i_k}\! \ldots \theta_{i_{s+1}}\!\theta_{i_{s-1}}\!\ldots \theta_{i_1} \fz_{i_1\ldots i_k}\!\right]~~
\ee
where we noticed that $\Delta_a(\theta_i)=\ioda_W(\theta_i)=-\i \pd_i
W$. The reduced contraction of $\Omega$ with every term in this sum vanishes
for degree reasons with the single exception of the contraction with
the term $(-1)^d\theta^d\ldots \theta^1 (\updelta_a \fz_{1\ldots d})$,
which equals $(-1)^d\varphi(z)\updelta_a\fz_{1\ldots d}$. Thus:
\ben
\label{rel2}
\Omega\lrcorner _0(\Delta_a\fz)=(-1)^d\varphi(z) \updelta_a\fz_{1\ldots d}~~.
\een
Comparing \eqref{rel1} and \eqref{rel2} gives the conclusion. \qed
\end{proof}

\

\noindent Consider the degree $\mu$ map of $\Z_2$-graded complexes
$\lambda_\Omega:(\PV(X,End(E)),\Delta_a)\rightarrow (\cA(X),\bpd)$
defined through:
\ben
\label{lambdaDef}
\lambda_\Omega\eqdef (\Omega\lrcorner_0)\circ\str=\str\circ (\Omega\lrcorner_0)~~,
\een
where we used relation \eqref{strlrcorner}. Combining everything, we
have the following commutative diagram of $\Z_2$-graded complexes,
where the $\Z_2$-degree of each map is indicated in square brackets:
\ben
\label{diagram:strOmega}
\scalebox{0.9}{
\xymatrix{
& &\ar[ddl]^{[0]}_{\str}(\PV(X,End(E)),\Delta_a)\ar[dddd]^{[\mu]}_{\lambda_\Omega}\ar[ddr]_{[\mu]}^{\Omega\lrcorner_0} & &\\
& & & &\\
& (\PV(X),\updelta_W)\ar[ddr]^{[\mu]}_{\Omega\lrcorner_0}& &\ar[ddl]_{[0]}^{\str}(\End_{\DF(X,W)}(a),\updelta_a) &\\
& & & &\\
& & (\cA(X),\bpd)& &
}}
\een

\subsection{The extended disk algebra}

The $\cinf$-module $\PV_e(X,End(E))\eqdef \rOmega(X,\wedge TX\otimes End(E))$ carries a
natural structure of $\cinf$-superalgebra which makes it isomorphic
with the graded tensor product:
\be
\PV_e(X,End(E))\simeq \rOmega^{\ast,0}(X){\hat \otimes}_\cinf \PV(X,End(E))~~,
\ee
where $\rOmega^{\ast,0}(X)=\rGamma_\sm(X,\wedge T^\ast
X)=\bigoplus_{k=0}^d \rOmega^{k,0}(X)$. We trivially extend $\Delta_a$
to an odd $\O(X)$-linear differential on this algebra which we denote by the same
symbol:
\be
\Delta_a\eqdef \id_{\rOmega^{\ast,0}(X)}{\hat \otimes}_\cinf \Delta_a~~.
\ee
For any $\rho\in \rOmega^{i,0}(X)$ and any $\fz\in \PV(X,End(E))$, we have: 
\be
\Delta_a(\rho\otimes \fz)=(-1)^i\rho\otimes (\Delta_a\fz)~~.
\ee

\begin{Definition}
The {\em extended disk algebra} is the unital and $\Z_2$-graded
$\O(X)$-linear dg-algebra $(\PV_e(X,End(E)),\Delta_a)$. The {\em
  cohomological extended disk algebra} is the unital $\Z_2$-graded
$\O(X)$-linear algebra $\HPV_e(X,a)$ defined as the total cohomology
algebra of the extended disk algebra:
\be
\HPV_e(X,a)\eqdef \rH(\PV_e(X,End(E)),\Delta_a)~~.
\ee
\end{Definition}

\

\noindent Let $\Omega$ be a holomorphic
volume form on $X$. For any element $s\in \rOmega^{1,\ast}(X,\wedge
TX\otimes End(E))$, let $s^d\in \rOmega^{d,\ast}(X,\wedge TX\otimes
End(E))$ denote the $d$-th power of $s$ computed in this
superalgebra. Since $\Omega$ is a nowhere-vanishing section of
$\wedge^d T^\ast X$, there exists a unique element ${\det}_{\Omega} s\in
\PV(X, End(E))$ such that:
\be
s^d=(-1)^{\frac{d(d-1)}{2}}\Omega\otimes ({\det}_{\Omega} s)~~.
\ee
This gives a map $\det_\Omega:\rOmega^{1,\ast}(X,\wedge TX\otimes
End(E))\rightarrow \PV(X, End(E))$. For $\rho\in
\rOmega^{1,0}(X)$ and $\omega\in \PV(X,End(E))$, we have:
\be
{\det}_\Omega(\rho\otimes \omega)=({\det}_\Omega\rho) \omega^d~~,
\ee
where $\det_\rOmega\rho\in \cinf$ is defined through the relation: 
\be
\rho^d=({\det}_\Omega\rho) \Omega\in \rOmega^{d,0}(X)~~.
\ee
In a complex coordinate chart $(U, (z^1,\ldots, z^d))$ such that
$\Omega=_U\varphi(z)\dd z^1\wedge \ldots \wedge \dd z^d$ (where $f\in \cO_X(U)$
does not vanish on $U$), we can expand $s=_U\sum_{i=1}^d \dd
z^i\otimes s_i$ with $s_i\in \PV(U, End(E))$ and we have:
\ben
\label{dets}
{\det}_\Omega s=_U \frac{1}{\varphi(z)}\epsilon^{i_1\ldots i_d} s_{i_1}\ldots s_{i_d}~~,
\een
where $\epsilon^{i_1\ldots i_d}$ is the Levi-Civita symbol and we use
implicit summation over repeated indices, the multiplication being
taken in the algebra $\PV(U,End(E))$.

\subsection{The twisted curvature}

The natural isomorphism $End(T X)\simeq T^\ast X\otimes TX$ maps the
identity endomorphism $\id_{TX}\in \rGamma(X,End(TX))$ of the
holomorphic tangent bundle $TX$ into a holomorphic section $\theta\in
\rGamma(X,T^\ast X\otimes TX)\subset \rOmega^{1,0}(X,TX)$.
Let $G$ be a K\"{a}hler metric on $X$ and $\omega_G\in
\rOmega^{1,1}(X)$ be the K\"{a}hler form of $G$. Let $\ba=(E,h,D)$ be a
Hermitian factorization of $W$. Define:
\ben
\label{Vdef}
V_\ba^G\eqdef \mpd_\ba D+ F_\ba -\omega_G \id_E \in
\rOmega^{1,0}(X,End^\1(E))\oplus \rOmega^{1,1}(X,End^\0(E))~~,
\een 
where $F_\ba\in \rOmega^{1,1}(X,End^\0(E))$ is the curvature of the Chern connection of $(E,h)$ 
and $\mpd_a$ was defined in Subsection \ref{subsec:hhf}.

\

\begin{Definition}
The {\em twisted curvature} of the Hermitian holomorphic
factorization $\ba$ determined by $G$ is defined through:
\be
A_\ba^G\eqdef \theta\otimes \id_E+\i V_\ba^G\in \rOmega^{1,0}(X,TX\otimes End^\0(E))\oplus \rOmega^{1,0}(X,End^\1(E))\oplus \rOmega^{1,1}(X,End^\0(E))~~.  
\ee
\end{Definition}

\begin{remark}
We have:
\ben
\label{Aspace}
A_\ba^G\in \rOmega^{1,\ast}(X,\wedge TX\otimes End(E))\simeq \rOmega^{1,0}(X)\otimes_\cinf \PV(X,End(E))~~. 
\een
Choose local complex coordinates $z^1,\ldots, z^d$ defined on an open
subset $U\subset X$. Using the isomorphism in \eqref{Aspace}, we have
the following expansions, where juxtaposition in the second equality
denotes multiplication in the algebra $\PV(X,End(E))$:
\beqa
\theta~&=_U&\sum_{i=1}^d \dd z^i \otimes \theta_i~~\\
F~&=_U&\sum_{i=1}^d\dd z^i\otimes \left(\sum_{j=1}^d \dd {\bar z}^j F_{i{\bar j}}\right)~~\\
\omega_G&=_U&\sum_{i=1}^d\dd z^i \otimes \left(\sum_{j=1}^d \i G_{i{\bar j}}\dd {\bar z}^j\right)~~.
\eeqa
Here: 
\be
\theta_i\eqdef\partial_i:=\frac{\partial}{\partial z^i}\in\rGamma(U, TX)\subset \cA^0(U,TX)=\PV^{-1,0}(U)~~
\ee
and $G_{i{\bar j}}\in \cC^\infty(U)$, $F_{i{\bar j}}\in \rGamma_\sm(U,End^\0(E))$. 
We also have $V_\ba^G=_U\sum_{i=1}^d \dd z^i\otimes V_i$ and $A_\ba^G=_U\sum_{i=1}^d \dd z^i\otimes A_i$, where:
\beqan
\label{VAlocal}
&& V_i=_U\nabla^{1,0}_i D+\sum_{j=1}^d \dd \bar{z}^j \! \otimes \! (F_{i{\bar j}} \! - \! \i  G_{i{\bar j}}\id_E)\in \cA^0(U,End^\1(E))\! 
\oplus\! \cA^1(U,End^\0(E))\subset \PV(U,End(E))^\1 \nn\\
&& A_i=_U\theta_i \otimes \id_E+\i V_i\in \PV(U,End(E))^\1~~.
\eeqan
\end{remark}

\subsection{The disk kernel of a Hermitian holomorphic factorization}

\

\

\begin{Definition}
The {\em disk kernel} of the Hermitian holomorphic factorization
$\ba=(E,h,D)$ determined by $\Omega$ and by the K\"{a}hler metric $G$ is the element
$\Pi_\ba:=\Pi_\ba^{\Omega,G}\in \PV(X,End(E))$ defined through the relation:
\ben
\label{PiDef}
\Pi_\ba^{\Omega, G}=\frac{1}{d!}{\det}_\Omega A_\ba^G~~.
\een
\end{Definition}
If $(U,(z^1,\ldots, z^d))$ is a complex coordinate chart with
$\Omega=_U\varphi(z)\dd z^1\wedge \ldots \wedge \dd z^d$, relations
\eqref{VAlocal} and \eqref{dets} give:
\ben
\label{PiCoord}
\Pi_\ba=_U\frac{1}{d!\varphi(z)} \epsilon^{i_1\ldots i_d} \big(\theta_{i_1}\id_E+\i V_{i_1}\big)\ldots \big(\theta_{i_d}\id_E+\i V_{i_d}\big)~~.
\een
Notice that $A_i$ have odd total $\Z_2$-degree and hence:
\be
\deg \Pi_\ba=\mu={\hat d}\in \Z_2~~.
\ee

\begin{Lemma}
\label{lemma:rels}
The following relations hold in any complex coordinate
chart $(U,(z^1,\ldots, z^d))$ on $X$:
\beqan
&& \updelta_a V_i=(\pd_i W)\otimes \id_E~~\label{Vrel1}~~\\
&& \updelta_W \theta_i=-\i \pd_i W~~ \label{Vrel2}~~\\
&& \Delta_a(A_i)=0~~. \label{Vrel3}
\eeqan
\end{Lemma}

\begin{proof}

\

\begin{enumerate}[1.]
\itemsep 0.0em
\item Since $D$ is holomorphic, we have $\bbpd_a D=0$. Since
  $\omega_G$ and $F_\ba$ are closed $(1,1)$-forms, we also have $\bpd
  \omega_G=\bbpd_a F_\ba=0$. Thus $\bbpd_a V_\ba^G=\bbpd_a \mpd_\ba
  D=[F,D]$ (see \eqref{mpdrels}) and:
\be
\updelta_a V_\ba^G=\bbpd_a V_\ba^G+[D,V_\ba^G]=[F,D]+[D,\mpd_\ba D+F]=[D,\mpd_a D]=(\pd W) \id_E~~,
\ee
where in the second relation we used the condition $D^2=W\id_E$, which
implies $[D,\mpd_a D]=(\pd W)\id_E$. Thus \eqref{Vrel1} holds. 
\item Since $\theta_i$ is holomorphic, we have: 
\be
\updelta_W\theta_i=\ioda_W(\theta_i)=\ioda_W(\partial_i)= -\i \pd_i W~~.
\ee
\item Using relations \eqref{Vrel1} and \eqref{Vrel2}, we compute: 
\be
\Delta_a(A_i)=\Delta_a (\theta_i+\i V_i)=(\updelta_W\theta_i) \id_E+\i \updelta_a V_i=0~~.
\ee
\end{enumerate}
\qed
\end{proof}

\subsection{The twisted Atiyah class}

Recall that the Atiyah class (see \cite{Atiyah}): 
\be
A(E)\in \rH^{1,1}(X,End(E))\simeq \rH^1(T^\ast X\otimes End(\cE))
\ee 
of a holomorphic vector bundle $E$ on $X$ coincides (see \cite[Proposition
  4.3.10]{Huybrechts}) with the $\bbpd_{End(E)}$-cohomology class of
the curvature of the Chern connection $F\in \rOmega^{1,1}(X,E)$
determined by any Hermitian connection $h$ on $E$. In this subsection,
we introduce a ``twisted'' version of the Atiyah class for holomorphic
factorizations. Let $\ba=(E,h,D)$ be a Hermitian holomorphic
factorization of $W$ with underlying holomorphic factorization
$a=(E,D)$.

\

\begin{Proposition}
\label{prop:Aclosed}
We have: 
\ben
\label{Aclosed}
\Delta_a A_\ba=0~~
\een
and: 
\ben
\label{PiRel}
\Delta_a \Pi_\ba=0~~.
\een
Moreover, the $\Delta_a$-cohomology class $[\Pi_\ba]_{\Delta_a}\in
\HPV(X,a)$ depends only on $\Omega$ and on the $\Delta_a$-cohomology class
$[A_\ba]_{\Delta_a}\in \HPV_e(X,a)$.
\end{Proposition}

\begin{proof}
Both statements regarding $A_\ba$ follow from \eqref{Vrel3} upon
using the fact that $\Delta_a$ is a graded derivation. The fact that
$[\Pi_\ba]_{\Delta_a}$ depends only on $\Omega$ and $[A_\ba]_{\Delta_a}$
follows immediately using $\Delta_a$-closure of $A_\ba$ and the
definition of $\Pi_\ba$. ~\qed
\end{proof}

\

\begin{Definition}
The {\em twisted Atiyah class} of $\ba$ induced by $G$ is the
$\Delta_a$-cohomology class $[A_\ba^G]_{\Delta_a}\in \HPV_e(X,a)$ of $A_\ba^G$.
\end{Definition}

\

\begin{Definition}
We say that the {\em $\pd\bpd$-lemma holds for $(1,1)$-forms on $X$} if 
any $(1,1)$-form $\omega\in \Omega^{1,1}(X)$ which is both $\pd$ and $\bpd$-closed 
as well as $\dd$-exact can be written as $\omega=\bpd\pd \varphi$ for some 
smooth complex-valued function $\varphi\in \cC^\infty(X)$. 
\end{Definition}

\

\noindent Recall that any compact K\"ahler manifold satisfies the
$\pd\bpd$-lemma for all $(p,q)$ forms (see \cite[Proposition 6.17,
page 144]{Voisin}). In our case, even though $X$ is K\"ahlerian, it
may not satisfy the $\pd\bpd$-lemma for $(1,1)$ forms since it is
non-compact. The $\pd\bpd$-lemma holds for $(1,1)$-forms on $X$ iff
the natural map $H_{BC}^{1,1}(X)\rightarrow H^2_{\mathrm{dR}}(X,\C)$
is injective\footnote{This statement (which does not require
compactness of $X$) follows directly from the definition $H_{BC}(X)\eqdef
[\ker(\pd)\cap \ker(\bpd)]/\im (\pd\bpd)$ of Bott-Chern cohomology.},
where $H^{1,1}_{BC}(X)$ and $H^2_{\mathrm{dR}}(X,\C)$ denote the
corresponding Bott-Chern (see \cite[Section 8.1]{Demailly}) and de Rham
cohomology groups of $X$. Bott-Chern cohomology was introduced in
\cite{BC}.

\

\begin{Proposition}
\label{prop:At}
The twisted Atiyah class $[A_\ba^G]_{\Delta_a}$ of a Hermitian
holomorphic factorization $\ba=(E,D,h)$ is independent of the choice
of admissible metric $h$ on $E$.  Moreover, when the $\pd\bpd$-lemma
holds for $(1,1)$-forms on $X$, the twisted Atiyah class
$[A_\ba^G]_{\Delta_a}$ depends only on the K\"{a}hler class
$[\omega_G]\in \rH^{1,1}(X)$ of $G$ and on the Atiyah class $A(E)\in
\rH^{1,1}(X,End(E))$ of the holomorphic vector bundle $E$.
\end{Proposition}

\begin{proof}
We first show that the twisted Atiyah class does not depend on the
choice of $h$.  Let $\nabla$ be the Chern connection of $(E,h)$ and
$F\in \rOmega^{1,1}(X,End^\0(E))$ denote its curvature. Let $h'$ be
another admissible Hermitian metric on $E$ and $F'\in
\rOmega^{1,1}(X,End^\0(E))$ be the curvature of the Chern connection
$\nabla'$ of $(E,h')$. Let $\bbpd:=\bbpd_{End(E)}=\bbpd_a$ and
$\mpd=\mpd_\ba$, $\mpd':=\mpd_{\ba'}$, where $a'$ denotes the
Hermitian holomorphic factorization $(E,h,D)$. We have
$\nabla'-\nabla=\nabla'^{1,0}-\nabla^{1,0}=S$ for some $S\in
\rOmega^{1,0}(X,End^\0(E))$. This implies $\mpd_E^{h'}=\mpd_E^h+S$. In
turn, this gives:
\be
\mpd'=\mpd+[S,\cdot]~~,~~F'=F+\bbpd S~~,
\ee
where $[\cdot, \cdot]$ denotes the graded commutator. 
Let $\ba'\eqdef (E,h',D)$ and $A:=A_\ba^G$,
$A':=A_{\ba'}^{G'}$. Relation \eqref{Vdef} gives:
\ben
\label{Vpr}
A'=A+[S, D]+\bbpd (S-\rho\id_E)=A+\updelta_a B=A+\Delta_a B~~,
\een
where $B\eqdef S- \rho\id_E\in \rOmega^{1,0}(X,End^\0(E))$. Thus
$[A']_{\Delta_a}=[A]_{\Delta_a}$~.~

Let us now assume that the $\pd\bpd$-lemma holds for $(1,1)$-forms on
$X$.  Let $G'$ be another K\"{a}hler metric on $X$ whose K\"{a}hler
form $\omega':=\omega_{G'}$ belongs to the same de Rham cohomology
class as $\omega:=\omega_G$.  Since $\omega'-\omega$ has type $(1,1)$,
the conditions $\dd \omega=\dd\omega'=0$ imply that $\omega'-\omega$
is both $\pd$- and $\bpd$- closed. Since $\omega'-\omega$ is $d$-exact
and the $\pd\bpd$-lemma holds for $(1,1)$-forms on $X$, there exists
$\varphi\in \cinf$ such that:
\be
\omega'-\omega=\bpd\pd\varphi~~.
\ee 
Setting $\rho\eqdef \pd\varphi\in \rOmega^{1,0}(X)$, this gives: 
\be
\omega'=\omega+\bpd \rho~~.
\ee
\qed
\end{proof}

\begin{remark}
Propositions \ref{prop:Aclosed} and \ref{prop:At} imply that the
$\Delta_a$-cohomology class $[\Pi_\ba]_{\Delta_a}\in \HPV(X,a)$
depends only on the holomorphic volume form $\Omega$, on the
K\"{a}hler class $[\omega_G]$ and on the Atiyah class $A(E)$, provided 
that the $\pd\bpd$-lemma holds for $(1,1)$-forms on $X$.
\end{remark}

\section{Off-shell boundary-bulk maps}
\label{sec:BoundaryBulk}

\begin{Definition}
The {\em canonical off-shell boundary-bulk map} of the Hermitian holomorphic
factorization $\ba=(E,h, D)$ determined by $\Omega$ and by the K\"{a}hler metric $G$
is the $\cinf$-linear map
$f_\ba^B:=f_\ba^{B,\Omega, G}:\End_{\DF(X,W)}(a)\rightarrow \PV(X)$ defined
through:
\be
f_\ba^{B,\Omega,G}(\alpha)\eqdef \str(\Pi_\ba^{\Omega, G}\alpha)~~,~~\forall \alpha\in \End_{\DF(X,W)}(a)=\cA(X,End(E))~~.
\ee
\end{Definition}
Notice that $f_\ba^B$ has total $\Z_2$-degree $\mu$. 

\

\begin{Proposition}
We have: 
\ben
\updelta_W\circ f^B_\ba=(-1)^d f^B_\ba\circ \updelta_a~~.
\een
In particular, $f^B_\ba$ descends to an $\O(X)$-linear map from
$\End_{\DF(X,W)}(a)$ to $\HPV(X,W)$.
\end{Proposition}

\begin{proof}
For any $\alpha\in \cA(X,End(E))$, we have: 
\be
\label{Rel1}
f_\ba^B(\updelta_a \alpha)=\str[\Pi_\ba (\updelta_a\alpha)]=(-1)^d \str[\Delta_a(\Pi_\ba \alpha)]=(-1)^d \updelta_W [\str(\Pi_\ba \alpha)]=(-1)^d\updelta_W f_\ba^B(\alpha)~~,
\ee
where in the last equality we noticed that
$\Delta_a\alpha=\delta_a\alpha$ and used \eqref{PiRel} and
\eqref{strDelta} ~\qed
\end{proof}

\

\begin{Definition}
The {\em cohomological boundary-bulk map} of $\ba=(E,h, D)$
is the $\O(X)$-linear map $f_\ba:=f_\ba^{\Omega, G}:\End_{\HDF(X,W)}(a)\rightarrow
\HPV(X,W)$ induced by $f_\ba^{B,\Omega, G}$ on cohomology.
\end{Definition}

\section{Off-shell bulk-boundary maps}
\label{sec:BulkBoundary}

\subsection{Canonical off-shell bulk-boundary maps}

\

\

\begin{Definition}
The {\em canonical off-shell bulk-boundary map} of the Hermitian
holomorphic factorization $\ba=(E,h, D)$ determined by $\Omega$ and by
the K\"{a}hler metric $G$ is the $\cinf$-linear map
$e_\ba^B:=e_\ba^{B,\Omega, G}:\PV(X)\rightarrow
\End_{\DF(X,W)}(a)$ defined through:
\be
e_\ba^{B,\Omega,G} (\omega)\eqdef  \Omega\lrcorner_0 \left(\omega \Pi_\ba^{\Omega, G} \right)~~,~~\forall \omega\in \PV(X)~~.
\ee
\end{Definition}

\noindent Notice that $e_\ba^B$ has total $\Z_2$-degree $\0$.

\

\begin{Proposition}
\label{prop:eBclosed}
We have:
\ben
\updelta_a\circ e_\ba^B=(-1)^d e_\ba^B \circ \updelta_W~~.
\een
In particular, $e^B_\ba$ descends to an $\O(X)$-linear map from $\HPV(X,W)$ to $\End_{\DF(X,W)}(a)$.
\end{Proposition}

\begin{proof}
Let $\omega\in \PV(X)$. Using Lemma \ref{lemma:DeltaContr}, we
compute:
\be
\updelta_a(e_\ba^B (\omega))=\updelta_a [\Omega\lrcorner_0(\omega\Pi_\ba)]= (-1)^d \Omega\lrcorner_0 \left(\Delta_a(\omega)\Pi_\ba\right)=
(-1)^d \Omega\lrcorner_0 \left(\updelta_W(\omega)\Pi_\ba\right)=(-1)^d e_\ba^B(\updelta_W\omega)~~,
\ee
where we used the fact that $\Pi_\ba$ is $\Delta_a$-closed (see
\eqref{PiRel}) and noticed that $\Delta_a \omega=\updelta_W\omega$. \qed
\end{proof}

\

\begin{Definition}
The {\em cohomological bulk-boundary map} of $\ba=(E,h,D)$ is the
$\O(X)$-linear map
$e_\ba:=e_\ba^{\Omega, G}:\HPV(X,W)\rightarrow \End_{\HDF(X,W)}(a)$ induced by
$e_\ba^{B,\Omega, G}$ on cohomology.
\end{Definition}

\subsection{Tempered off-shell bulk-boundary maps}

\

\

\begin{Definition}
For any $\lambda\geq 0$ the {\em $\lambda$-tempered off-shell
  bulk-boundary map} of the Hermitian holomorphic factorization
$\ba=(E,h, D)$ determined by $\Omega$ and by the K\"{a}hler metric $G$
is the $\cinf$-linear map $e_\ba^{(\lambda)}:=e_\ba^{(\lambda),\Omega,
  G}:\PV(X)\rightarrow
\End_{\DF(X,W)}(a)$ defined through:
\be
e_\ba^{(\lambda),\Omega, G}\eqdef  U_\ba^G(-\lambda)\circ e_\ba^{B,\Omega, G}\circ U_G(\lambda)~~.
\ee
\end{Definition}

\noindent Notice that $e_\ba^{(\lambda)}$ has total $\Z_2$-degree $\0$. We have:
\ben
\label{e0}
e_\ba^{(0)}= e_\ba^{B}
\een
and:
\be
e_\ba^{(\lambda)}(\omega)=e^{\lambda L_\ba^G} \Omega\lrcorner_0 \left(e^{-\lambda L_G} \omega \Pi_\ba \right)~~,~~\forall \omega\in \PV(X)~~.
\ee

\

\noindent In what follows, we fix a holomorphic volume form $\Omega$ and a K\"{a}hler metric $G$ on $X$ and
denote $e_\ba^{(\lambda),\Omega, G}$ by $e_\ba^{(\lambda)}$ and $e_\ba^{B,\Omega, G}$
by $e_\ba^B$. We also denote $L_G$, $v_G$ by $L$, $v$
and $L_\ba^G$, $v_\ba^G$ by $L_\ba$, $v_\ba$.

\

\begin{Proposition}
For any $\lambda\geq 0$, we have:
\ben
\label{eclosed}
\updelta_a\circ e_\ba^{(\lambda)}=(-1)^d e_\ba^{(\lambda)}\circ \updelta_W~~.
\een
\end{Proposition}

\begin{proof}
Follows immediately from Proposition \ref{prop:eBclosed} using
relations \eqref{UGclosed} and \eqref{Uaclosed}. \qed
\end{proof}

\

\begin{Proposition}
\label{prop:onshelltr}
Let $\lambda \geq 0$. For any $\omega\in \ker \updelta_W$, we have:
\ben
\label{ee0}
e_{\ba}^{(\lambda)}(\omega)=e_\ba^B(\omega)~~\mod ~~\im \updelta_a~~.
\een
In particular, the map induced by $e_\ba^{(\lambda)}$ on cohomology
does not depend on $\lambda$ and coincides with the cohomological
bulk-boundary map.
\end{Proposition}

\begin{proof}
For any $\omega\in \PV(X)$ such that $\updelta_W\omega=0$, we compute:  
\be
\frac{\dd}{\dd \lambda} e_\ba^{(\lambda)}(\omega)=\frac{\dd}{\dd \lambda} \left[e^{\lambda L_\ba} e_\ba^B(e^{-\lambda L}\omega)\right]=L_\ba e_\ba^{(\lambda)}(\omega)-e_\ba^{(\lambda)}(L\omega)=
(\updelta_\ba v_\ba) e_\ba^{(\lambda)}(\omega)-e_\ba^{(\lambda)}\big((\updelta_W v)\omega\big)~~,
\ee
where we used the fact that $L_\ba=\updelta_\ba v_\ba$ and
$L=\updelta_W v$. We have:
\be
(\updelta_\ba v_\ba) e_\ba^{(\lambda)}(\omega)=\updelta_\ba \left[v_\ba e_\ba^{(\lambda)}(\omega)\right]+v_\ba \updelta_\ba e_\ba^{(\lambda)}(\omega)=\updelta_\ba \left[v_\ba e_\ba^{(\lambda)}(\omega)\right]~~,
\ee
where we noticed that relation \eqref{eclosed} implies: 
\be
\updelta_\ba e_\ba^{(\lambda)}(\omega)=(-1)^de_\ba^{(\lambda)}(\updelta_W \omega)=0~~,
\ee
since $\updelta_W\omega=0$. On the other hand, we have: 
\be
e_\ba^{(\lambda)}\big((\updelta_W v) \omega\big)=e_\ba^{(\lambda)}\big(\updelta_W (v \omega)\big)=(-1)^d\updelta_\ba e_\ba^{(\lambda)}(v\omega)~~,
\ee
where we once again used relation \eqref{eclosed}. Thus: 
\be
\frac{\dd}{\dd \lambda} e_\ba^{(\lambda)}(\omega)=\updelta_\ba \left[v_\ba e_\ba^{(\lambda)}(\omega)-(-1)^d e_\ba^{(\lambda)}(v \omega)\right]~~,
\ee
which implies that $e_\ba^{(\lambda)}$ is independent of $\lambda$
modulo $\updelta_\ba$-exact terms. This implies relation \eqref{ee0}
upon recalling that $e_\ba^{(0)}=e_\ba^B$. \qed
\end{proof}

\subsection{Independence of metric data}

\

\

\begin{Proposition}
The cohomological bulk-boundary and boundary-bulk maps $e_\ba$ and
$f_\ba$ of a Hermitian holomorphic factorization $\ba=(E,D,h)$ do not
depend on the choice of the admissible Hermitian metric $h$ on $E$.
Moreover, when the $\pd\bpd$-lemma holds for $(1,1)$-forms on $X$ then
$e_\ba$ and $f_\ba$ depend only on $\Omega$, on the K\"{a}hler class
$[\omega_G]\in \rH^{1,1}(X)$ of $G$ and on the Atiyah class $A(E)\in
\rH^{1,1}(X,End(E))$ of the holomorphic vector bundle $E$.
\end{Proposition}

\

\noindent  Accordingly, we denote $f_\ba$ and $e_\ba$ by $f_a$ and $e_a$ (since they are independent of $h$). 

\

\begin{proof}
Let $G$ and $G'$ be two K\"{a}hler metrics on $X$ having the same
K\"{a}hler class and $h$ and $h'$ be two admissible Hermitian metrics
on $a=(E,D)$. Let $\ba=(E,h,D)$ and $\ba'=(E,h',D)$.  Let
$f_B:=f_{\ba}^{B,\Omega, G}$, $f'_B:=f_{\ba'}^{B,\Omega, G'}$ and
$e_B:=e_\ba^{B,\Omega, G}$, $e'_B:=e_{\ba'}^{B,\Omega, G'}$. Let
$\Pi:=\Pi_\ba^{\Omega, G}$ and $\Pi':=\Pi_{\ba'}^{\Omega, G'}$. When
$G'=G$ or when $G'\neq G$ but the $\pd\bpd$-lemma holds for
$(1,1)$-forms on $X$, Proposition \ref{prop:At} implies:
\be
\Pi'=\Pi+\Delta_a(T)~~
\ee
for some $T\in \rOmega(X,\wedge TX \otimes End(E))$.  This implies
that the following relations hold for any $\alpha\in \ker \updelta_a$
and any $\omega\in \ker \updelta_W$ of pure $\Z_2$-degree:
\ben
\label{r1}
f'_B(\alpha)=\str(\Pi'\alpha)=f_B(\alpha)+\str[(\Delta_\ba T)\alpha]=f_B(\alpha)+\str[\Delta_{\ba}(T\alpha)]=f_B(\alpha)+\updelta_W\str(T\alpha)
\een
and: 
\beqan
\label{r2}
e'_B(\omega)=\Omega\lrcorner_0 (\omega\Pi')&=& e_B(\omega)+\Omega\lrcorner_0 [\omega\Delta_a\Pi]=e_B(\omega)+(-1)^{\deg\omega}  \Omega\lrcorner_0\Delta_a(\omega\Pi)=\nn\\
&=& e_B(\omega)+(-1)^{d+\deg\omega} \updelta_a[\Omega\lrcorner_0 (\omega\Pi)]~~,
\eeqan
where in the last equalities we used relation \eqref{strDelta} and
Lemma \ref{lemma:DeltaContr}. Relations \eqref{r1} and \eqref{r2}
imply that $f'_B$ and $e'_B$ induce the same maps on cohomology as
$f_B$ and $e_B$, respectively. \qed
\end{proof}

\subsection{Adjointness relations}

Let $G$ be a fixed K\"{a}hler metric on $X$ and $\ba=(E,h,D)$ be a
Hermitian holomorphic factorization of $W$ with underlying holomorphic
factorization $a=(E,D)$. Let $f_\ba^B:=f_\ba^{B,\Omega,G}$ and
$e_\ba^B:=e_\ba^{B,\Omega, G}$. For any $\omega\in \PV(X)$ and any
$\alpha\in \End_{\DF_c(X,W)}(a)=\cA(X,End(E))$, we have: 
\be
\Omega\lrcorner_0 [\omega f_\ba^B (\alpha)]=\str[e_\ba^B (\omega)\alpha]=\lambda_\Omega(\omega \Pi_\ba \alpha)~~,
\ee
where $\Pi_\ba:=\Pi_\ba^{\Omega, G}$ and $\lambda_\Omega$ was defined
in \eqref{lambdaDef}. This identity can be viewed as a local
adjointness relation between $e_\ba^B$ and $f_\ba^B$. When either
$\omega$ or $\alpha$ has compact support, applying $\int_\Omega$ to
the relation above gives the global adjointness relation:
\be
\Tr_B[\omega f_\ba^B (\alpha)]=\tr_a^B[e_\ba^B (\omega)\alpha]=\int_\Omega \lambda_\Omega(\omega \Pi_\ba \alpha)~~.
\ee
The following result shows that similar adjointness relations are
satisfied by the tempered traces and tempered bulk-boundary and
boundary-bulk maps.

\

\begin{Proposition}
Let $\lambda\geq 0$. For any $\omega\in \PV(X)$ and any
$\alpha\in \End_{\DF(X,W)}(a)$, we have:
\ben
\label{adj}
\Omega\lrcorner_0 \big[e^{-\lambda L} \omega f_\ba^B(\alpha)\big]=\str\big[e^{-\lambda L_\ba} e_\ba^{(\lambda)} (\omega)\alpha\big]=\lambda_\Omega\big(e^{-\lambda L}\omega\Pi_\ba \alpha\big)~~.
\een
\end{Proposition}

\begin{proof}
Using the definitions of $f_\ba$ and $e_\ba^{(\lambda)}$, we compute: 
\be
\Omega\lrcorner_0 \big[e^{-\lambda L} \omega f_\ba^B(\alpha)\big]=\Omega\lrcorner_0 \big[ e^{-\lambda L}\omega~\str(\Pi_\ba\alpha)\big]=\lambda_\Omega(e^{-\lambda L}\omega\Pi_\ba \alpha)
\ee
and:
\be
\str\big[e^{-\lambda L_\ba} e_\ba^{(\lambda)} (\omega)\alpha\big]=\str\big(\Omega\lrcorner_0 \big[e^{-\lambda L} \omega \Pi_\ba\big]\alpha \big)=
\Omega\lrcorner_0 \big[e^{-\lambda L} \omega~\str(\Pi_\ba \alpha)\big]=\lambda_\Omega(e^{-\lambda L}\omega\Pi_\ba \alpha)~~,
\ee
showing that \eqref{adj} holds. \qed
\end{proof}

\

\begin{Proposition}
\label{prop:adj}
Let $\lambda\geq 0$. For any $\omega\in \PV_c(X)$ and any
$\alpha\in \End_{\DF_c(X,W)}(a)$, we have:
\ben
\label{adjointness}
\Tr^{(\lambda)}(\omega f_\ba^B(\alpha))=\tr_\ba^{(\lambda)} (e_\ba^{(\lambda)} (\omega)\alpha)=\int_\Omega \lambda_\Omega(\omega \Pi_\ba \alpha)~~.
\een
\end{Proposition}

\begin{proof}
Follows immediately by applying $\int_\Omega$ to both sides of relation \eqref{adj}.~\qed
\end{proof}

\

\begin{Corollary}
\label{cor:adj}
Let $\ba=(E,h, D)$ be a Hermitian holomorphic factorization of
$W$. For any $\omega\in \HPV_c(X,W)$ and any
$\alpha\in \End_{\HDF_c(X,W)}(a)$, we have:
\be
\Tr_c(\omega f_a (\alpha))=\tr_a^c(e_a(\omega)\alpha)~~.
\ee
When the critical set $Z_W$ is compact, we have: 
\be
\Tr(\omega f_a (\alpha))=\tr_a(e_a(\omega)\alpha)~~,~~\forall \omega\in \HPV(X,W)~~,
~~\forall \alpha\in \End_{\HDF(X,W)}(a)~~.
\ee
\end{Corollary}

\begin{proof}
The first statement follows immediately from Proposition
\ref{prop:adj}. This implies the second statement upon using the
definitions of $\Tr$ and $\tr_a$ and Propositions \ref{prop:i} and
\ref{prop:j}.~\qed
\end{proof}

\section{Proof of the main theorem}
\label{sec:Datum}

 Assume that $W$ has compact critical set. Let $f:=
(f_a)_{a\in \Ob \HDF(X)}$ and $e :=(e_a)_{a\in \Ob\HDF(X)}$. In
this case, Propositions \ref{prop:i} and \ref{prop:j} allow us to
transport traces $\Tr_c$ and $\tr^c\!:=\!(\tr_a^c)_{a\in \Ob\HDF(\!X,\!W\!)}$
into traces $\Tr$ and $\tr:= (\tr_a)_{a\in \Ob\HDF(X,W)}$ defined on
$\HPV(X,W)$ and $\HDF(X,W)$ respectively. Proposition
\ref{prop:bulkfd} shows that the algebra $\HPV(X,W)$ is
finite-dimensional over $\C$ while Proposition \ref{prop:boundaryfd}
shows that the category $\HDF(X,W)$ is Hom-finite over $\C$. On the
other hand, Proposition \ref{prop:cancyclic} implies that
$(\HDF_c(X,W),\tr^c)$ (and hence also $(\HDF(X,W),\tr)$) is a
pre-Calabi-Yau category. Combining these results gives Theorem
\ref{thm}. Moreover, Corollary \ref{cor:adj} implies that the
cohomological bulk-boundary maps are adjoint to the cohomological
boundary-bulk maps with respect to $\Tr$ and $\tr$.

In view of Theorem \ref{thm}, proving Conjecture \ref{conj} amounts to
showing that the cohomological bulk and boundary traces are
non-degenerate and that the topological Cardy constraint is satisfied
when the critical set of $W$ is compact. The first of these problems
is addressed in \cite{tserre}. When $X$ is Stein and the critical set
of $W$ is compact, the conjecture follows by adapting the results of
\cite{DM,PV,D}, as we sketch in the next section.

\section{Simplifications when $X$ is Stein}
\label{sec:Stein}

Suppose that $X$ is a Stein Calabi-Yau manifold\footnote{Notice that
any Stein manifold $X$ is K\"ahlerian since it can be embedded
holomorphically as a complex submanifold of some affine space $\C^N$;
a K\"ahler form on $X$ is then obtained by restricting the K\"ahler
form of $\C^N$.}  and that the critical locus $Z_W$ is compact. Then
$Z_W$ is necessarily finite since any compact analytic subset of a
Stein manifold is finite (see \cite[Chapter V.4, Theorem 3]{GR}). This
situation is studied in \cite{nlg2}, where it is shown that the
following statements hold for such {\em Stein LG pairs} $(X,W)$ with
finite critical set:
\begin{enumerate} \itemsep 0.0em
\item $\HPV(X,W)$ is concentrated in degree zero and isomorphic as an
$\O(X)$-algebra with the {\em Jacobi algebra}
$\Jac(X,W)=\O(X)/\J(X,W)$, where
$\J(X,W)=\ioda_W(\rGamma(X,TX))\subset \O(X)$ is the critical ideal of
$W$ (see \cite[Proposition 4.3]{nlg2}).
\item $\DF(X,W)$ is equivalent with the $\O(X)$-linear dg-category
$\F(X,W)$ whose objects are the holomorphic factorizations of $W$ but
whose morphism spaces from $a_1=(E_1,D_1)$ to $a_2=(E_2,D_2)$ are
given by the $\Z_2$-graded $\O(X)$-module $\rGamma(X,Hom(E_1,E_2))$ of
{\em holomorphic} sections of the bundle $Hom(E_1,E_2)=E_1^\vee\otimes
E_2$, endowed with the defect differential $\fd_{a_1,a_2}$ of Section
\ref{sec:DF} (see \cite[Theorem 5.4]{nlg2}). In particular,
$\HDF(X,W)$ is equivalent with the total cohomology category
$\HF(X,W)$ of $\F(X,W)$.
\item For any holomorphic factorization $a$ of $W$, we have a natural
isomorphism $\HPV(X,a)\!\simeq\!
\Jac(X,W)\otimes_{\O(X)} \End_{\HF(X,W)}(a)$ (see \cite[Proposition
6.4]{nlg2}).
\end{enumerate}
As shown in \cite{nlg2}, these statements follow by combining Cartan's
theorem B (see \cite[page 124]{GR}) with certain spectral sequence
arguments. The class of Stein manifolds contains that of non-singular
affine varieties over $\C$ since the analytic space associated to an
affine variety is Stein. However, most Stein manifolds are not
analyticizations of affine varieties. Thus models associated to Stein
LG pairs provide a wide extension of the class of affine
Landau-Ginzburg theories; see reference \cite{nlg2} for examples.

\subsection{Germs of holomorphic functions and of holomorphic sections}

For every point $p\in X$ and any holomorphic function $f\in \O(X)$,
let ${\hat f}_p\in \cO_{X,p}$ denote the germ of $f$ at $p$. Since we
assume that $X$ is connected, the morphism of $\C$-algebras
$\germ_p:\O(X)\rightarrow \cO_{X,p}$ defined through $\germ_p(f)\eqdef
{\hat f}_p$ is injective\footnote{This follows from the open mapping
theorem, which implies that the zero set of a holomorphic function
which is not identically zero on a connected complex manifold $X$ is a
closed subset of $X$ which has empty interior. Hence a holomorphic
function which vanishes on an open subset of $X$ must vanish
identically on $X$.}. This morphism is far from being surjective since
a holomorphic function defined on a small neighborhood of $p$ need not
extend holomorphically to all of $X$. Using local holomorphic charts
$(U_p,z^1_p,\ldots, z^d_p)$ of $X$ such that $p\in U_p$ and
$z^1_p(p)=\ldots=z^d_p(p)=0$ allows us to identify $\cO_{X,p}$ with
the ring $\C\{z^1,\ldots, z^d\}$ of power series in $d$ complex
variables which converge in a neighborhood of the origin of
$\C^d$. Recall that $\cO_{X,p}\simeq \C\{z^1,\ldots, z^d\}$ is a
Noetherian local ring which is a unique factorization domain (see
\cite[Chapter 2]{GRs}). Given a holomorphic vector bundle $E$ defined
over $X$ and a holomorphic section $s\in \rGamma(X,E)$, let ${\hat
s}_p\in \cO_{X,p}\otimes_\C E_p$ denote the germ of $s$ at $p$, where
$E_p$ denotes the fiber of $E$ at $p$.

\begin{remark} Notice that $\cO_{X,p}$ is much larger than the
localization of the ring $\O(X)$ at the multiplicative system given by
the complement of the maximal ideal $m_p(X)$ consisting of those
elements of $\O(X)$ which vanish at $p$. For example, the complex
plane $\C$ is Stein and Calabi-Yau, but $\C\{z\}$ is larger than the
localization of $\O(\C)$ at $m_0(\C)$ (consider the germ at the origin
of the analytic function $\log(1-z)$). This is quite different from
the case of affine algebraic varieties, whose sheaf of regular
functions has the property that its stalk at every closed point is the
localization of the coordinate ring at the multiplicative system given
by the complement of the corresponding maximal ideal.
\end{remark}

\subsection{Analytic Milnor algebras at the critical points}

Choosing local holomorphic coordinates $(z^1_p,\ldots, z^d_p)$ of $X$
centered at a critical point $p\in Z_W$ allows us to view the germ
${\hW}_p$ as an element ${\tW}_p$ of the $\C$-algebra $\C\{z^1,\ldots,
z^d\}$. Denoting $\frac{\pd {\hW}_p}{\pd z^i_p}\in \cO_{X,p}$ by
$\pd_i {\hW}_p$, let:
\be
\rM(\hW_p)\eqdef \frac{\cO_{X,p}}{\langle \pd_1 {\hW}_p,\ldots, \pd_d
{\hW}_p\rangle}\simeq \frac{\C\{z^1,\ldots, z^d\} }{\langle
\frac{\partial {\tW}_p}{\partial z^1}, \ldots, \frac{\partial
{\tW}_p}{\partial z^d}\rangle}~~
\ee
denote the analytic Milnor algebra of the analytic function germ
${\hW}_p$. Since each critical point $p$ of $W$ is isolated, its
Milnor number $\nu_p=\dim_\C \rM(\hW_p)$ is finite (see
\cite{Greuel,Arnold}). By Tougeron's finite determinacy theorem (see
op. cit.), the germ $\hW_p$ is right-equivalent with its jet
$j_{\nu_p+1}(\hW_p)$ of order $\nu_p+1$, which can be identified with
the polynomial ${\tw}_p\eqdef j_{\nu_p+1}({\tW}_p)\in
\C[z^1,\ldots, z^d]$. Equivalently, one can choose local holomorphic
coordinates $z^1_p,\ldots, z^d_p$ centered at $p$ such that ${\tW}_p$
is a polynomial of total degree $\nu_p+1$. This gives isomorphisms of
$\C\{z^1,\ldots, z^d\}$-algebras:
\be 
\rM({\tW}_p)\simeq \rM(j_{\nu_p+1}({\tW}_p))=\frac{\C\{z^1,\ldots,
z^d\}}{\langle \frac{\partial {\tw}_p}{\partial z^1}, \ldots,
\frac{\partial {\tw}_p}{\partial z^d}\rangle}~~,~~\forall p\in
Z_W~~.
\ee
The particular case $\nu_p=1$ arises when $p$ is a non-degenerate
critical point, i.e. when the Hessian of $W$ at $p$ is non-degenerate
(in this case, one also says that $W$ is ``holomorphic Morse'' at
$p$). Then $\tilde {w}_p$ is a quadratic polynomial and Tougeron's
theorem reduces to the holomorphic Morse lemma.

\begin{remark}
For any critical point $p\in Z_W$, let $\rM_a({\tw})\eqdef
\C[z^1,\ldots, z^d]/\langle \pd_1 {\tw}_p,\ldots, \pd_d {\tw}_p\rangle$ 
be the {\em algebraic} Milnor algebra of the polynomial
${\tw}_p\in \C[z^1,\ldots, z^d]$. The inclusion $\C[z^1,\ldots,
z^d]\subset \C\{z^1,\ldots, z^d\}$ induces a well-defined morphism of
$\C$-algebras $\pi_p:\rM_a({\tw}_p)\rightarrow \rM({\tw}_p)$
which is surjective. This follows from the fact that the ideal
$\langle \pd_1{\tw}_p,\ldots, \pd_d{\tw}_p\rangle$ has
finite codimension inside $\C\{z^1,\ldots, z^d\}$ (since the Milnor
number $\nu_p$ is finite), which implies that $\rM({\tw}_p)$
admits a monomial basis (see \cite[p. 111]{Dimca}). As shown in
loc. cit., this morphism is bijective when the polynomial ${\tilde
w}_p$ happens to be quasi-homogeneous of some positive integral degree
with respect to some choice of positive integral weights for the
variables $z^1,\ldots, z^p$. In this special situation, one can
identify the analytic Milnor algebra $\rM(\hW_p)\simeq \rM({\tw}_p)$ 
with the algebraic Milnor algebra $\rM_a({\tw}_p)$.  It is
known (see \cite{Saitoq} or \cite[page 120]{Dimca}) that ${\tw}_p$
 is quasi-homogeneous iff $\hW_p$ is a quasi-homogeneous germ in
the sense that it satisfies $\hW_p\in \langle \pd_1 \hW_p,\ldots,
\pd_d \hW_p\rangle$, which in turn happens iff the Milnor and Tyurina
numbers of $\hW_p$ coincide. Notice that much of the literature on
B-type Landau-Ginzburg models with a single isolated critical point
considers only quasi-homogeneous germs. This is appropriate for those
special Landau-Ginzburg models which are scale-invariant --- unlike
the more general models of the present paper, which are not subject to
any scale-invariance condition.
\end{remark}

\subsection{Presentation of the Jacobi algebra as a direct sum of analytic Milnor algebras}

For every critical point $p\in Z_W$, the map $\germ_p:\O(X)\rightarrow
\cO_{X,p}$ induces a well-defined morphism of $\O(X)$-algebras:
\be
\Lambda_p:\Jac(X,W)\rightarrow \rM({\hW}_p)~~.
\ee
Here the right hand side is viewed as an algebra over $\O(X)$ using the
obvious scalar multiplication:
\be
f q_p \eqdef {\hat f}_p q_p ~~,~~\forall f\in \O(X)~,~\forall q_p\in \rM({\hW}_p)~~.
\ee
Consider the morphism of $\O(X)$-algebras defined through:
\be
\Lambda\eqdef \oplus_{p\in Z_W}\Lambda_p:\Jac(X,W)\rightarrow \oplus_{p\in Z_W}\rM({\hW}_p)~~.
\ee
The following result shows that, for Stein Landau-Ginzburg models with finite
critical set, the cohomological bulk algebra $\HPV(X,W)$ constructed
in this paper reduces to a direct sum of analytic Milnor algebras.

\

\begin{Proposition}
\label{prop:JacMilnor}
The map $\Lambda$ is an isomorphism of $\O(X)$-algebras. 
\end{Proposition}

\begin{proof}
The critical sheaf $\cJ_W=\im(\ioda_W:TX\rightarrow \cO_X)\subset
\cO_X$ is the ideal sheaf of the analytic set $Z_W$ defined inside $X$
by the condition $\pd W=0$. We thus have a short exact sequence of
sheaves of $\cO_X$-modules:
\ben
\label{cJseq}
0\longrightarrow \cJ_W\longrightarrow \cO_X\longrightarrow \cO_{Z_W}\longrightarrow 0~~
\een  
which induces the long exact sequence: 
\be
0\longrightarrow \J(X,W)\longrightarrow \O(X)\longrightarrow \rH^0(X,\cO_{Z_W})\longrightarrow \rH^1(X,\cJ_W)\longrightarrow \ldots ~~.
\ee
The ideal sheaf $\cJ_W$ is coherent by Cartan's coherence theorem (see
\cite[Section II.4.4]{Demailly}). Since $X$ is Stein, this implies
$\rH^1(X,\cJ_W)=0$ by Cartan's theorem B (see \cite[page 124]{GR}),
thus giving an isomorphism $\Jac(X,W)\stackrel{\sim}{\rightarrow}
\rH^0(X,\cO_{Z_W})$.  Since $Z_W$ consists of a finite set of points,
its structure sheaf is a direct sum of skyscraper sheaves whose
nontrivial stalks are given by the analytic Milnor algebras $\rM({\hW}_p)$.  In
particular, we have $\cO_{Z_W,p}=\rM({\hW}_p)$ for all $p\in Z_W$ and
$\rH^0(X,\cO_{Z_W})=\oplus_{p\in Z_W} \rM({\hW}_p)$.  The map
$\Lambda$ defined above is the resulting isomorphism
$\Jac(X,W)\stackrel{\sim}{\rightarrow} \oplus_{p\in Z_W}
\rM({\hW}_p)$, which gives the conclusion. \qed
\end{proof}

\

\begin{remark}
\label{rem:ac} Since we assume that $X$ is connected, the germ ${\hW}_{p_0}$
 at any fixed critical point $p_0\in Z_W$ determines $W$
globally on $X$ and in particular it determines the germs
 ${\hW}_p$ at all other critical points $p\in Z_W$. Also note that
Tougeron's theorem does {\em not} hold for general non-isolated
singularities. Because of this, the local theories of analytic and
algebraic non-isolated singularities (which are relevant for the
models considered in this paper when $X$ is not Stein and $Z_W$ does
not consist of isolated points) are already rather different; this
difference is of course more pronounced when one considers global
situations. We refer the reader to \cite[page 122]{Arnold} for a
classical example (due to Whitney) of a non-isolated local analytic
singularity which cannot be reduced to algebraic form.
\end{remark}

\subsection{Matrix factorizations of the critical germs of $W$}

For any $p\in Z_W$, let $\MF(\cO_{X,p}, {\hW}_p)$ denote the
$\Z_2$-graded $\cO_{X,p}$-linear differential category of finite rank
matrix factorizations of ${\hW}_p$ over the commutative local ring
$\cO_{X,p}$. Its objects are pairs ${\hat a}=(M,Q)$, where
$M=M^\0\oplus M^\1$ is a $\Z_2$-graded free $\cO_{X,p}$-module whose
even and odd components $M^\0$ and $M^\1$ have finite rank and $Q$ is
an odd endomorphism of $M$ such that $Q^2={\hW}_p\otimes \id_M$. Given
two matrix factorizations ${\hat a}_1=(M_1,Q_1)$ and ${\hat
a}_2=(M_2,Q_2)$ of $\hW_p$, the dg-module
$\Hom_{\MF(\cO_{X,p},{\hW}_p)}({\hat a}_1,{\hat a}_2)$ is the
$\Z_2$-graded $\cO_{X,p}$-module
$\underline{\Hom}_{\cO_{X,p}}(M_1,M_2)$, endowed with the obvious odd
differential induced by $Q_1$ and $Q_2$. Let $\HMF(\cO_{X,p},{\hW}_p)$
be the total cohomology category of $\MF(\cO_{X,p},{\hW}_p)$. The
relations $Q_i^2={\hW}_p\id_{M_i}$ imply that the ideal $\langle
\pd_1{\hW}_p,\ldots, \pd_d{\hW}_p\rangle$ annihilates the
$\Z_2$-graded module $\Hom_{\HMF(\cO_{X,p},{\hW}_p)}({\hat a}_1,{\hat
a}_2)$. The latter can therefore be viewed as a graded module over
the analytic Milnor algebra $\rM({\hW}_p)$. Accordingly,
$\HMF(\cO_{X,p},{\hW}_p)$ can be viewed as an $\rM({\hW}_p)$-linear
$\Z_2$-graded category. Notice that $\MF(\cO_{X,p},{\hW}_p)$ and
$\HMF(\cO_{X,p},{\hW}_p)$ can also be viewed as $\O(X)$-linear
categories using the morphism $\germ_p$ and that
$\HMF(\cO_{X,p},{\hW}_p)$ can also be viewed as a $\Jac(X,W)$-linear
category using the morphism $\Lambda_p$.

\subsection{Presentation of $\HF(X,W)$ through matrix factorizations of the critical germs}

\

\

\begin{Definition}
Let $a=(E,D)$ be a holomorphic factorization of $W$. The {\em germ of
$a$ at the critical point $p\in Z_W$} is the matrix factorization 
${\hat a}_p$ of $\hW_p$ defined through:
\be
{\hat a}_p\eqdef (\cO_{X,p}\otimes_\C E_p, {\hD}_p)~~,
\ee
where $\cO_{X,p}\otimes_\C \!E_p$ is viewed as $\Z_2$-graded free
$\cO_{X,p}$-module of finite rank and ${\hD}_p\in\!
\cO_{X,p}\otimes \!\End_\C(\!E_p)\!\simeq_{\cO_{X,p}} \!\!\End_{\cO_{X,p}\!}(\cO_{X,p}\otimes_\C
E_p)$ denotes the germ of the holomorphic section $D\!\in\!
\rGamma(\!X\!,\!End(\!E)\!)$ at the point $p$.
\end{Definition}

\

\noindent Viewing $\MF(\cO_{X,p},\hW_p)$ as an $\O(X)$-linear
dg-category, consider the dg-functor $\Xi_p:\F(X,W)\rightarrow
\MF(\cO_{X_p},{\hW}_p)$ defined as follows:
\begin{enumerate} \itemsep 0.0em
\item For any $a\in \Ob\F(X,W)$, let $\Xi_p(a)\eqdef {\hat a}_p$\,.
\item For any $a_1=(E_1,D_1),~a_2=(E_2,D_2)\in \Ob\F(X,W)$ and any
$s\in \Hom_{\F(X,W)}(a_1,a_2)=\rGamma(X,Hom(E_1,E_2))$, let
$\Xi_p(s)\eqdef {\hat s}_p$\,.
\end{enumerate}

\noindent Also consider the dg-functor\footnote{Here and below $\vee$ denotes
the coproduct of dg-categories or additive categories, depending on
the context.}  $\Xi:\F(X,W)\rightarrow \vee_{p\in Z_W}
\MF(\cO_{X,p},{\hW}_p)$ given by:
\be
\Xi(a)\eqdef \oplus_{p\in Z_W} \Xi_p(a)~~,~~\Xi(s)\eqdef\oplus_{p\in Z_W}\Xi_p(s)~~.
\ee
We denote by $\Xi_\ast$ the functor induced by $\Xi$ on
cohomology. The relation $D^2=W\id_E$ implies that the critical ideal
of $W$ annihilates the $\O(X)$-modules
$\Hom_{\HF(X,W)}(a_1,a_2)$. This allows us to view $\HF(X,W)$ as
$\Jac(X,W)$-linear category and $\Xi_\ast$ as a $\Jac(X,W)$-linear
functor.

\

\begin{Proposition}
The $\Jac(X,W)$-linear functor: 
\be
\Xi_\ast:\HF(X,W)\rightarrow \vee_{p\in Z_W} \HMF(\cO_{X,p},{\hW}_p)~~.
\ee
is full and faithful. 
\end{Proposition}

\begin{proof} 
Let $a_1=(E_1,D_1)$ and $a_2=(E_2,D_2)$ be two holomorphic
factorizations of $W$. Since the sheaf ${\cal H}om(E_1,E_2)$ of local
holomorphic sections of the bundle $Hom(E_1,E_2)$ is locally free (and
hence flat), the short exact sequence \eqref{cJseq} induces a short
exact sequence of sheaves of $\cO_X$-modules:
\be
0\longrightarrow \cJ_W\otimes_{\cO_X} {\cal H}om(E_1,E_2)\longrightarrow {\cal H}om(E_1,E_2)\longrightarrow \cO_{Z_W}\otimes_{\cO_X} {\cal H}om(E_1,E_2)\longrightarrow 0~~.
\ee
In turn, this induces a short exact sequence in the Abelian category
of unbounded cochain complexes of sheaves of $\cO_X$-modules:
\ben
\label{Cseq}
0\longrightarrow C_1^\bullet \longrightarrow C_2^\bullet\longrightarrow C_3^\bullet \longrightarrow 0~~,
\een
where $C_i^\bullet$ are the 2-periodic complexes with differential $\fd_{a_1,a_2}$ and components given by:
\ben
\label{Cdef}
\!\!C_1^k\eqdef \cJ_W\otimes_{\cO_X} \!{\cal H}om(E_1,E_2)^{\hat k}~,~~C_2^k\eqdef\! {\cal H}om(E_1,E_2)^{\hat k}~,~~C_3^k\eqdef \cO_{Z_W}\otimes_{\cO_X} \!{\cal H}om(E_1,E_2)^{\hat k}~.
\een
Here $k\in \Z$ and ${\hat k}\eqdef k \, \mod \, 2\in \Z_2$. Since the
complex $(C_1^\bullet,\fd_{a_1,a_2})$ is acyclic, the short exact
sequence \eqref{Cseq} gives isomorphisms $\H(C_2^\bullet)\!\stackrel{\sim}{\rightarrow}\!
\H(C_3^\bullet)$, where $\H$ denotes the ``periodic hypercohomology''
of the corresponding complex. Since $X$ is Stein and the sheaves in
\eqref{Cdef} are coherent, Cartan's theorem B gives
$\H(C_2^\bullet)\!\simeq\! \Hom_{\HF(X,W)}(a_1,\!a_2\!)$ and
$\H(C_3^\bullet)\!\simeq\! \oplus_{p\in Z_W}\Hom_{\HMF(\cO_{X,p},\hW_p)}({\hat a}_1,{\hat
a}_2)$. Thus we have an isomorphism of $\Z_2$-graded $\Jac(X,W)$-modules:
\be
\Hom_{\HF(X,W)}(a_1,a_2) \stackrel{\sim}{\rightarrow} \bigoplus_{p\in Z_W}\Hom_{\HMF(\cO_{X,p},\hW_p)}({\hat a}_1,{\hat a}_2)
\ee 
which corresponds to the action of the functor $\Xi$ on
morphisms. This shows that $\Xi_\ast$ is full and faithful. 
\qed
\end{proof}

\subsection{Presentation of the bulk and boundary maps}

Let $\Omega$ be a holomorphic volume form on $X$. The 
germ of $\Omega$ at any critical point $p\in Z_W$ can be written as: 
\be
{\hOmega}_p={\hat \varphi}_p\dd z_p^1\wedge \ldots \wedge \dd z_p^d~~,
\ee
where $\hat\varphi_p\in \cO_{X,p}$ is an invertible holomorphic function germ at $p$. 
For any admissible Hermitian metric $h$ on $E$, the
cohomological disk kernel of the Hermitian holomorphic factorization $\ba=(E,h,D)$
reduces to the following quantity modulo the identifications discussed above:
\be
\Pi_\ba\equiv\frac{\i^d}{d!} \bigoplus_{p\in Z_W}{\det}_{{
\hOmega}_p} (\pd {\hD}_p)= \frac{\i^d}{d!{\hat
\varphi}_p}\epsilon^{i_1\ldots i_d}\bigoplus_{p\in Z_W} \pd_{i_1}{
\hD}_p\ldots \pd_{i_d} {\hD}_p \in \bigoplus_{p\in
Z_W}\End_{\HMF(\cO_{X,p},\hW_p)}({\hat a}_p)~~.
\ee
Moreover, the cohomological boundary-bulk and bulk-boundary
maps of $\ba$ reduce to\footnote{The sign factor in the second
equality above arises from the equality ${\hOmega}_p\lrcorner_0
\det_{{\hOmega}_p} \theta=d!  (-1)^{\frac{d(d-1)}{2}}$, which
itself follows from the relation $(\dd z^1_p\wedge \ldots \wedge \dd
z^d_p)\lrcorner (\pd_1\wedge \ldots \pd_d)=
(-1)^{\frac{d(d-1)}{2}}$.}:
\beqa
&& f_\ba(s)\equiv \frac{\i^d}{d!}\bigoplus_{p\in Z_W}\str\left({\det}_{{\hOmega}_p}(\pd {\hD}_p) \, {\hat s}_p\right)~~\\
&& e_\ba(f)\equiv \i^d (-1)^{\frac{d(d-1)}{2}}\bigoplus_{p\in Z_W} {\hat f}_p \id_{E_p}~~.
\eeqa
Here $f\in \Jac(X,W)=\O(X)/\J(X,W)$ identifies with the family of germs
$({\hat f}_p)_{p\in Z_W}$, while $s\in \End_{\HF(X,W)}(a)$ identifies
with the family of germs $({\hat s}_p)_{p\in Z_W}$. As explained in
the previous sections, the cohomological traces are equal on-shell
with their tempered versions, which are obtained by inserting
appropriate $\lambda$-dependent quantities in the expressions
above. In the limit $\lambda\rightarrow +\infty$, one finds as in
\cite{LG2} that the cohomological bulk trace can be written as:
\ben
\label{Trconj}
\Tr(f)=\sum_{p\in Z_W} A_p\, \Res_p\left[\frac{{\hat f}_p{\hOmega}_p}{{\det}_{{\hOmega}_p}(\pd W)}\right]~~,
\een
while the cohomological boundary trace of the holomorphic
factorization $a=(E,D)$ of $W$ can be expressed as a sum of traces of
generalized Kapustin-Li type (see \cite{LG2,KL}):
\ben
\label{trconj}
\tr_a(s)=\frac{(-1)^{\frac{d(d-1)}{2}}}{d!}\sum_{p\in Z_W} 
A_p\, \Res_p\left[\frac{\str\left({\det}_{{\hOmega}_p}(\pd {\hD}_p) {\hat s}_p\right){\hOmega}_p}
{{\det}_{{\hOmega}_p}(\pd \hW_p)}\right]~~.
\een
Here $A_p$ are normalization constants and $\Res_p$ denotes the
Grothendieck residue on $\cO_{X,p}$. 

\begin{remark}
The quantity ${\det}_{{\hOmega}_p}(\pd {\hD}_p)$ is odd since $D$
is odd. Hence $\str\left({\det}_{{\hOmega}_p}(\pd {\hD}_p) {\hat
s}_p\right)$ vanishes unless ${\hat s}_p$ has $\Z_2$-degree equal to
$\mu(X,W)=d\, \mod \, 2$, in which case
$\str\left({\det}_{{\hOmega}_p}(\pd {\hD}_p) {\hat
s}_p\right)=(-1)^d \str\left({\hat s}_p {\det}_{{\hOmega}_p}(\pd
{\hD}_p)\right)$, where we noticed that $d^2=d \, \mod \, 2$. Thus
\eqref{trconj} can also be written as:
\be
\tr_a(s)=\frac{(-1)^{\frac{d(d+1)}{2}}}{d!}\sum_{p\in Z_W} 
A_p\, \Res_p\left[\frac{\str\left({\hat s}_p {\det}_{{\hOmega}_p}(\pd {\hD}_p)\right){\hOmega}_p}
{{\det}_{{\hOmega}_p}(\pd \hW_p)}\right]~~.
\ee
\end{remark}

\

\noindent The relations above show that, for Stein Landau-Ginzburg
models with finite critical set, non-degeneracy of the cohomological
bulk trace corresponds to that of the Grothendieck trace (which is
well-known), while that of the cohomological boundary traces follows
as in \cite{DM,PV}. Moreover, the topological Cardy
constraint \eqref{Cardy} corresponds to the ``Hirzebruch Riemann-Roch
formula for matrix factorizations'' which was discussed rigorously in
\cite{PV}. The special case $X=\C^d$ with $\Omega=\dd z^1\wedge \ldots \wedge
\dd z^d$ and $W$ a holomorphic function with a single critical point
recovers the formula proposed in \cite{KL} for the cohomological boundary trace.

We note that our point of view on non-degeneracy of traces is rather
different from that adopted in \cite{DM,PV}. In \cite{tserre}, we
prove non-degeneracy of the cohomological bulk and boundary traces in
full generality by starting with expressions \eqref{intOmega} and
\eqref{trdef} of the present paper and adapting the argument used in
\cite{Serre}; this approach also applies to situations when $X$ is not
Stein and the critical locus of $W$ does not consist of isolated
points. The study of bulk and boundary traces is thereby replaced by
the problem of extracting residue formulas generalizing \eqref{Trconj}
and \eqref{trconj} to this much wider setting. This reflects the point
of view (already advocated in \cite{LG2} and also pursued in
\cite{DS3}) that the degenerate limit when the localization parameter
$\lambda$ is taken to infinity obscures the geometry of the problem,
hence this limit should be taken with the sole purpose of extracting
residue descriptions for the worldsheet correlators of the topological
model. In this approach, residue expressions are not taken as the
starting point for defining the model, but as an equivalent
description of its on-shell (i.e., cohomological) data in the
``infrared limit'', a limit in which the worldsheet of the model
``localizes'' on the critical set.

\section{Conclusions and further directions}
\label{sec:conclusions}

We gave a rigorous construction of candidates for off-shell models
of the defining datum of the open-closed topological field theory
defined by a general Landau-Ginzburg pair $(X,W)$, where $X$ is a
non-compact K\"ahlerian manifold and $W$ is any non-constant
holomorphic function defined on $X$, and proved that the axioms of an
open-closed TFT datum defined are satisfied by this datum at the
cohomology level, except for non-degeneracy of the bulk and boundary
traces and for the topological Cardy constraint.  As we show in
\cite{tserre}, non-degeneracy of the bulk and boundary traces can be
proved by appropriately adapting Serre's argument from \cite{Serre} and
combining it with certain spectral sequence arguments. On the other
hand, the topological Cardy constraint can be viewed as a ``deformed''
variant of the Hirzebruch-Riemann-Roch theorem, a subject on which we
hope to report in the future. 

While the framework presented in this paper is both extremely general
and physically fully motivated by the path integral arguments of
\cite{LG1,LG2}, our differential-geometric formulation of the TFT
datum is not fully amenable to explicit computation except in special
situations such as when $X$ is Stein \cite{nlg2}. In the general case
when $X$ is not Stein and the critical locus $Z_W$ is compact but not
reduced to a finite set of points, we expect that the distributional
limit of the tamed bulk and boundary traces introduced in this paper
when the real parameter $\lambda$ is taken to $+\infty$ produces a
current which is closely related to those considered in \cite{AW}. It
would be interesting to fully elucidate this connection.

Another direction in which our approach could be extended concerns the
theory of Landau-Ginzburg models with defects (see \cite{CR,CM}), which
is currently understood globally only when $X$ is an affine
space. Based on the general structure of two-dimensional topological
field theories with defects, one expects the existence of a bicategory
whose objects are K\"ahlerian Landau-Ginzburg pairs $(X,W)$ with
compact critical set and whose category of 1-morphisms from $(X_1,W_1)$ to
$(X_2,W_2)$ is the twisted Dolbeault category of factorizations
$\HDF(X_{12}, W_{12})$ of the LG pair $(X_{12},W_{12})$ defined
through $X_{12}\eqdef X_1\times X_2$, and $W_{12}\eqdef W_2\circ
\pi_2-W_1\circ \pi_1$, where $\pi_i:X_{12}\rightarrow X_i$ are the
canonical projections. The horizontal composition $a_{23}\odot
a_{12}\in \Ob\HDF(X_{13},W_{13})$ of two 1-morphisms $a_{12}\in \Ob
\HDF(X_{12},W_{12})$ and $a_{23}\in \Ob \HDF(X_{23},W_{23})$ should
describe the ``fusion product'' of defects; this could be constructed
using a ``twisted'' analogue of the Fourier-Mukai transform.  The
axiomatic description of open-closed topological field theory with
defects requires the existence of non-degenerate pairings (which
generalize the boundary traces considered in this paper) endowing this
bicategory with a pivotal structure generalizing that found in
\cite{CR,CM}. Such pairings can be constructed starting from the traces of
the present paper and their non-degeneracy could be proved by adapting
the argument of \cite{tserre}. The special sub-bicategory consisting of Stein
Landau-Ginzburg pairs is already quite rich. As simple but already interesting
objects, it includes Landau-Ginzburg pairs for which $X$ is an arbitrary
non-compact Riemann surface, a case which we consider in detail in
\cite{bezout,edd} from the perspective of non-Noetherian
commutative algebra.

Another interesting research direction concerns the problem of
determining the ``D-brane superpotential'' on the category $\HDF(X,W)$
when $(X,W)$ is a general LG pair in the sense of this paper. Since
$\DF(X,W)$ admits the traces $\tr_a$ (which are cohomologically
non-degenerate by the results of \cite{tserre}), the ``non-commutative
Darboux theorem'' of \cite{KS} ensures the existence of a strictly
cyclic minimal $A_\infty$ structure on $\HDF(X,W)$ which makes the
latter into a quasi-equivalent model of the dg-category $\DF(X,W)$.
When equipped with this $A_\infty$ structure, $\HDF(X,W)$ can be
interpreted as an open topological string field theory as in
\cite{HLL}, whose string field action plays the role of the ``D-brane
superpotential''. The D-brane
superpotential encodes the crucial information relevant for
homological mirror symmetry. More precisely, we expect that the mirror
of a general LG pair $(X,W)$ should be a (generally non-toric)
manifold $\breve{X}$, which has the property that the strictly cyclic
minimal model of its Fukaya category is quasi-equivalent with the
strictly cyclic minimal model of $\HDF(X,W)$. As in \cite{CIL}, this
amounts to existence of an equivalence of BV systems between the
corresponding topological open string field theories. Notice, however,
that the explicit computation of a strictly cyclic minimal model of
$\HDF(X,W)$ (and hence of the D-brane superpotential) is a non-trivial
problem already in the case when $X$ is Stein, given that the direct
construction of such a minimal model requires the choice of a special
cohomological splitting when applying the homological perturbation
lemma\footnote{See \cite{Nils} for some explicit results of this
type.}. While this lies well outside the scope of the present paper,
we mention that a different approach (which views the model
constructed in the present paper as a deformation of the open-closed
datum of the B-type topological sigma model with non-compact target
$X$ obtained by ``turning on'' the Landau-Ginzburg superpotential $W$)
may ultimately prove to be more fruitful computationally.

We also mention that the models considered in this paper (which, when $X$ is not Stein,
allow the critical set of $W$ to be non-discrete) may allow one to
give an effective description to at least some of the ``hybrid
models'' of \cite{Witten}. Should this indeed be the case, they 
could help clarify certain aspects of the dynamics of linear sigma
models which have remained rather mysterious until now.

Finally, we mention that the compactness condition on the critical
locus of $W$ is motivated by our requirement that the on-shell bulk
and boundary state spaces of the topological field theory be
finite-dimensional. Relaxing this condition requires that one views
the latter as topological vector spaces spaces, for which a version of
non-degeneracy of the bulk and boundary traces can still be proved
upon appropriately generalizing the argument of \cite{tserre}. In
order to pursue such an extension, one must however first address
the question of the proper mathematical definition of open-closed TFT
data in the infinite-dimensional context, a problem which has not yet
been addressed in the literature.


\appendix

\section{Some expressions in local holomorphic coordinates}
\label{app:coord}

\noindent Choosing local holomorphic coordinates $z=(z_1,\ldots, z_d)$
defined on an open subset $U\subset X$, set $\pd_i:=\frac{\pd}{\pd
  z_i}$ and $\bpd_i:=\frac{\pd}{\pd \bar z_i}$. We have:
\beqa 
&& T X|_{U}~={\rm Span}_\C \Big\{ \pd_1,\ldots, \pd_d\Big\}
 ~~~,~~~ \overline{T}X|_{U}~={\rm Span}_\C \Big\{\bpd_1,\ldots, \bpd_d\Big\}~~\\
&& T^\ast X|_{U}={\rm Span}_\C \{  \dd z_1,\ldots,  \dd z_d\}  ~~,
~~ \overline{T}^\ast X|_{U}={\rm Span}_\C \{  \dd \bar z_1,\ldots,  \dd \bar z_d\}~~. 
\eeqa
For any increasingly ordered subset $\{t_1,\ldots,t_i\}\subset
\{1,\ldots, d\}$ of size $|I|=i$, define:
\be
\dd z_I := \dd z_{t_1}\wedge \dd z_{t_2}\wedge\dots \wedge \dd  z_{t_i}~~,
~~~\pd_I:= \pd_{t_1}\wedge \ldots \wedge\pd_{t_i}
\ee
and their complex conjugates: 
\be
\dd\bar z_I := \dd\bar z_{t_1}\wedge \dd\bar z_{t_2}\wedge\dots \wedge \dd\bar z_{t_i}~~,
~~~\bpd_I:= \bpd_{t_1}\wedge \ldots \wedge\bpd_{t_i}~~.
\ee
For $\omega\in \PV^{-i,j}(X)$ and $\eta\in \PV^{-k,l}(X)$, we can expand:
\be
\omega =_U \sum_{|I|=i,|J|=j} \omega^I_{~J} \dd \bar z_J\otimes \pd_I~~,
~~~\eta =_U \sum_{|K|=k,|L|=l} \eta^K_{~L} \dd \bar z_L\otimes \pd_K
\ee
with $\omega^I_{~J} ,\eta^K_{~L} \in C^\infty(X)$. Then:
\be
\omega\wedge\eta=_U\sum_{I,J,K,L}(-1)^{il}\omega^I_{~J} \eta^K_{~L} 
(\dd \bar z_J\wedge\dd \bar z_L)\otimes(\pd_I\wedge\pd_K)~~,
\ee
where $|I|=i$, $|J|=j$, $|K|=k$ and $|L|=l$. We also have:
\be
\bpd \omega=_U\sum_{|I|=i,|J|=j}[(\bpd \omega^I_J)\wedge \dd \bar z_J] \otimes\pd_I
=\sum_{|I|=i,|J|=j}\sum_{r=1}^d(\bpd_r \omega^I_J)(\dd \bar z_r\wedge\dd\bar z_J)\otimes \pd_I\in \PV^{-i,j+1}(X)~~.
\ee
Let $\Omega$ be a holomorphic volume form on $X$. Then:
\be
\Omega=_U \varphi(z)\dd z^{1}\wedge \ldots \wedge \dd z^{d}~~
\ee
for some $\varphi\in \cO_X(U)$. The following local expansion holds
for any $\omega\in \PV^{-i,j}(X)$ (see \cite{LLS}):
\ben
\label{pdOmega}
\pd_\Omega\omega=_U\frac{1}{\varphi}\sum_{|I|=i,|J|=j}~\sum_{r=1}^d [\pd_r (\varphi \omega^I_J)](\dd z^r)\lrcorner 
(\dd \bar z_J \otimes \pd_I)~\in \PV^{-i+1,j}(X)~~.\\
\een
Let us prove relation \eqref{ioda_bracket}. Using \eqref{pdOmega} in
\eqref{OmegaBracket} and noticing that $\pd_\Omega W=0$ due to degree reasons, we compute:
\beqa
\{W,\omega\}_{_\Omega}&=_U& \pd_{\Omega}(W \omega)-(\pd_{\Omega}W)\wedge
\omega-(-1)^{|W|}W \pd_{\Omega}\omega=\pd_{\Omega}(W \omega)-W \pd_\Omega\omega= \\
&=&\frac{1}{\varphi}\sum_{|I|=i,|J|=j}\sum_{r=1}^d \Big[ \pd_r (\varphi W\omega^I_J)(\dd z^r) \lrcorner
(\dd \bar z_J \otimes\pd_I)- W \pd_r (\varphi \omega^I_J) (\dd z^r)\lrcorner (\dd \bar z_J \otimes\pd_I)  \Big] \\
&=&\frac{1}{\varphi} \sum_{|I|=i,|J|=j}\sum_{r=1}^d(\varphi (\pd_r W)\omega^I_J)(\dd z^r)\lrcorner  (\dd \bar z_J \otimes\pd_I) 
= \sum_{|I|=i,|J|=j}\sum_{r=1}^d(\pd_r W)(\dd z^r)\lrcorner \omega~~.
\eeqa
Recall that:
\be
\ioda_W=-\i\,\iota_{\pd W}=_U-\i\sum_{r=1}^d (\partial_r W)\dd z^r\lrcorner~~,
\ee
where we used the expansion $\pd W=_U\sum_{r=1}^d (\partial_r W)\dd z^r$~.
Comparing with the expression above gives \eqref{ioda_bracket}.

\section{Superconnection formulation of Landau-Ginzburg TFT data}
\label{app:sc}

\noindent The path integral argument of references \cite{LG1,LG2} is formulated
using a certain class of superconnections, leading to a
description of the category of topological D-branes which is
equivalent with that given in Section \ref{sec:DF}. The
equivalence of the two descriptions follows from the
Koszul-Malgrange correspondence, as we explain in this appendix.

\subsection{The Koszul-Malgrange correspondence}
\label{subsec:KM}
Recall that $\VB_\sm(X)$ denotes the $\cinf$-linear category of
finite rank locally-free sheaves of $\cC^\infty$-modules (whose
objects we identify with smooth vector bundles defined on $X$), while
$\VB(X)$ denotes the $\O(X)$-linear category of finite rank
locally-free sheaves of $\cO_X$-modules (whose objects we identify
with holomorphic vector bundles defined on $X$). The latter category
has a well-known description in terms of smooth complex vector bundles
endowed with integrable $(0,1)$-connections (see \cite{KM,Pali}), as we
recall below.

Given a complex vector bundle $S$ on $X$, a {\em $(0,1)$-connection}
on $S$ is a $\C$-linear map $\cD:\rGamma_\sm(X,S)=\cA^0(X,S)\rightarrow
\cA^1(X,S)$ which satisfies the Leibniz rule:
\be
\cD(f s)=(\bpd f)\otimes s+f (\cD s)~~,~~\forall f\in \cinf~,~\forall s\in \rGamma_\sm( X,S)~~.
\ee
A {\em Dolbeault derivation} is a $\C$-linear map
$\bbpd:\cA(X,S)\rightarrow \cA(X,S)$ which is homogeneous of degree
one with respect to the rank grading of $\cA(X,S)$ and satisfies:
\be
\bbpd(\rho \otimes s)=(\bpd \rho)\wedge s+(-1)^k\rho \otimes (\bbpd s)~~,
~~\forall \rho \in \cA^k(X)~,~\forall s\in \rGamma_\sm(X,S)~~.
\ee
Any Dolbeault derivation on $S$ restricts to a $(0,1)$-connection.
Conversely, any $(0,1)$-connection on $S$ extends uniquely to a Dolbeault
derivation. This correspondence gives a bijection between $(0,1)$-connections 
on $S$ and Dolbeault derivations on $S$.

A $(0,1)$-connection on $S$ is called {\em integrable} if its
Dolbeault derivation $\bbpd$ is a differential (i.e. if it satisfies $\bbpd^2=0$). 
A {\em Dolbeault pair} defined on $X$ is a pair $(S,\bbpd)$, where $S$ is a complex vector
bundle on $X$ and $\bbpd$ is an integrable Dolbeault derivation
on $S$.  Given two such pairs $(S_1,\bbpd_1)$ and $(S_2,\bbpd_2)$, a
{\em morphism of Dolbeault pairs} from $(S_1,\bbpd_1)$ to
$(S_2,\bbpd_2)$ is a morphism 
$f:S_1\rightarrow S_2$ of $\VB_\sm(X)$ such that $\bbpd_2\circ (\id_{\cA(X)}\otimes
f)=(\id_{\cA(X)}\otimes f)\circ \bbpd_1$. With these definitions,
Dolbeault pairs over $X$ and morphisms of such form a category which
we denote by $\cP(X)$.

Given a Dolbeault pair $(S,\bbpd)$, the locally-free sheaf of
$\cO_X$-modules $\cO_X(S,\bbpd)$ defined by $\cO_X(S,\bbpd)(U)\eqdef
\ker \bbpd|_{\rGamma_\sm(U,\cS)}$ for any open subset $U\subset X$
is called the sheaf of $\bbpd$-holomorphic sections of $S$. This
determines a holomorphic vector bundle $\cK(S,\bbpd)$ whose underlying
complex vector bundle equals $S$. A morphism of Dolbeault pairs
$f:(S_1,\bbpd_1)\rightarrow (S_2,\bbpd_2)$ induces a morphism between
the corresponding sheaves of $\bbpd$-holomorphic sections
$\cK(f):\cO(S_1,\bbpd_1)\rightarrow \cO(S_2,\bbpd_2)$, i.e. a morphism
$\cK(f)\in \Hom_{\VB(X)}(\cK(S_1,\bpd_1),\cK(S_2,\bpd_2))$ in the
category $\VB(X)$. This correspondence gives a functor
$\cK:\cP(X)\rightarrow \VB(X)$. Let $\Phi:\VB(X)\rightarrow
\VB_\sm(X)$ be the functor which forgets the holomorphic structure
and $\Psi:\cP(X)\rightarrow \VB_\sm(X)$ be the functor which
forgets the Dolbeault derivation.  We also have a functor
$\cM:\VB(X)\rightarrow \cP(X)$ which sends a holomorphic vector bundle
$E$ into the Dolbeault pair $\cM(E)\eqdef (\Phi(E), \bbpd_E)$ (where
$\bbpd_E$ is the Dolbeault derivation defined on $\Phi(E)$ by the
holomorphic structure of $E$) and sends a morphism $f$ of $\VB(X)$
into the underlying morphism $\cM(f)\eqdef \Phi(f)$ of complex vector
bundles. This functor satisfies $\Psi=\Phi\circ \cK$.  The
Koszul-Malgrange theorem (see \cite{KM,Pali}) states that $\cK$ and $\cM$
are mutually quasi-inverse equivalences of categories between $\cP(X)$
and $\VB(X)$.  In particular, the fiber of $\Phi$ at a complex vector
bundle $S$ can be identified with the set of all integrable $(0,1)$
connections on $S$.

\subsection{Complex vector superbundles}

\

\

\begin{Definition}
A {\em complex vector superbundle} on $X$ is a complex $\Z_2$-graded
vector bundle, i.e. a complex vector bundle $S$ endowed with a direct
sum decomposition $S=S^\0\oplus S^\1$, where $S^\0$ and $S^\1$ are complex
linear sub-bundles of $S$. 
\end{Definition}

\

\begin{Definition}
Let $\VB_\sm^s(X)$ denote the $\Z_2$-graded $\cinf$-linear category
defined as follows:
\begin{itemize}
\item The objects are the complex vector superbundles defined on $X$.
\item Given two complex vector superbundles $S$ and $T$ over $X$, the
  space of morphisms from $S$ to $T$ is the $\cinf$-module
  $\Hom_\sm(S,T)\eqdef \rGamma_\sm(X,Hom(S,T))$, endowed with
  the $\Z_2$-grading with homogeneous components:
\beqa
&& \Hom^\0_\sm(S,T)\eqdef \rGamma_\sm(X, Hom(S^\0,T^\0))\oplus \rGamma_\sm(X,Hom(S^\1,T^\1))~~\nn\\
&& \Hom^\1_\sm(S,T)\eqdef \rGamma_\sm(X, Hom(S^\0,T^\1))\oplus \rGamma_\sm(X, Hom(S^\1,T^\0))~~.
\eeqa
\item The composition of morphisms is induced by that of
  $\VB_\sm(X)$.
\end{itemize}
\end{Definition}

\

\noindent Let $\VB^{s,\0}_\sm(X)$ be the non-full subcategory of
$\VB_\sm^s(X)$ obtained by restricting to even morphisms. We
have an obvious equivalence of categories:
\be
\VB_\sm^{s,\0}(X)\simeq \VB_\sm(X)\times \VB_\sm(X)~~.
\ee
The {\em graded direct sum} of $S$ with $T$ is the direct sum $S\oplus
T$ of the underlying complex vector bundles, endowed with the
$\Z_2$-grading given by:
\be
(S\oplus T)^\kappa=S^\kappa\oplus T^\kappa ~~,~~\forall \kappa\in \Z_2~~. 
\ee
The {\em graded tensor product} of $S$ with $T$ is the ordinary tensor
product $S\otimes T$ of the underlying complex vector bundles, endowed
with the $\Z_2$-grading given by:
\beqa
& & (S\otimes T)^\0\eqdef (S^\0\otimes T^\0)\oplus (S^\1\otimes T^\1)~~\\
& & (S\otimes T)^\1\eqdef (S^\0\otimes T^\1)\oplus (S^\1\otimes T^\0)~~.
\eeqa
The {\em graded dual} of $S$ is the
complex vector superbundle whose underlying complex vector bundle is
the ordinary dual $S^\vee$ of $S$, endowed with the
$\Z_2$-grading given by:
\be
(S^\vee)^\kappa\eqdef (S^\kappa)^\vee ~~,~~\forall \kappa\in \Z_2~~. 
\ee
The {\em graded bundle of morphisms} from $S$ to $T$ is the bundle
$Hom(S,T)\eqdef S^\vee\otimes T$, where $S^\vee$ is the graded dual
and $\otimes$ is the graded tensor product. Thus $Hom(S,T)$ is the
usual bundle of morphisms, endowed with the $\Z_2$-grading given by:
\beqa
&& Hom^\0(S,T)\eqdef Hom(S^\0,T^\0)\oplus Hom(S^\1,T^\1)~~\nn\\
&& Hom^\1(S,T)\eqdef Hom(S^\0,T^\1)\oplus Hom(S^\1,T^\0)~~.
\eeqa
We have $\Hom^\kappa_\sm(S,T)=\rGamma_\sm(X, Hom^\kappa(S,T))$.
When $T=S$, we set $End(S)\eqdef Hom(S,S)$ and $End^\kappa(S)\eqdef
Hom^\kappa(S,S)$ etc.

\paragraph{The module of $S$-valued $(0,\ast)$-forms}

For any complex vector superbundle $S=S^\0\oplus S^\1$ on $X$, the
$\cinf$-module of smooth sections $\rGamma_\sm(X,S)$ is
$\Z_2$-graded with homogeneous components $\rGamma_\sm^\kappa(M,S)\eqdef
\rGamma_\sm(M,S^\kappa)$. Accordingly, the $\cinf$-module
$\cA(X,S)\eqdef \cA(X)\otimes_{\cinf} \rGamma_\sm(X,S)$ has an
induced $\Z\times \Z_2$-grading. The $\Z$-grading is called the {\em
  rank grading} and corresponds to the decomposition:
\be
\cA(X,S)=\bigoplus_{k=0}^d \cA^k(X,S)~~.
\ee
The $\Z_2$-grading is called the {\em bundle grading} and 
corresponds to the decomposition:
\be
\cA(X,S)=\cA(X,S^\0)\oplus \cA(X,S^\1)~~.
\ee
For any $\alpha\in \cA^k(X,S)$, let $\rk\alpha\eqdef k$. For 
any $\alpha\in \cA(X,S^\kappa)$, let $\sigma(\alpha)\eqdef \kappa\in \Z_2$. 
The {\em total grading} of $\cA(X,S)$ is the $\Z_2$-grading given by
the decomposition $\cA(X,S)=\cA(X,S)^\0\oplus \cA(X,S)^\1$, where: 
\beqa
& &\cA(X,S)^\0\eqdef \bigoplus_{i=\ev} \cA^i(X,S^\0)\oplus \bigoplus_{i=\odd} \cA^i(X,S^\1)~~\\
& &\cA(X,S)^\1\eqdef \bigoplus_{i=\ev} \cA^i(X,S^\1)\oplus \bigoplus_{i=\odd} \cA^i(X,S^\0)~~.
\eeqa
For any $\alpha\in \cA(X,S)^\kappa$, we set $\deg\alpha=\kappa\in
\Z_2$. For all $\rho\in \cA^k(X)$ and all $s \in \rGamma(M,S^\kappa)$,
we have:
\begin{equation}
\deg(\rho\otimes s) = {\hat k}+ \kappa \in \Z_2~~.
\end{equation}

\subsection{$\Z_2$-graded Dolbeault pairs}

\

\

\begin{Definition}
Let $S$ be a complex vector superbundle. A Dolbeault derivation
$\bbpd$ on $S$ is called {\em compatible} with the $\Z_2$-grading of
$S$ if it preserves the bundle grading of $\cA(X,S)$, which
means that it satisfies:
\be
\bbpd(\cA(X,S^\kappa))\subset \cA(X,S^\kappa)~~,~~\forall \kappa\in \Z_2~~.
\ee
\end{Definition}

\

\noindent A compatible Dolbeault derivation decomposes as
$\bbpd=\bbpd^\0\oplus \bbpd^\1$, where $\bbpd^\kappa$ are Dolbeault
derivations on $S^\kappa$ for each $\kappa\in \Z_2$. Conversely, given
Dolbeault derivations $\bbpd^\kappa$ on $S^\kappa$, their direct sum
is a compatible Dolbeault derivation on $S$.

\

\begin{Definition}
A {\em $\Z_2$-graded Dolbeault pair} on $X$ is a pair $(S,\bbpd)$, where $S$
is a complex vector superbundle on $X$ and $\bbpd$ is an integrable
Dolbeault derivation on $S$ which is compatible with the
$\Z_2$-grading of $S$.  
\end{Definition}

\

\begin{Definition}
The {\em category $\cP^s(X)$ of $\Z_2$-graded Dolbeault pairs} on $X$ is
the  $\O(X)$-linear category defined as follows:
\begin{itemize}
\itemsep 0.0em
\item The objects are the $\Z_2$-graded Dolbeault pairs defined on $X$.
\item Given two $\Z_2$-graded Dolbeault pairs $(S_1,\bbpd_1)$ and
  $(S_2,\bbpd_2)$ on $X$, the corresponding Hom space is the
  $\cinf$-module:
\be
\Hom_{\cP^s(X)}((S_1,\bbpd_1),\!(S_2,\bbpd_2))\!\eqdef 
\!\!\big\{f\!\in \rGamma_\sm(X,Hom^\0(S_1,S_2))\big|\bbpd_2\circ (\id_{\cA(X)}\otimes f)=(\id_{\cA(X)}\otimes f)\circ \bbpd_1\big\}~~
\ee
\item The composition of morphisms is inherited from $\VB_\sm^s(X)$.
\end{itemize}
\end{Definition}

\

\noindent Recall the Koszul-Malgrange functor $\cK$ defined in
Subsection \ref{subsec:KM}. Let $\cK^s:\cP^s(X)\rightarrow
\VB^{s,\0}(X)$ be the functor
which:
\begin{itemize}
\itemsep 0.0em
\item Sends a $\Z_2$-graded Dolbeault pair $(S,\bbpd)$ into the
  holomorphic vector bundle $\cK(S,\bbpd)$, viewed as a holomorphic
  vector superbundle by endowing it with the grading
  $\cK(S,\bpd)^\kappa\eqdef \cK(S^\kappa,\bpd^\kappa)$, where
  $\kappa\in \Z_2$
\item Sends a degree zero morphism $f:(S_1,\bbpd_1)\rightarrow
  (S_2,\bbpd_2)$ of $\Z_2$-graded Dolbeault pairs which intertwines
  $\bbpd_1$ and $\bbpd_2$ into $\cK(f)\eqdef f$, where $f$ is viewed
  as a degree zero morphism of locally-free sheaves of
  $\cO_X$-supermodules.
\end{itemize}
The Koszul-Malgrange theorem implies that $\cK^s$ is an equivalence of
$\O(X)$-linear categories between $\cP^s(X)$ and $\VB^{s,\0}(X)$. This
gives equivalences of $\O(X)$-linear categories:
\ben
\label{sKM}
\cP^s(X)\simeq \VB^{s,\0}(X)\simeq \VB(X)\times \VB(X)~~.
\een

\subsection{Dolbeault superconnections}

The following definition is a variant of the concept of 
superconnection due to Quillen (see \cite{Quillen}).

\

\begin{Definition} 
Let $S$ be a complex vector superbundle on $X$. A \emph{Dolbeault
  superconnection} on $S$ is a $\C$-linear map
$\fD:\cA(X,S)\rightarrow \cA(X,S)$ such that:
\begin{enumerate}[1.]
\itemsep 0.0em
\item $\fD$ is homogeneous of odd degree with respect to the total
  $\Z_2$-grading of $\cA(X,S)$, i.e.:
\be
\fD(\cA(X,S)^\kappa)\subset \cA(X,S)^{\kappa+1}~,~~\forall \kappa\in \Z_2~~.
\ee
\item The following condition is satisfied for all $k=0,\ldots, d$, all
  $\rho\in \cA^k(X)$ and all $\alpha \in \cA(X,S)$:
\be 
\fD(\rho\wedge \alpha) = (\bpd \rho)\wedge \alpha + (-1)^k \rho \wedge (\fD \alpha)~~.
\ee
\end{enumerate}
\end{Definition}

\noindent The difference of two Dolbeault superconnections defined on
the same complex vector superbundle $S$ is an endomorphism of the
$\cA(X)$-supermodule $\cA(X,S)$ which is odd with respect to the total
grading of the latter, i.e. an element of $\cA(X,End(S))^\1$. Thus:

\

\begin{Proposition}
The space of Dolbeault superconnections on a complex vector
superbundle $S$ is an affine space modeled on the vector space
$\cA(X,End(S))^\1$.
\end{Proposition}

\

\begin{Definition}
Let $S$ be a complex vector superbundle on $X$. A Dolbeault
superconnection on $S$ is called {\em flat} if it satisfies the
condition $\fD^2=0$. A {\em Dolbeault superpair} is a pair $(S,\fD)$,
where $S$ is a complex vector superbundle. A Dolbeault superpair is called 
flat if $\fD$ is a flat Dolbeault superconnection on $S$.
\end{Definition}

\subsection{Diagonal, flat and special Dolbeault superconnections}

\

\

\begin{Definition}
A Dolbeault superconnection on a complex vector superbundle $S$
defined on $X$ is called {\em diagonal} if it is homogeneous of degree
one with respect to the rank grading of $\cA(X,S)$. A Dolbeault
superpair $(S,\fD)$ is called {\em diagonal} if $\fD$ is a diagonal
Dolbeault superconnection.
\end{Definition}

\

\noindent Notice that a diagonal Dolbeault superconnection on $S$ is
the same as a Dolbeault derivation on $S$ which is compatible with the
$\Z_2$-grading of $S$. Diagonal and flat Dolbeault superpairs can be identified
with $\Z_2$-graded Dolbeault pairs as follows.  Any flat and diagonal
Dolbeault superconnection on $S$ decomposes as
$\fD=\bbpd^{\0}\oplus\bbpd^{\1}$, where $\bbpd^\kappa$ are integrable
Dolbeault derivations on the complex vector sub-bundles
$S^\kappa$. Thus $\fD$ is a compatible and integrable Dolbeault
derivation on $S$ and $(\cS,\fD)$ can be viewed as a $\Z_2$-graded
Dolbeault pair. Conversely, any compatible and integrable Dolbeault
derivation on $S$ can be viewed as a flat and diagonal Dolbeault
superconnection. Accordingly, flat and diagonal Dolbeault superpairs can be
identified with holomorphic vector superbundles.

\subsection{Special Dolbeault factorizations}

Let $(X,W)$ be a Landau-Ginzburg pair.

\

\begin{Definition}
A Dolbeault superconnection $\fD$ on a complex vector superbundle $S$
defined on $X$ is called {\em special} if it can be written in the
form:
\ben
\label{cBdec}
\fD=\fD_0 + \id_{\cA(X)}\otimes D~~,
\een
where $\fD_0$ is a diagonal Dolbeault superconnection on $S$, $D\in
\rGamma_\sm(X,End^\1(S))$ is a globally-defined smooth odd
endomorphism of $S$ and $\otimes$ is the tensor product 
taken over $\cinf$. In this case, $\fD_0$ and $D$ are uniquely
determined by $\fD$ and are called respectively the {\em diagonal} and
{\em off-diagonal} parts of $\fD$, while the decomposition
\eqref{cBdec} is called the {\em canonical decomposition} of $\fD$.  A
Dolbeault superpair $(S,\fD)$ is called {\em special} if $\fD$ is a
special Dolbeault superconnection on $S$.
\end{Definition}

\

\begin{Definition}
A {\em special Dolbeault factorization} of $W$ is a special Dolbeault
superpair $(S,\fD)$ such that $\fD^2=W\id_{\cA(X,S)}$.
\end{Definition}

\

\noindent Special Dolbeault factorizations can be identified with
holomorphic factorizations as follows. Let $(S,\fD)$ be a special
Dolbeault factorization of $W$ with diagonal and off-diagonal parts
$\fD_0$ and $D$. Separating ranks in the condition
$\fD^2=W\id_{\cA(X,S)}$ gives the system:
\ben
\left\{
\begin{aligned}
&  \fD_0^2=0\\
&  \fD_0\circ (\id_{\cA(X)}\otimes D)+(\id_{\cA(X)}\otimes D)\circ \fD_0=0~~\\
&  D^2=W\id_S~~.
\end{aligned}
\right.
\een
In particular, $\fD_0$ is a flat diagonal Dolbeault superconnection on
$S$ which we view as a compatible and integrable Dolbeault derivation as explained
above. Let $E\eqdef \cK^s(S,\fD_0)$ be the holomorphic vector
superbundle defined by $(S,\fD_0)$. Then $\fD_0$ coincides with the
(compatible) Dolbeault operator $\bbpd_E$ of $E$ and the second
condition above amounts to:
\ben
\label{Dholo}
\bbpd_E^\ad(D)=0~,
\een
where $\bbpd_E^\ad=\bbpd_{End(E)}$ is the Dolbeault derivation of the
holomorphic vector bundle $End(E)$. The condition above means that
$D$ is a holomorphic section of $E$. Thus $(E,D)$ is a holomorphic
factorization of $W$. By the Koszul-Malgrange correspondence, any
holomorphic factorization of $W$ can be obtained in this manner.

\subsection{The twisted category of special Dolbeault factorizations}

Let $\fD_\sm^s(X)$ denote the $\cinf$-linear $\Z\times \Z_2$-graded
category whose objects are the complex vector superbundles defined on
$X$ and whose morphism spaces are given by:
\be
\Hom_{\fD^s_\sm}(S_1,S_2)=\cA(X,Hom(S_1,S_2))~~,
\ee
where the $\cinf$-bilinear composition of morphisms
$\circ\!:\!\cA(X,\!Hom(S_2,\!S_3))\times \cA(X,\!Hom(S_1,\!S_2))\!\rightarrow
\cA(X,Hom(S_1,S_3))$ is determined uniquely through the condition:
\be
(\rho\otimes f)\circ (\eta\otimes g)=(-1)^{\sigma(f)\rk \eta}(\rho\wedge \eta) \otimes (f\circ g)~~
\ee
for all pure rank forms $\rho,\eta\in \cA(X)$ and all pure degree
elements $f\in \rGamma(X,Hom(S_2,S_3))$ and $g\in \rGamma(X,
Hom(S_1,S_2))$. Given two special Dolbeault factorizations
$(S_i,\fD_i)$ ($i=1,2$) of $W$ (where $\fD_i$ have diagonal and off-diagonal parts 
denoted respectively $\fD_{i0}$ and $D_i$), the space $\cA(X,Hom(S_1,S_2))$
carries two natural differentials:
\begin{itemize}
\itemsep 0.0em
\item The {\em Dolbeault differential}
  $\bbpd:=\bbpd_{(S_1,\fD_1),(S_2,\fD_2)}$ is the unique $\O(X)$-linear
  map which satisfies:
\be
\bbpd \alpha =\fD_{20}\circ \alpha-(-1)^{\deg \alpha}\alpha\circ \fD_{10}~~
\ee
for elements $\alpha\in \cA(X,Hom(S_1,S_2))$ of pure total degree. 
This differential preserves the bundle grading and has
degree $+1$ with respect to the rank grading. 
\item The {\em defect differential } is the $C^\infty\!(\!X\!)$-linear
endomorphism $\md\!:=\!\md_{(\!S_1,\fD_1\!),(\!S_2,\fD_2\!)}$ of $\cA(\!X\!,\!Hom(\!S_1,\!S_2\!)\!)$
which is determined uniquely by the condition:
\be
\md (\rho\otimes f)=(-1)^{\rk \rho} \rho \otimes (D_2\circ f)-(-1)^{\rk \rho +\sigma(f)} \rho\otimes (f\circ D_1)
\ee
for all pure rank forms $\rho \in \cA(X)$ and all pure degree elements
$f\in \rGamma_\sm(X,Hom(S_1,S_2))$. This endomorphism squares to zero
because $D_i$ squares to $W_i\id_{S_i}$. Moreover, this differential
preserves the rank grading and is odd with respect to the bundle
grading. 
\end{itemize}
Since $D_1$ and $D_2$ are holomorphic, these two differentials anticommute. Thus:
\be
\bbpd^2=\md^2=\bbpd\circ\md+\md \circ \bbpd=0~~.
\ee
The {\em twisted differential} $\updelta:=\updelta_{(S_1,\fD_1),(S_2,\fD_2)}$ is the 
total differential of this bicomplex: 
\be
\updelta\eqdef \bbpd+\md~~.
\ee
The twisted differential is odd with respect to the total $\Z_2$-grading. 

\

\begin{Definition}
The {\em twisted category $\DF_\sm(X,W)$ of special
  Dolbeault factorizations} is the $\Z_2$-graded $\cinf$-linear
  dg-category defined as follows:
\begin{itemize}
\itemsep 0.0em
\item The objects are the special Dolbeault factorizations of $W$.
\item For any two special Dolbeault factorizations $(S_1,\fD_1)$ and
  $(S_2,\fD_2)$ of $W$, the space of morphisms from $(S_1,\fD_1)$ to
  $(\cS_2,\fD_2)$ is the $\cinf$-module: \be
  \Hom_{\DF_\sm(X,W)}((S_1,\fD_1),(S_2,\fD_2))\eqdef
  \cA(X,Hom(S_1,S_2))~~, \ee endowed with the total grading and with
  the twisted differential $\updelta_{(S_1,\fD_1),(S_2,\fD_2)}$.
\item The composition of morphisms is that of $\fD^s_\sm(X)$.
\end{itemize}
\end{Definition}

\subsection{Equivalence of $\DF_\sm(\!X,\!W\!)$ with $\DF(\!X,\!W\!)$}

\noindent Recall the twisted Dolbeault category $\DF(X,W\!)$ of
holomorphic factorizations introduced in Subsection \ref{sec:DF}.
Let $\K:\DF_\sm(X,W)\rightarrow \DF(X,W)$ be the functor defined as
follows:
\begin{itemize}
\itemsep 0.0em
\item Given a special Dolbeault factorization $(S,\fD)$ of $W$, let
  $\K(S,\fD)\eqdef (\cK^s(S,\fD_0),D)$, where $\fD_0$ and $D$ are the diagonal and off-diagonal parts of $\fD$.
\item Given $\alpha\in \Hom_{\DF_\sm(X,W)}((S_1,\fD_1),
  (S_2,\fD_2))\!=\!\cA(X,Hom(S_1,S_2))$, let $\K(\alpha)\!\eqdef\! \alpha$,
  viewed as an element of
  $\Hom_{\DF(X,W)}(\cK^s(S_1,\fD_{10}),\cK^s(S_2,\fD_{20}))\!=\!\cA\big(X,
  Hom(\cK^s(S_1,\fD_{10}), \cK^s(S_2,\fD_{20}))\big)$.
\end{itemize}

\noindent The proof of the following statement is now obvious: 

\

\begin{Proposition}
\label{prop:DFDFinf}
The dg-functor $\K$ is an equivalence of dg-categories between the
twisted category $\DF_\sm(X,W)$ of special Dolbeault factorizations
and the twisted Dolbeault category $\DF(X,W)$ of holomorphic
factorizations.
\end{Proposition}

\subsection{Extension of the boundary coupling to compact oriented surfaces with corners}

For completeness, we outline briefly how the Lagrangian formulation given in
\cite{LG2} extends from compact oriented surfaces with boundary to
compact oriented surfaces with corners. Let $\Sigma$ be a compact
oriented surface with corners.  Let $\Int \Sigma$ denote the interior
of $\Sigma$ and $C\eqdef \Sigma\setminus \Int\Sigma$ denote the
topological frontier. Let $(C_s)_{s=1,\ldots, n}$ denote the connected
components of $C$, where $n\eqdef \Card \pi_0(C)$. Notice that each
$C_s$ is homeomorphic with a circle. Let
$V_s\subset C_s$ denote the set of those corners of $\Sigma$ which lie
on $C_s$. Giving $C_s$ the orientation induced from that of $\Sigma$,
we pick an element $v_s^0 \in V_s$ and enumerate $V_s$ starting from
$v_s^0$ increasingly with respect to the orientation of $C_s$. This
gives an enumeration $V_s={v_s^0,\ldots, v_s^{m_s-1}}$, where
$m_s\eqdef \Card V_s$. For any $k\in \Z$, let $[k]_s\eqdef [k\, \mod
  \, m_s]\in \{0,\ldots, m_s-1\}$; thus $[-1]_s=m_s-1$ and
$[m_s]_s=0$. For any $i\in \{0,\ldots, m_s-1\}$, let $I_s^i$ denote
that connected component (an open segment) of the set $C_s\setminus
V_s$ which starts at $v_s^i$ and ends at $v_s^{[i+1]_s}$ with respect
to the orientation of $C_s$. For every $i\in \{0,\ldots, m_s-1\}$,
pick a special Dolbeault factorization $(S_s^i,\fD_s^i)$ of $W$ and an
element $\omega_s^i\in
\Hom_{\DF_\sm(X,W)}((S_s^{[i-1]_s},\fD_s^{[i-1]_s}),(S_s^{i},\fD_s^{i}))=\cA(X,Hom(S_s^{[i-1]_s},S_s^{i}))$.
Now pick admissible Hermitian metrics $h_s^i$ on the complex vector
superbundles $S_s^i$. Then the boundary coupling of the B-type
topological Landau-Ginzburg model defined on $\Sigma$ is constructed
as in Subsection 4.2. of \cite{LG2}, except that the superconnection
factor $\cU_s$ given in \cite[eq. (4.16)]{LG2} is replaced by:
\ben
\label{cU}
\cU_s\! \eqdef \! \str \left[\left(\! P e^{-\int_{I_s^{m_s-1}}\dd \tau M_s^{m_s-1}(\tau)}\right)\omega_s^{m_s-1}\!\!
 \ldots \left(P e^{-\int_{I_s^1}\dd\tau M_s^1(\tau)}\right)\omega_s^1 \left(P e^{-\int_{I_s^0}\dd\tau M_s^0(\tau)}\right)\omega_s^0\right]~,
\een
where $P$ is the path ordering symbol and $M_s^i(\tau)$ is given by
equation \cite[eq. (4.17)]{LG2} in terms of the special Dolbeault
factorization $(S_s^i,\fD_s^i)$ and of the metrics $h_s^i$ (the dagger
in that formula becomes the Hermitian conjugate taken with respect to
$h_s^i$). Of course, the bundle-valued forms $\omega_s^i$ must be
interpreted super-geometrically using the first identification given in
\cite[eq. (3.22)]{LG2}. Notice that the quantities denoted in
loc. cit. by $F$ and $G$ are denoted in this paper by $v$ and $u$
(cf. \cite[eq. (4.13)]{LG2} and equation \eqref{Dmatrix}). An example
is drawn in Figure \ref{fig:corners}.

\begin{figure}[H]
\begin{center}
\scalebox{0.6}{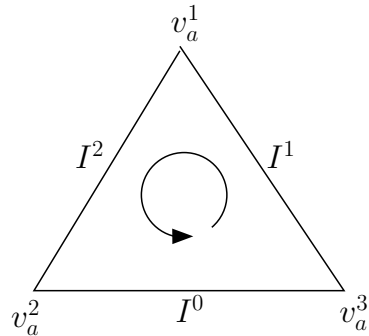}
\caption{Example of a compact surface with three corners and connected
  frontier. In this example, we have $n=1$ and we do not indicate the
  index $s$.}
\label{fig:corners}
\end{center}
\end{figure}

\noindent Using \eqref{cU}, it is not hard to see that the path integral
arguments of \cite{LG2} extend to the case of compact oriented
surfaces with corners, thereby leading to the conclusion that the
total cohomology category $\HDF_\sm(X,W)$ provides a mathematical
model of the category of topological D-branes of the B-type
topological Landau-Ginzburg theory defined by the LG pair $(X,W)$
(where we assume that $X$ is Calabi-Yau). Proposition
\ref{prop:DFDFinf} allows us to identify the latter with the total
cohomology category $\HDF(X,W)$ of the Dolbeault category of
holomorphic factorizations. This identification justifies the
statement made in Section \ref{sec:DF} that $\HDF(X,W)$ provides a
mathematical model for the category of topological D-branes which
agrees with the construction of \cite{LG2}.

\begin{acknowledgements}
This work was supported by the research grant IBS-R003-S1, ``Constructive string
field theory of open-closed topological B-type strings''. 
\end{acknowledgements}

\end{document}

%% file: corners.pdf_t
\begin{picture}(0,0)%
\includegraphics{corners.pdf}%
\end{picture}%
\setlength{\unitlength}{4144sp}%
\begingroup\makeatletter\ifx\SetFigFont\undefined%
\gdef\SetFigFont#1#2#3#4#5{%
  \reset@font\fontsize{#1}{#2pt}%
  \fontfamily{#3}\fontseries{#4}\fontshape{#5}%
  \selectfont}%
\fi\endgroup%
\begin{picture}(3322,3268)(3811,-3581)
\put(6346,-1906){\makebox(0,0)[lb]{\smash{{\SetFigFont{20}{24.0}{\rmdefault}{\bfdefault}{\updefault}{\color[rgb]{0,0,0}$I^1$}%
}}}}
\put(4456,-1906){\makebox(0,0)[lb]{\smash{{\SetFigFont{20}{24.0}{\rmdefault}{\bfdefault}{\updefault}{\color[rgb]{0,0,0}$I^2$}%
}}}}
\put(5446,-3436){\makebox(0,0)[lb]{\smash{{\SetFigFont{20}{24.0}{\rmdefault}{\bfdefault}{\updefault}{\color[rgb]{0,0,0}$I^0$}%
}}}}
\put(7066,-3436){\makebox(0,0)[lb]{\smash{{\SetFigFont{20}{24.0}{\rmdefault}{\bfdefault}{\updefault}{\color[rgb]{0,0,0}$v_a^3$}%
}}}}
\put(5401,-556){\makebox(0,0)[lb]{\smash{{\SetFigFont{20}{24.0}{\rmdefault}{\bfdefault}{\updefault}{\color[rgb]{0,0,0}$v_a^1$}%
}}}}
\put(3826,-3481){\makebox(0,0)[lb]{\smash{{\SetFigFont{20}{24.0}{\rmdefault}{\bfdefault}{\updefault}{\color[rgb]{0,0,0}$v_a^2$}%
}}}}
\end{picture}%